\documentclass[11pt]{amsart}
\usepackage{amsmath,amsthm,amscd,amssymb,tikz,enumitem}
 \pdfoutput=1
\usepackage[margin=.75 in]{geometry}
\usepackage{graphicx}
\usepackage[bookmarks=false]{hyperref}
\usepackage[labelfont=bf]{caption}
\usepackage{subcaption}
\allowdisplaybreaks
\usepackage{color}
\usepackage{cancel}

\numberwithin{equation}{section}

\newcommand{\mb}{\mathbf}
\newcommand{\bs}{\boldsymbol}

\newcommand{\R}{\mathbb{R}}

\newcommand{\x}{\mathbf{x}}

\newcommand{\ba}{\begin{aligned}}
\newcommand{\ea}{\end{aligned}}
\newcommand{\bt}{\begin{thm}}
\newcommand{\et}{\end{thm}}
\newcommand{\bc}{\begin{corollary}}
\newcommand{\ec}{\end{corollary}}
\newcommand{\bl}{\begin{lemma}}
\newcommand{\el}{\end{lemma}}
\newcommand{\bpf}{\begin{proof}}
\newcommand{\epf}{\end{proof}}
\newcommand{\bpb}{\begin{problem}}
\newcommand{\epb}{\end{problem}}
\newcommand{\bd}{\begin{definition}}
\newcommand{\ed}{\end{definition}}
\newcommand{\bn}{\begin{note}}
\newcommand{\en}{\end{note}}
\newcommand{\bq}{\begin{question}}
\newcommand{\eq}{\end{question}}
\newcommand{\bp}{\begin{proposition}}
\newcommand{\ep}{\end{proposition}}
\newcommand{\be}{\begin{example}}
\newcommand{\ee}{\end{example}}
\newcommand{\bex}{\begin{exercise}}
\newcommand{\eex}{\end{exercise}}
\newcommand{\ben}{\begin{enumerate}}
\newcommand{\een}{\end{enumerate}}
\newcommand{\nn}{\nonumber}

\newcommand{\T}{\mathbb{T}}

\theoremstyle{plain}
\newtheorem{theorem}{Theorem}
\numberwithin{theorem}{section}
\newtheorem{maintheorem}[theorem]{Theorem}
\newtheorem{lemma}[theorem]{Lemma}
\newtheorem{corollary}[theorem]{Corollary}
\newtheorem{proposition}[theorem]{Proposition}
\theoremstyle{definition}

\newtheorem{remark}[theorem]{Remark}

\newtheorem{example}[theorem]{Example}
\newtheorem{definition}[theorem]{Definition}

\newtheorem{problem}[theorem]{Problem}
\newcommand{\eref}[1]{(\ref{e.#1})}
\newcommand{\tref}[1]{Theorem \ref{t.#1}}
\newcommand{\lref}[1]{Lemma \ref{l.#1}}
\newcommand{\pref}[1]{Proposition \ref{p.#1}}
\newcommand{\cref}[1]{Corollary \ref{c.#1}}
\newcommand{\fref}[1]{Figure \ref{f.#1}}
\newcommand{\sref}[1]{Section \ref{s.#1}}
\newcommand{\aref}[1]{Assumption \ref{a.#1}}

\newcommand{\dref}[1]{Definition \ref{d.#1}}

\newcommand{\grad}{\nabla}

\author{William M Feldman}
\address{University of Utah, Salt Lake City, USA}
\email{feldman@math.utah.edu}
\author{Inwon C Kim}
\thanks{W. Feldman was partially supported by the NSF grant DMS-2009286}
\address{UCLA, Los Angeles, USA}
\email{ikim@math.ucla.edu}
\thanks{I. Kim was partially supported by the NSF grant DMS-2153254}
\author{Aaron Zeff Palmer}
\address{UCLA, Los Angeles, USA}
\email{azp@math.ucla.edu}
\thanks{A. Palmer was partially supported by the Air Force grant FA9550-18-1-0502}
\title{The sharp interface limit of an Ising Game}

\begin{document}

\begin{abstract}
  The \emph{Ising model} of statistical physics has served as a keystone example of phase transitions, thermodynamic limits, scaling laws, and many other phenomena and mathematical methods.  We introduce and explore an \emph{Ising game}, a variant of the Ising model that features competing agents influencing the behavior of the spins. With long-range interactions, we consider a mean-field limit resulting in a nonlocal potential game at the mesoscopic scale. This game exhibits a phase transition and multiple constant Nash-equilibria in the supercritical regime.  
  
  Our analysis focuses on a sharp interface limit for which potential minimizing solutions to the Ising game concentrate on two of the constant Nash-equilibria.   We show that the mesoscopic problem can be recast as a mixed local/nonlocal space-time Allen-Cahn type minimization problem. We prove, using a $\Gamma$-convergence argument, that the limiting interface minimizes a space-time anisotropic perimeter type energy functional. This macroscopic scale problem could also be viewed as a problem of optimal control of interface motion.      Sharp interface limits of Allen-Cahn type functionals have been well studied.  We build on that literature with new techniques to handle a mixture of local derivative terms and nonlocal interactions.  The boundary conditions imposed by the game theoretic considerations also appear as novel terms and require special treatment.
\end{abstract}

\maketitle

\section{Introduction}

This article develops an \emph{Ising game} as a prototypical example for \emph{spin games} and the phenomenon of phase transitions in mean-field games.   Game theoretic models incorporate rational agent behavior, introducing additional complexity to the particle models of statistical physics.  These models are suited to applications including social dynamics, economics, and neural networks.

We begin with first introducing our framework. To put our work into context, we survey a series of results in the literature, including the passage from the discrete \emph{spin games} to continuous mean-field games.  Our main result, stated in \sref{mainresult-state}, focuses on a mesoscopic to macroscopic scaling limit for the \emph{Ising game}.

 \subsection{Motivation}

In economics, the study of games has been used to form insights into phenomena that arise when the players exhibit free will and decision-making in their choice of actions.  Beyond the original applications to economics and finance \cite{von1947theory}, game theoretic models  have been used in evolutionary biology \cite{Doebeli} and opinion dynamics \cite{ghaderi2013opinion}.   Along with games, one can consider distributed optimization problems -- where many individual agents take actions with a collective objective -- for example arising from the management of a smart energy grid \cite{xu2015robust} or training weights of a neural network \cite{pmlr-v97-tan19a}.

Phase transitions have been proposed to be important phenomena in understanding biological systems \cite{cavagna2010scale}, \cite{mora2011biological}, neural dynamics \cite{hopfield1982neural}, \cite{chialvo2010emergent}, and social behavior \cite{vanni2011criticality}.  Many frameworks exist to model such systems.  In this work, we consider an intersection between the frameworks of dynamic games and spin systems that allows for both concrete calculations and general mathematical analysis.

There are many interesting aspects that arise in phase transitions, including mesoscopic and macroscopic scaling limits and interface dynamics \cite{de1994glauber}, fluctuations in the mesoscopic limit and universality classes \cite{Masi_1996}, and spontaneous phase separation \cite{de1996glauber}.  All of these aspects have been studied at length for many different particle models.  In this work, we first briefly review the mesoscopic limit, which has been studied extensively in the context of mean-field games. We then focus our technical analysis on the macroscopic limit and interface dynamics, where we find novel features that require new techniques.

\subsection{Spin games}

Spin systems arise in the analysis of the magnetization of solid-state materials where the spin represents the magnetic moment of a particle.  Another common application of spin systems is that of the grand canonical ensemble of particles interacting as a fluid.  In these models, the spin is interpreted as a discretization of the particle density. In this way, spin systems can be used generally as a discretization of models with continuous state variables.    Spin systems have been considered in connection with the mean-field behavior of populations in \cite{horst2010dynamic}, \cite{horst2006equilibria}, and \cite{collet2016rhythmic}.  An Ising game with discrete player actions was studied in \cite{leonidov2021ising} played on graphs, where a dynamic evolution was considered that behaves similarly to the Ising model in the mean-field limit.

We combine the concept of spin systems with a multiplayer game, where each agent controls their own spin in an optimal manner.  The Ising game is a prototypical example of such a \emph{spin game} that mixes a discrete state variable with a continuous state position.  A fundamental distinction with models of statistical physics (as well as the evolution considered in \cite{leonidov2021ising}) is that players will look ahead, and their prediction of the future influences their control decisions. The spin game models provide an ideal environment to study the phenomenon of phase transitions with the addition of rational behavior.  

The combination of discrete and continuous state variables has yet to be considered in the literature on mean-field games, which extensively covers games with either discrete or continuous state variables.  A general treatment of finite-state mean-field games is given in \cite{carmona2018probabilistic} and \cite{gomes2013continuous}. Phase transitions were observed in \cite{kolokoltsov2016mean} as a bifurcation of the ergodic mean-field game system.  The solution to the master equation is analyzed in \cite{bayraktar2018analysis} and \cite{cecchin2019convergence}.  Phase transitions in continuous space mean-field games have been studied using bifurcation analysis in \cite{grover2018mean}.  Further analysis of the fluctuations about equilibria was undertaken in \cite{cecchin2019convergence}.

We find that the Ising game with nonlocal interactions undergoes a phase transition as the strength of player interactions changes. When the interaction strength is small, the players do not deviate from a `rest' behavior that results in independent spins with zero mean.  Above a critical interaction strength, the players will instead exert their control to align more closely with their neighbors, resulting in a nonzero mean.  

 The phase transition corresponds to a bifurcation of solutions to the mean-field game system. 
When we include the nonlocal interactions of continuous spatial variables, we find new dynamics that we can best understand by considering a macroscopic limit. 

\subsection{Macroscopic Limit}
 The literature on the Ising model and other macroscopic limits of phase transition models is vast.  For a treatment of the analogous results in the Ising model, see \cite{de1994glauber}, \cite{bodineau1999wulff}, \cite{alberti1996surface}.  Novel mathematical tools were developed in \cite{evans1992phase}, \cite{katsoulakis1995generalized}.  Our approach is more akin to the work on the \emph{Van der Waals - Allen - Cahn - Hilliard} model of gradient phase transitions \cite{modica1987gradient}.  Additional tools for similar models are developed in \cite{bouchitte1990singular}, \cite{alberti1998non}, \cite{conti2002gamma}, \cite{sandier2004gamma}.

 We establish a surprising equivalence between the mesoscopic spin field optimal control problem and a mixture of the local and nonlocal \emph{Van der Waals - Allen - Cahn - Hilliard} models.  The resulting macroscale interface minimizes a cost functional, which takes the form of an anisotropic space-time area.  Alternatively, the resulting interface evolves in time with controlled propagation speed. Closely related macroscopic models of distributed optimal control are considered in \cite{bressan2021optimalmoving}, \cite{bressan2021optimal}.

 While we use many tools from the related phase transition models, combining them in this new way introduces many novel aspects of the analysis.  The boundary conditions imposed by the game theoretic considerations also appear as novel terms and require special treatment.

 \subsection{Main result, outline, and open questions} \label{s.mainresult-state}

Our contributions consist of both the introduction and the analysis of spin game models. More precisely, we start with introducing and motivating a class of spin game models at the microscopic, mesoscopic, and macroscopic scales.  Next, we focus on the mesoscopic to macroscopic scaling limit.  We reduce the study of equilibria in the \emph{Ising game} to critical points of an energy functional that combines a kinetic energy term, a double-well potential, and a nonlocal energy in the spatial directions.  We introduce an analysis of the ``effective surface tension" and initial and terminal time boundary layer costs associated with this model.  Finally, we use new analytical techniques, within the established context of the sharp interface limit theory of phase field models, to handle the mix of local and nonlocal interaction costs and the additional boundary layer terms.
 
After expressing a more general framework, our analysis focuses on a more specific \emph{Ising game}, which consists of the following elements:
\begin{itemize}
    \item A spin field on the space-time domain,  $s:[0,T]\times \mathbb{T}^d\rightarrow [-1,1]$, that represents the mean spin.  The spin field is determined by selecting control policies, $a_\pm :[0,T]\times \mathbb{T}^d \rightarrow \R^+$, that represent the rates of flipping from $+1$ to $-1$ and from $-1$ to $+1$. The evolution equation for the spin field is given by
    \begin{equation}\label{coupling}
        \lambda\, \partial_\tau s(\tau,z) = a_-(\tau,z)\, \frac{1-s(\tau,z)}{2} - a_+(\tau,z)\, \frac{1+s(\tau,z)}{2},
    \end{equation}
    where the small parameter $\lambda>0$ corresponds to a mesoscopic scale.
    \item A Lagrangian function, $L: [0,1]\times(\mathbb{R^+})^2\rightarrow \R$, of the local mean spin and controls, which is a local running cost density associated with a player controlling the rate at which their spin flips.  We work with the form
    $$
        L(s,a_{\pm}) = \beta^{-1}\, a_+\big(\log(a_+)-1\big)\frac{1+s}{2}+\beta^{-1}\, a_-\big(\log(a_-)-1\big)\frac{1-s}{2},
    $$
    that closely resembles an entropic term in the Ising model. The parameter $\beta^{-1}>0$ has an interpretation as a cost coefficient and appears analogous to the  temperature. The convex Lagrangian $L$ enforces that the flipping rates are positive and encourages $a_\pm$ to coincide at a neutral value. Consequently, $L$ encourages the mean spin $s$ to rest at zero.  The derivation of this form of Lagrangian from a microscopic model is covered in Section 2 for further motivation.
    
    \item An interaction running potential cost density of the form
    $$
        -\frac{1}{2}s(\tau,z)\big(J^\lambda ( s*\tau,\cdot)\big)(z),
    $$
    where $J^\lambda(z) = \lambda^{-d}J(\lambda^{-1}\, z)$ is a nonnegative, rescaled interaction kernel that encourages players to align their spins with their neighbors at a length scale of $\lambda$.  The strength of the interaction is given by $\hat{J} =\int_{\R^d} J(x)dx$. As explained in Section \ref{sec:games}, minimizing the total potential cost ($\mathcal{C}^\lambda$ below) corresponds to Nash equilibrium strategies.  
    The competition of the Lagrangian and the interaction energy results in phase transition  where, when $\beta\, \hat{J} > 1$, players prefer to organize themselves at a constant Nash equilibrium with mean spin $\mathfrak{s}>0$ or $-\mathfrak{s}$. We denote the corresponding running cost density of the constant Nash equilibria by $\bs\Lambda$.
            
    \item An initial spin configuration $s_0:\T^d\rightarrow[-1,1]$, and a terminal cost of the form
    $$
        \int_{\mathbb{T}^d} g(z)\, s(T,z)dz.
    $$
    These initial and terminal conditions cause solutions to deviate from the constant Nash equilibria.    
\end{itemize}

 Our specific problem is now to minimize and to study the sharp interface limit $\lambda \to 0$ of the averaged rescaled cost (in the macroscopic $(\tau,z)$ coordinates), under the constraint \eqref{coupling} and the initial data $s(\cdot,0)=s_0$,
\begin{align*}
\mathcal{C}^\lambda\big(s,a_\pm\big)=&\ \lambda^{-1}\, \int_0^{ T}\int_{\mathbb{T}^d}\Big[L\big(s(\tau,z),a_{\pm}(\tau,z)\big)-\frac{1}{2}\, s(\tau,z)\, (J^\lambda*s(\tau,\cdot))(z) - \bs\Lambda\Big]dz\, d\tau \\
      &\ + \int_{\T^d}g(z)\,  s(T, z) dz.\nonumber
\end{align*}

As $\lambda\rightarrow 0$, the mean spin concentrates on the set $\{-\mathfrak{s},\mathfrak{s}\}$, except on an interface $\Sigma$ and the boundary layers at $\tau=0$ and $\tau=T$.  The contribution of the asymptotic behavior at the mesoscopic $\lambda$-scale of $\mathcal{C}^\lambda$ allows us to characterize the interface between the equilibrium states as a local minimizer of the macroscopic energy, as we will see in the context of the $\Gamma$-convergence.

 Our analysis begins with the illuminating observation that the cost can be decomposed, up to a total derivative, into the sum of a double-well potential, a Dirichlet-like strictly convex function of $\partial_\tau s$, and a nonlocal interaction cost.  The integrand of the space-time integral of $\mathcal{C}^\lambda$ becomes
  $$
       \frac{1}{\lambda}\mathcal{W}_\beta(s(\tau,z)) + \frac{1}{2\beta} \partial_\tau s(\tau,z) \Phi'(s(\tau,z))+ \frac{1}{2\beta\lambda}\Psi(s(\tau,z),\lambda\, \partial_\tau s(\tau,z)) +\frac{1}{4\lambda}\int_{\mathbb{T}^d} J^\lambda(z-w) |s(\tau,z)-s(\tau,w)|^2dw    
  $$
  (see Corollary~\ref{decomposition}). 
  Using this decomposition, we show that the rescaled spin variable converges to one of the stable equilibrium states $-\mathfrak{s}$ or $\mathfrak{s}$, with a transition interface $\Sigma$ in between the states. Moreover, we show that the cost in the macroscopic limit, in $\Gamma$-convergence sense, is given by the sum of initial and terminal time  boundary layers, and the space-time integral of an anisotropic interfacial energy with ``effective surface tension" $\bar{L}$
 $$
\int_{\Sigma} \bar{L}\big(\nu(\tau,z)\big) d\mathcal{H}^d
 $$
where $\mathcal{H}^d$ is the $d$-dimensional Hausdorff measure on the space-time interface $\Sigma = \partial_* \{s = \pm\mathfrak{s}\}$ between the $\{-\mathfrak{s},\mathfrak{s}\}$ and $\nu$ is the space-time normal on $\Sigma$ pointing toward the $\mathfrak{s}$-region.   

Our main result, which is stated in full below in Theorem~\ref{thm:main}, is summarized by:

 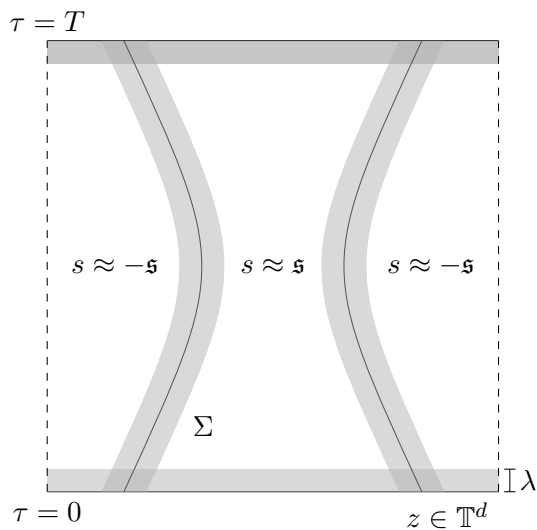
\begin{figure}
\begin{tikzpicture}[scale = 3]
\def\l{.1}
\draw (-1,-1) node[anchor = north] {$\tau=0$}--(1,-1) node[anchor = north east] {$z\in \T^d$};
\path[fill=gray!60,semitransparent] (-1,-1)--(1,-1)--(1,-1+\l)--(-1,-1+\l);
\draw (-1,1) node[anchor = south] {$\tau=T$}--(1,1);
\path[fill=gray!60,semitransparent] (-1,1)--(1,1)--(1,1-\l)--(-1,1-\l);
\path[fill=gray!60,semitransparent] (-1,1)--(1,1)--(1,1-\l)--(-1,1-\l);

\draw[dashed] (-1,-1) -- (-1,1);
\draw[dashed] (1,-1) -- (1,1);

\draw (-.66,-1) .. controls (-.2,0) .. (-.66,1);
\path[fill=gray!60,semitransparent] (-.66-\l,-1) .. controls (-.2-\l,0) .. (-.66-\l,1) -- (-.66+\l,1) .. controls (-.2+\l,0) .. (-.66+\l,-1);
\draw (-.5+\l,-.8) node[anchor = south west] {$\Sigma$};
\draw (.66,-1) .. controls (.2,0) .. (.66,1);
\path[fill=gray!60,semitransparent] (.66-\l,-1) .. controls (.2-\l,0) .. (.66-\l,1) -- (.66+\l,1) .. controls (.2+\l,0) .. (.66+\l,-1);

 \draw (1+.5*\l,-1)-- (1+.5*\l,-1+.5*\l) node[anchor=west]{$\lambda$} -- (1+.5*\l,-1+\l);
 \draw (1+.25*\l,-1)--(1+.75*\l,-1);
 \draw (1+.25*\l,-1+\l)--(1+.75*\l,-1+\l);
  \draw (0,0) node {$s \approx \mathfrak{s}$};
 \draw (.7,0) node {$s\approx -\mathfrak{s}$};
  \draw (-.7,0) node {$s\approx -\mathfrak{s}$};
\end{tikzpicture}
\caption{\label{f.lambda-scale-fig}Schematic diagram showing a cross section of $s$ solution in $d \geq 2$ (in $d=1$ such catenoidal type solution would not occur). Boundary layers at scale $\lambda$ are displayed around the phase interface and initial and final times.}
\end{figure}

 \begin{maintheorem}
 The mesoscopic scale cost functionals $\mathcal{C}^{\lambda}$ converge as $\lambda \to 0$, in the sense of  $\Gamma$-convergence on appropriate spaces, to the macroscopic scale cost
\begin{align*}
    \bar{V}(\bar{s}) =  \int_{\mathbb{T}^d}V^{init}\big({s}_0(z),\bar{s}(0,z)\big)dz+\int_{\mathbb{T}^d}V^{end}\big(\bar{s}(T,z),g(z)\big)dz +\int_{\Sigma} \bar{L}\big(\nu(\tau,z)\big) d\mathcal{H}^d
\end{align*}
among BV functions $\bar{s}:\T^d \to \{\pm \mathfrak{s}\}$ with $\Sigma=\Sigma(\bar{s})$ as the discontinuity set, and $\nu = \nu (\bar{s})$ as the measure-theoretic normal vector field on $\Sigma$, pointing toward the $\mathfrak{s}$-region.  

In particular, sequences of minimizers for $\mathcal{C}^{\lambda}$ are precompact in $L^1$ and any subsequential limit as $\lambda\to 0$ minimizes $\bar{V}$.
\end{maintheorem}

See \fref{lambda-scale-fig} for an illustration of a macroscopic limit $\bar{s}$.

 The initial and terminal time costs can be represented by a one-dimensional problem, showing that the solutions are locally constant in space near the boundary layer.  A remarkable relation appears between the initial and end times that 
 \begin{equation}\label{e.vendformula}
    V^{end}(\bar{s},g) = \inf_{s_0}\Big\{ V^{init}({s}_0,\bar{s}) + g\, s_0 + \frac{1}{\beta}\Phi(s_0)\Big\},
\end{equation}
(see Remark~\ref{layer_symmetry} below).  The $\Phi$ term contains all of the asymmetry between the initial and final time as the remaining terms in the decomposition of Corollary~\ref{decomposition} obey a time-reversal invariance as $s(\tau,z)$ goes to $s(-\tau,z)$. This relation also shows that without an end cost $g$ it is advantageous to relax back closer to $0$ from the equilibria $\pm\mathfrak{s}$ near the terminal time.  This feature is intimately connected with the forward-looking nature of control problems: the agents anticipate the end time and turn off their controls to save costs.

 The key element of the analysis is various quantitative versions of {\it patching estimates}, which we use to localize the profile of the spin variable near the transition interface. The other essential ingredient is the elegant idea introduced in \cite{alberti1998non} to study a nonlocal interaction energy, where the patching lemmas are applied to polyhedral regions to construct a recovery sequence for the $\Gamma$ convergence of the cost. While this approach does not readily yield quantitative error estimates, for our purpose it provides a relatively simple alternative of perturbing smooth surfaces to compare it with hyperplanes.

 There are interesting open questions that arise from our analysis. For instance, we can question the shape of the minimizer for $\bar{V}$ as well as the regularity or geometry of the interface.  We suspect that our limit Lagrangian $\bar{L}$ is at least continuous with respect to $\nu$ when the interaction kernel $J$ is isotropic, but it is not easy to check this due to the anisotropy created by the time variable. Even with a regular $\bar{L}$, the shape of the minimizing interface is not necessarily regular in higher dimensions, as we see from \cite{simons1968minimal}, \cite{morgan1991cone}.  It is also natural to ask whether $\bar{L}$ can be obtained by only considering planar traveling wave solutions. This is true in the case of the nonlocal interaction energy studied in \cite{alberti1998non}, see \cite{alberti98annalen}. We answer this question positively for the boundary layer costs, but it remains open for the interfacial cost. 
 
 Another natural question is on the asymptotic behavior of phase transition near $s=0$ when the ``inverse temperature" $\beta$ approaches the critical value where the local parts of the cost dominate and the double-well structure disappears via $\pm \mathfrak{s}$ converging to zero. 
 
 The boundary layer terms appearing in the macroscopic cost is a novel feature in our problem that merits further study: for instance, there is an apparent symmetry between the initial and terminal cost (see Remark \ref{layer_symmetry} and \eref{vendformula}).

While we specialize our model with a specific Lagrangian, for which calculation is convenient, we expect that our analysis extends directly to a more general class. For example, we can consider the Lagrangian functions that have the form
$$
L(s,a_\pm) = l(a_+)\frac{1+s}{2} + l(a_-)\frac{1-s}{2},
$$
where $l(a)$ is convex and satisfies $l'(a)\rightarrow -\infty$ as $a\rightarrow^+ 0 $ and $l'(a)\rightarrow +\infty$ as $a\rightarrow +\infty $.  

A more interesting and difficult open question is whether one can obtain similar results for nearest-neighbor interaction costs, analogous to what was achieved for the three-dimensional nearest-neighbor Ising model in \cite{bodineau1999wulff}.

    \section{Spin games}\label{sec:games}
    
    In this section, we introduce and provide a non-rigorous exposition on spin games: $N$-player games, mean-field control, and mean-field games. In the length scale spectrum considered in this work, the $N$-player game is a microscopic model, and the mean-field control and game problems are mesoscopic models.  The derivations discussed in this section are meant as motivation and contextualization for the rigorous mathematical work which we conduct later in the paper which considers a mesoscopic to macroscopic limit.
    
    \subsection{N-player spin games}

        We consider $N$ players with fixed positions on a uniform square lattice, $x^{N,i}\in \mathbb{T}^d$.  The collection of all positions is denoted as $\mb{x}^N\in \mathbb{T}^{dN}$.  Each player has a discrete spin state $\sigma^{N,i} \in \mathbb{S}=\{-1,1\}$, and the player controls the rate at which their spin flips according to the control $A^{N,i}_t\in \R^+$.  We denote the collection of all spins as $\bs\sigma^N\in \mathbb{S}^N$ and of all controls as $\mb{A}^N_t\in (\R^+)^N$. When determining their optimal strategy, each player may consider the states of all other players, which we encode into the empirical spin measure $m_{\bs\sigma^N,\mb{x}^N} \in \mathcal{M}(\mathbb{T}^d)$, the space of finite variation signed measures on $\T^d$,
        $$
            m_{\bs\sigma^N,\mb{x}^N} = \frac{1}{N}\sum_{i=1}^N \sigma^{N,i}\, \delta_{x^{N,i}}.
        $$
        Denote $\mathcal{P}(\mathbb{S}^N)$ the space of probability measures on spin configurations. We will consider the evolution of state distributions $\mu^N_t \in \mathcal{P}(\mathbb{S}^N)$.
        
        To ease the notation we drop $N$ when it can be inferred from the context.

         The problem consists of specifying the following:
        \begin{itemize}
            \item A Lagrangian function on the control space, $l: \mathbb{R^+}\rightarrow \R$, which is the cost associated with a player flipping their spin.
            \item An individual running player cost on the state and empirical measure space, $\tilde{f}:\mathbb{S}\times \mathbb{T}^d \times \mathcal{M}(\mathbb{T}^d)\rightarrow \R$.  We also consider the case of a global running cost that is a function only of the empirical measure $\bar{f}:\mathcal{M}(\mathbb{T}^d)\rightarrow \R$.
            \item A terminal cost on the state and empirical measure space, $\tilde{g}:\mathbb{S}\times \mathbb{T}^d \times \mathcal{M}( \mathbb{T}^d)\rightarrow \R$, and the analogous case of global terminal cost $\bar{g}:\mathcal{M}( \mathbb{T}^d)\rightarrow \R$.
            \item An initial distribution of states $\mu_0^N\in \mathcal{P}(\mathbb{S}^N)$. E.g., $\sigma^{i}$ are independent with mean $s_0(x^i)$ for $s_0:\T^d\rightarrow [-1,1]$.
        \end{itemize}

         An important aspect of game theoretic problems is the information available to the players.  We work here assuming full information, i.e., closed loop, where each player may choose their control as a function of the state of all the other players
         $$
            (t,\bs\sigma) \mapsto A^{i}_t(\bs\sigma).
         $$

    Given $\mb{A}$, we define the joint distribution $\mu_t^{\mb{A}}\in \mathcal{P}(\mathbb{S}^N)$ as the joint distribution of all players with spin $i$ flipping at rate $A^i_t(\bs\sigma)$, that is $\mu_t^{\mb{A}}$ is the solution of
    \begin{align}\label{eqn:N-evolution}
        \frac{d}{dt} \mu_t^{\mb{A}}(\bs\sigma) = \frac{1}{2}\sum_{i=1}^N\Big(A^i_t(\mathfrak{t}^i\bs\sigma)\mu^N_t(\mathfrak{t}^i\bs\sigma)-A^i_t(\bs\sigma){\mu}^N_t(\bs\sigma)\Big),
    \end{align}
    where $\mathfrak{t}^i\bs\sigma$ denotes the collections of spins with the $i$th component flipped to be $-\sigma^i$.
    
    {\it Global control problem.}
    We define the global cost to be
    \begin{align*}
        \mathfrak{C}^N(\mb{A}) =&\ \int_0^T\sum_{\bs\sigma\in \mathbb{S}^N} \Big[\frac{1}{N}\sum_{i=1}^Nl\big(A_t^{i}(\bs\sigma)\big)+\bar{f}\big(m_{\bs\sigma,\mb{x}}\big)\Big]\mu^{\mb{A}}_t(\bs\sigma)dt\\
        &\ +\sum_{\bs\sigma\in \mathbb{S}^N} \bar{g}(m_{\bs\sigma,\mb{x}})\mu^{\mb{A}}_T(\bs\sigma)
    \end{align*}
    and the control problem is
    \[\inf_{\mb{A}:[0,t]\times \mathbb{S}^d\rightarrow (\R^+)^N} \mathfrak{C}^N\big(\mb{A}\big).\]
    This has the form of either a standard optimal control problem with states in $\mathcal{P}(\mathbb{S}^N)$ or as a continuous time Markov decision process with discrete states in $\mathbb{S}^N$.
    We let $V^N:[0,T]\times \mathbb{S}^N\rightarrow \R$ be the value function that solves
    $$
        V^N_T(\bs\sigma)=-\bar{g}\big(m_{\bs\sigma,\mb{x}}\big),
    $$
    and
    $$
        \frac{\partial}{\partial t}V^N_t(\bs\sigma) + \frac{1}{N}\sum_{i=1}^N h\big(\frac{N}{2}\, \partial^iV^N_t(\bs\sigma)\big)=\bar{f}(m_{\bs\sigma,\mb{x}}),
    $$
    where $h$ is the Legendre transform of $a\mapsto l(a)$ and the discrete finite gradient is
    $$
        \partial^i V^N_t(\bs\sigma)=V^N_t(\mathfrak{t}^i\bs\sigma)- V^N_t(\bs\sigma).
    $$  
    By standard theory, the optimal control is then given by
    $$
        A^{i}_t(\bs\sigma) = h'\big(\frac{N}{2}\, \partial^i V^N_t(\bs\sigma)\big).
    $$
    
    {\it $N$-player game.}
    We define the individual costs to be
    \begin{align*}
        \mathfrak{c}^{N,i}(\mb{A}) =&\ \int_0^T\sum_{\bs\sigma\in \mathbb{S}^N}\Big[l\big(A_t^{i}(\bs\sigma)\big)+\tilde{f}\big(\sigma^i,x^i,m_{\bs\sigma,\mb{x}}\big)\Big]\mu^{\mb{A}}_t(\bs\sigma)dt\\
        &\ +\sum_{\bs\sigma\in \mathbb{S}^N}\tilde{g}(\sigma^i,x^i,m_{\bs\sigma,\mb{x}})\mu^{\mb{A}}_T(\bs\sigma).
    \end{align*}
    We consider the differential game played by the $N$ players. 
    A Nash equilibrium is collection of controls $\mb{A}$ such that for each $i$ we have
    $$
        \mathfrak{c}^{N,i}(\mb{A}) \leq \inf_{B}\mathfrak{c}^{N,i}\big((\ldots, A^{i-1},B,A^{i+1},\ldots)\big).
    $$

    We look for coupled solutions $v^{N,i}_t$ with
    $$
        v^{N,i}_T(\bs\sigma)=-\tilde{g}(\sigma^i,x^{i},m_{\bs\sigma,\mb{x}}),
    $$
    and
    \begin{align*}
        0=&\ \frac{\partial}{\partial t}v^{N,i}_t(\bs\sigma) + \frac{1}{2}\sum_{\substack{ 1\leq j \leq N\\j\not=i}} A^{j}_t(\bs\sigma) \partial^j v_t^{N,i}(\bs\sigma)\\
        &\ + h\big(\frac{1}{2}\, \partial^i v_t^{N,i}(\bs\sigma)\big)-\tilde{f}\big(\sigma^i,x^{i},m_{\bs\sigma,\mb{x}}\big),
    \end{align*}
    with
    $$
        A^{j}_t(\bs\sigma) = h'\big(\frac{1}{2}\, \partial^j v_t^{N,j}(\bs\sigma)\big).
    $$

\subsection{Mean-field spin games}\label{sec:mean_feld_game}

Since the dependence of each player's costs on the other players is only in terms of the empirical spin measure, one expects the system to limit to a mean-field game as $N \to \infty$. Specifically, the \emph{random} empirical spin measures, $m_{\bs\sigma,\mb{x}}$, concentrate on a flow of \emph{deterministic} spin fields $s(t,x)$ corresponding to the mean of $\sigma^i$ for $x^i$ near $x$. We follow this concept in order to, non-rigorously, derive the corresponding mean-field game system in the infinite-player limit.

For the mean-field version, we consider control policies $a_\pm(t,x)$.  We work in terms of the spin field $s(t,x)$, which represents the average state of the players near $x$ at time $t$.  We note that the density of players in state $\pm1$ can be recovered as $\frac{1\pm s(t,x)}{2}$.    The evolution of the spin field is given by
  \begin{align}\label{eqn:evolution_mf}
    \frac{\partial}{\partial t} s(t,x) =&\ a_-(t,x)\frac{1-s(t,x)}{2}-a_+(t,x)\frac{1+s(t,x)}{2}\\
    s(0,x)=&\ s_0(x).\nonumber
  \end{align}

We assume that $\tilde{f}(\sigma,x,m)=-\tilde{f}(-\sigma,x,m)$, and when $s(x)dx=m(dx)$ we write $f(x,s) = \tilde{f}(+1,x,m)$.  When $m_{\bs\sigma,\mb{x}}$ concentrates at $s(t,x)$ under the probability measure $\mu_t(\bs\sigma)$, we have
$$
    \sum_{\bs\sigma\in \mathbb{S}^N}\frac{1}{N}\sum_{i=1}^N\tilde{f}(\sigma^i,x^i,m_{\bs\sigma,\mb{x}})\mu_t(\bs\sigma) \approx  \sum_{\sigma \in \mathbb{S}}\int_{\mathbb{T}^d}\sigma f\big(x,s(t,\cdot)\big)\frac{1+\sigma\, s(t,x)}{2}dx =  \int_{\mathbb{T}^d} f\big(x,s(t,\cdot)\big)s(t,x)dx,
$$
and we do the same for $\tilde{g}$ and $g$.  We also abuse notation slightly, to write $\bar{f}(s)=\bar{f}(m)$ when $s(x)dx=m(dx)$. The state space of spin fields $s(t,\cdot)$ is denoted by $X$ which is the unit ball in $L^\infty(\T^d)$.

The global cost is given by
    \begin{align*}
        \mathcal{C}(s,a) =&\ \int_0^T\Big[ \int_{\mathbb{T}^{d}}L\big(s(t,x),a_\pm(t,x)\big)dx+\bar{f}\big(s(t,\cdot)\big)\Big] dt +\bar{g}\big(s(T,\cdot)\big),
    \end{align*}
    where
    $$
        L(s,a_\pm) = \sum_{\sigma\in \mathbb{S}} l(\sigma,x,a)\frac{1 + \sigma\, s}{2}.
    $$
    
   {\it Mean-field global control.} The global optimal control problem is
    \[\inf_{s,a_\pm} \Big\{\mathcal{C}(s,a); (s,a) \hbox{ solves } (\ref{eqn:evolution_mf})\Big\}.\]

    The value function is defined
    \begin{align}\label{eqn:dynamic programming}
        V_{t_0}(s_0)=\sup_{s,a_\pm} \Big\{&-\int_{t_0}^T \Big[\int_{\mathbb{T}^{d}}L\big(s(t,x),a_\pm(t,x)\big)dx+\bar{f}\big(s(t,\cdot)\big)\Big] dt -\bar{g}\big(s(T,\cdot)\big);\\
        &\ \frac{\partial}{\partial t} s(t,x) = a_-(t,x)\frac{1-s(t,x)}{2}-a_+(t,x)\frac{1+s(t,x)}{2} \hbox{ on } (t,T]\times \mathbb{T}^d, \nonumber\\
        &\ s(t_0,\cdot)=s_0\Big\}.\nn
    \end{align}
   The McKean-Vlasov equation (\ref{eqn:evolution_mf}) can be expressed as, for all $\phi \in C^1\big([0,T]\times X;\R\big)$,
    \begin{align*}
    &\ \phi_T\big(s(T,\cdot)\big)-\phi_0(s(0,\cdot)\big) \\
     =&\   \int_0^T \Big[ \frac{\partial}{\partial t}\phi_t(s(t,\cdot)\big) +\sum_{\sigma\in \mathbb{S}}\int_{\mathbb{T}^d}D\phi_t(s(t,\cdot)\big)(x)(-\sigma)a_\sigma(t,x)\frac{1+\sigma\, s(t,x)}{2}dx\Big] dt,
    \end{align*}
    where $D\phi(r)$ is a Frech\'{e}t derivative.
    
   Let the Hamiltonian
    $$
        H(s,p) = \sum_{\sigma\in \mathbb{S}} h(- \sigma\, p) \frac{1+\sigma\, s}{2}.
    $$
Formally following the normal derivation using the McKean-Vlasov equation one arrives at the Hamilton-Jacobi-Bellman equation on the spin-field space
    \begin{align}\label{eqn:HJB_Wasserstein}
        \frac{\partial}{\partial t} V_t(s)+\int_{\T^d} H\big( s(x),D V_t(s)({x})\big)dx = \bar{f}(s),
    \end{align}
    with 
    $$
        V_T(s)=-\bar{g}(s).
    $$
    
    {\it Mean-field spin game.} For the mean-field game, we fix a flow of spin-fields $r(t,\cdot) \in X$, and we consider the cost
    \begin{align*}
       \mathcal{C}(s,a_\pm;r) =&\ \int_0^T \int_{\mathbb{T}^{d}}\Big[L\big(s(t,x),a_\pm(t,x)\big)+f\big(x,r(t,\cdot)\big)s(t,x)\Big]dx\, dt\\
        &\  +\int_{\mathbb{T}^{d}}g\big(x,r(T,\cdot)\big)s(T,x)dx.
    \end{align*}
    We are looking for Nash equilibria, i.e. $(s,a_\pm)$ that satisfy (\ref{eqn:evolution_mf}) such that
    $$
        \mathcal{C}(s,a_\pm;s)\leq \mathcal{C}(s',a_\pm';s)
    $$
    for all $(s',a'_\pm)$ that satisfy (\ref{eqn:evolution_mf}).  We can equivalently consider the set-valued map 
    $$
        \Phi(r)=\Big\{s:\ (s,a_\pm)\hbox{ satisfy  (\ref{eqn:evolution_mf}) and   minimize } \mathcal{C}(s,a_\pm;r)\Big\},
    $$
         and we want to find a fixed point $s\in \Phi(s)$.

    In the game case, the value function $v_t(x,s)$ solves
    \begin{align}\label{eqn:master}
        0=&\ \frac{\partial}{\partial t} v_t(x,s)+\sum_{\sigma \in \mathbb{S}}\int_{\T^d} A_t\big(\sigma,y,s\big)(-\sigma)\frac{1+\sigma\, s(y)}{2} Dv_t(x,s)(y)dy \\
        &\ +\sum_{\sigma\in \mathbb{S}}\frac{\sigma}{2}h(-\sigma\, v_t(x,s)\big)- f(x,s),\nonumber
    \end{align}
    with
    $$
        v_t(x,s)=- g(x,s),
    $$
    and
    $$
        A_t\big(\gamma,y,s\big)=h'\big(-\gamma\, v_t(y,s)\big).
    $$

    \subsection{Survey of results}
     We now list a few standard results, which are common in either finite-state mean-field games or continuous mean-field games \cite{carmona2018probabilistic}, \cite{gomes2013continuous}.  The proofs can all be adapted to spin games.
     
    \emph{Potential games}. 
    A potential game occurs when the costs, $f(x,s)$ and $g(x,s)$ are derived from potential costs as $f(x,s)=D\bar{f}(s)(x)$ and $g(x,s)= D\bar{g}(s)(x)$.  In this case, the Nash equilibria for the mean-field game correspond to critical points of the global control problem.

    \begin{proposition}\label{prop:potential}
        We suppose that $f(x,s)=D\bar{f}(s)(x)$ and $g(x,s)= D\bar{g}(s)(x)$.  If $V$ is a solution to (\ref{eqn:HJB_Wasserstein}) on $[0,T]\times X$ then $v_t(x,s)=DV_t(s)(x)$ solves (\ref{eqn:master}).
    \end{proposition}
    \begin{proof}
        Equation (\ref{eqn:master}) is obtained by differentiating (\ref{eqn:HJB_Wasserstein}) with respect to the field argument, $s$.  In particular, we have
        \begin{align*}
            &\ D\int_{\T^d} H\big( s(y),D V_t(s)({y})\big)dy(x)\\
            =&\ \sum_{\sigma\in \mathbb{S}}\frac{\sigma}{2}\, h\big(-\sigma\, v(x,s)\big) + \int_{\mathbb{T}^d}\sum_{\sigma\in \mathbb{S}}(-\sigma)\partial_p h\big(-\sigma\, DV_t(s)(y)\big)\frac{1+\sigma\, s(y)}{2}D^2V_t(s)(y,x)dy\\
            =&\ \sum_{\sigma\in \mathbb{S}}\frac{\sigma}{2}\, h\big(-\sigma\, DV_t(s)(x)\big) + \int_{\mathbb{T}^d}\sum_{\sigma\in \mathbb{S}}(-\sigma)A_t(\sigma,y,s)\frac{1+\sigma\, s(y)}{2}Dv(x,s)(y)dy.
        \end{align*}
    \end{proof}

    \emph{Mean-field Nash System}. 
    In either the game or global case  we have the mean-field Nash system, which is
    \begin{align}\label{eqn:mfg_system}
        \frac{\partial}{\partial t} s(t,x) =&\ a_-(t,x)\frac{1-s(t,x)}{2}-a_+(t,x)\frac{1+s(t,x)}{2}\\
        -\frac{\partial}{\partial t} p(t,x)=&\ \sum_{\sigma\in \mathbb{S}}\frac{\sigma}{2}h\big(-\sigma\, p(t,x)\big)- f\big(x,s(t,\cdot)\big)\nn\\
        a_\sigma(t,x)=&\ h'\big(- \sigma\, p(t,x)\big).\nn
    \end{align}
    With $s(0,x)=s_0(x)$  and $p(T,x)=-g\big(x,s(T,\cdot)\big)$. 

    \emph{Monotonicity}. 
    If $f$ and $g$ are monotone, i.e.,
    $$
        \Big(f(x,s^1)-f(x,s^2)\Big)\Big(s^1(x)-s^2(x)\Big) \geq 0, {\rm \ and,\ } \Big(g(x,s^1)-g(x,s^2)\Big)\Big(s^1(x)-s^2(x)\Big)\geq 0,
    $$ 
    then the solution to (\ref{eqn:mfg_system}) is unique.  We will be interested in the phenomena that arise without this property.

    \begin{proposition}
    Suppose that $(s^1,p^1)$ and $(s^2,p^2)$ are two solutions to (\ref{eqn:mfg_system}).  If $f$ and $g$ are monotone, then $s^1=s^2$ and $p^1=p^2$.
    \end{proposition}
    \begin{proof}
        The proof follows exactly the same idea as the monotonicity argument for continuous games as in, for example, Proposition 3.2 of \cite{cardaliaguet2019master}, by showing that the quantity
        $$
            \sum_{\sigma \in \mathbb{S}}\big(\sigma\, p^1(t,x)-\sigma\, p^2(t,x)\big)\frac{1+\sigma\, s^1(t,x)}{2} - \sum_{\sigma \in \mathbb{S}}\big(\sigma\, p^1(t,x)-\sigma\, p^2(t,x)\big)\frac{1+\sigma\, s^2(t,x)}{2}
        $$
        decreases along the flow, and is nonnegative at the end-time due to the monotonicity of $g$.
    \end{proof}

\emph{Convergence of $N$-player games}. 
Solutions of the $N$-player game / global control problem converge to the solutions of the mean-field game / global control problem when the interactions are in a mean-field form. Without the uniqueness of solutions, it is often necessary to consider a weaker randomized notion of solutions. On the other hand, the solution to the master equation (\ref{eqn:master}) constructs approximate solutions to the $N$-player problem.  Results on the convergence of continuous games can be found in \cite{cardaliaguet2019master} and \cite{lacker2020convergence}.  The convergence problem for finite state games has been analyzed in \cite{cecchin2019convergence} and \cite{gomes2013continuous}.  The master equation has also been used to determine the fluctuations about the mean and a large deviations principle \cite{cecchin2019convergence}, \cite{delarue2019master}.  We expect the results for the framework of \emph{spin games} to follow the identical trends from these works, although we do not pursue them in detail here.

A remarkable aspect of these results is that the resulting mean-field game system (\ref{eqn:mfg_system}) does not depend on the information structure of the players, or even whether the problem originated as a game or as a global distributed optimal control problem.  We focus on this system as the starting point for our macroscopic convergence analysis.

\section{Ising game and Macroscopic limit}

We now specify a problem formulation for which we consider in depth the question of a macroscopic scaling limit.  The problem, modeled after the statistical Ising model, exhibits a phase transition, where in the `ordered' phase two stable stationary equilibria solutions are present. In the macroscopic limit, all equilibria will concentrate on these two solutions except on a codimension-one interface in space-time.  We consider only the case of a potential game, in which case the global optimizers correspond to Nash equilibria. In contrast with the Ising model, the interface is `controlled', to minimize an inhomogeneous space-time surface area, which can also be viewed as a minimization of the speed of propagation along the front. See the discussion at the end of Section 3.4.

We work on the `mesoscopic' domain $[0,\lambda^{-1}\, T]\times \lambda^{-1}\mathbb{T}^d$, where $\lambda^{-1}\mathbb{T}^d$ is the $d$-dimensional torus of width $\lambda^{-1}$ that can be associated with $[0,\lambda^{-1}]^d\subset \R^d$.  We recall that in the discussion in Section \ref{sec:games} we have already passed from a `microscopic' scale, which appears in the mesoscopic scale as a length scale of order $\lambda^{-1}\, N^{-1/d}$  (so we are effectively considering that $N>>\lambda^{-d}>>1$). 

The interaction will be determined by a kernel satisfying the following assumptions:
\begin{enumerate}[label = (A\arabic*)]
\item $J : \R^d \to \R$ is non-negative and has finite total mass
\[ \hat{J}:= \int_{\R^d} J(x)dx.\] \label{a.Jpos}
\item \label{a.Jmoment} Finite first moment
\[ \int_{\R^d} |h|J(h) dh < + \infty.\]
\end{enumerate}

  The interaction acts at a distance of order one on the {\it mesoscopic scale} where we use $(t,x)$, which will appear as a distance of order $\lambda$ on the {\it macroscopic scale} where we use $(\tau,z)=(\lambda\, t, \lambda\, x)$. 
We consider the convolution on a torus, for $\eta\in L^2(\lambda^{-1}\, \mathbb{T}^d)$, $J*\eta\in L^2(\lambda^{-1}\, \mathbb{T}^d)$ as
$$
    (J*\eta)(x) = \int_{ \mathbb{R}^d} J(x-y)\, \eta(y)\, dy,
$$
where we have used the periodic extension of $\eta$ to $\R^d$.

\subsection{Problem statement and macroscopic scaling}
We rescale the cost by subtracting the cost of the stationary equilibria, $\bs\Lambda$, and multiplying by $\lambda^{d}$ to capture the costs on a co-dimension one region. We consider the asymptotics as $\lambda\rightarrow^+ 0$ of the problem to minimize
\begin{align}\label{eqn:cost}
\mathcal{C}^\lambda\big(s,a_\pm\big):=&\ \lambda^d\, \int_0^{\lambda^{-1}\, T}\int_{\lambda^{-1}\mathbb{T}^d}\Big[L\big(s(t,x),a_{\pm}(t,x)\big)+f\big(x,s(t,\cdot)\big) - \bs\Lambda\Big]dx\, dt \\
      &\ + \lambda^{d}\int_{\lambda^{-1}\T^d}g(\lambda\, x)\, s(\lambda^{-1}\, T, x)dx\nonumber
\end{align}
where we define, inspired from the microscopic problem with $l(a) = \beta^{-1}\, a\, (\log(a) - 1)$,
\begin{align}\label{eqn:Lagrangian_cost}
    L(s,a_{\pm}) := \beta^{-1}\, a_+\big(\log(a_+)-1\big)\frac{1+s}{2}+\beta^{-1}\, a_-\big(\log(a_-)-1\big)\frac{1-s}{2},
\end{align}
and
\begin{align}\label{eqn:interaction_cost}
    f(x,s) := -\frac{1}{2}\, s(x)\, (J*s)(x),
\end{align}
subject to the constraint
\begin{align}\label{eqn:evolution}
    \partial_t s(t,x) =&\ a_-(t,x)\, \frac{1-s(t,x)}{2} - a_+(t,x)\, \frac{1+s(t,x)}{2},\\
    s(0,x)=&\ s_0(\lambda\, x).\nonumber
\end{align}

The Lagrangian incentivizes the neutral strategy, $a_\pm = 1$, where the spin switching rate is always 1.  With this strategy, the mean spin would lie at rest at $s=0$.  The parameter $\beta$ has the same effect as the inverse temperature in the Ising model, although the interpretation here is a control penalty and not inherently statistical.  In the same fashion as the Ising model, the interaction incentivizes agreeing with nearby spins when $\hat{J}>0$.

The constant, $\bs\Lambda$ that corresponds to the cost of the stationary equilibria, is given by
$$
    \bs\Lambda := -\frac{\hat{J}}{2}-\frac{1}{2\, \beta^2\, \hat{J}}.
$$

\begin{remark}
We assume that $g$ does not depend on the spin field for simplicity, although our techniques would allow such dependence.  In particular, it is natural to allow $g$ to depend on $s(x)$ locally, which is slightly different from the terminal cost in Section \ref{sec:games}, where it was assumed that $g$ was defined over the empirical measures. This local form meshes naturally with our macroscopic analysis and could be the limiting result of a slightly more complicated microscopic problem.
\end{remark}

 We introduce the costate $p(t,x)$, so that $a_{\pm}$ maximizes the Hamiltonian,
\begin{align*}
    H(s,p) : =&\ \max_{a_\pm} \left\{\big(a_-\, \frac{1-s}{2} - a_+\, \frac{1+s}{2}\big)\, p-L(s,a_{\pm})\right\}\\
    =&\ \beta^{-1}\, \cosh(\beta\, p) - \beta^{-1} \, s\, \sinh\big(\beta\, p),
\end{align*}
where the maximum occurs at 
$$
    a_-=e^{\beta\, p}, \quad a_+=e^{-\beta\, p}.
$$
The optimality equations, equivalent to (\ref{eqn:mfg_system}), are
\begin{align}\label{eqn:Ising_mfg_system}
\partial_t s(t,x) =&\ \sinh\big(\beta\, p(t,x)\big) - s(t,x)\, \cosh\big(\beta\, p(t,x)\big)\\
-\partial_t p(t,x) =&\ -\beta^{-1}\, \sinh\big(\beta\, p(t,x)\big) + \big(J*s(t,\cdot)\big)(x), \nonumber
\end{align}
with
\begin{align*}
    s(0,x) =&\ s_0(\lambda\, x)\\
    -p(\lambda^{-1}\, T,x) = &\ g\big(\lambda\, x).
\end{align*}
Because we work directly with the energy, which we will view as a function of $(s,\partial_t s)$, we mostly will not refer to (\ref{eqn:Ising_mfg_system}) nor the costate $p$.

We note that the problem can also be posed in the macroscopic coordinates of 
$$(\tau,z):=(\lambda\, x,\lambda\, t). $$  In the new variables the cost can be now written as 
\begin{align*}
    \mathcal{C}^\lambda\big(s,a_\pm\big)=\lambda^{-1}\int_0^{ T}\int_{\mathbb{T}^d}\Big[L\big(\hat{s}(\tau,z),\hat{a}_{\pm}(\tau,z)\big)+f^\lambda\big(z,\hat{s}(\tau,\cdot)\big)-\bs\Lambda\Big]dz\, d\tau + \int_{\T^d}g(z)\, \hat{s}( T, z)dz
\end{align*}
with $\hat{s}(\tau,z) = s(\lambda^{-1}\, \tau\, \lambda^{-1}\, z)$ moving with the 'fast' dynamics
\begin{align*}
    \lambda\, \partial_\tau \hat{s}(\tau,z) = \hat{a}_-(\tau,z)\, \frac{1-\hat{s}(\tau,z)}{2} - \hat{a}_+(\tau,z)\, \frac{1+\hat{s}(\tau,z)}{2},
\end{align*}
and with the 'short' range interaction
$$
    f^\lambda(z,\hat{s}) := -\frac{1}{2}\, \hat{s}(z)\, (J^\lambda*\hat{s})(z),
$$
with
$$
    J^\lambda(z) := \lambda^{-d}J(\lambda^{-1}\, z).
$$

The corresponding optimality equations are
\begin{align*}
\lambda\, \partial_\tau \hat{s}(\tau,z) =&\ \sinh\big(\beta\, \hat{p}(\tau,z)\big) - \hat{s}(\tau,z)\, \cosh\big(\beta\, \hat{p}(\tau,z)\big)\\
-\lambda\, \partial_\tau \hat{p}(\tau,z) =&\ -\beta^{-1}\, \sinh\big(\beta\, \hat{p}(\tau,z)\big) + \big(J^\lambda*\hat{s}(\tau,\cdot)\big)(z), 
\end{align*}
with
\begin{align*}
    \hat{s}(0,z) =&\ s_0(z)\\
    -\hat{p}(T,z) = &\ g(z).
\end{align*}
 We remark that, in this form, the system can be easily simulated using forward-backward iteration: see Figure \ref{fig:simulation} for some results from simulations.

\begin{figure}
\centering
     \begin{subfigure}[b]{0.48\textwidth}
         \centering
         \includegraphics[width=\textwidth]{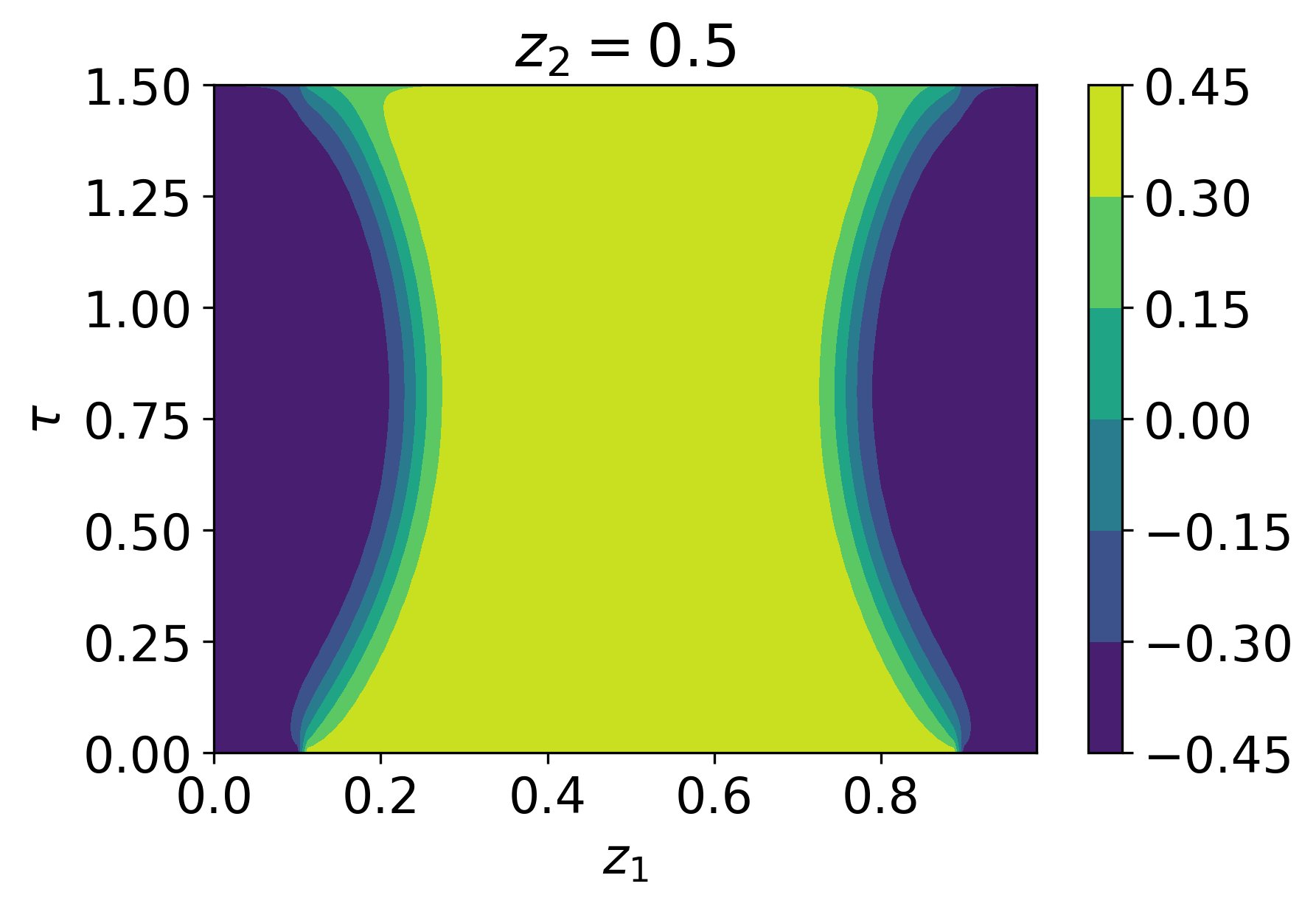}

     \end{subfigure}
     \hfill
     \begin{subfigure}[b]{0.48\textwidth}
         \centering
         \includegraphics[width=\textwidth]{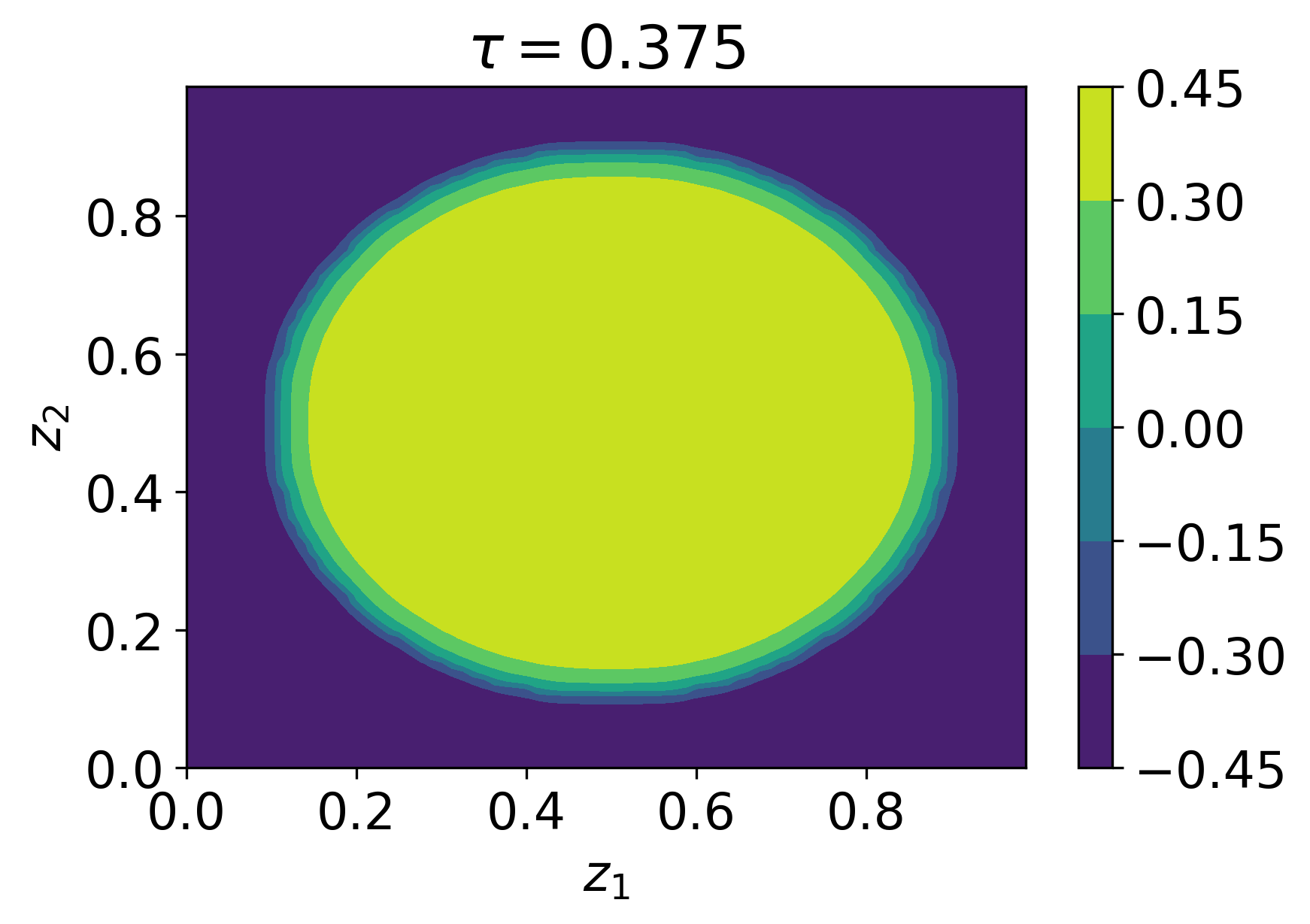}

     \end{subfigure}
    \caption{\label{fig:simulation} A spatial slice and a time slice of a 2D simulation, with $\beta^{-1}=0.9$, $\hat{J}=1$. $T=1.5$, where $J$ is a standard Gaussian and $\lambda = 1/40$. Numerics were implemented using a simple forward-backward iteration of (\ref{eqn:Ising_mfg_system}) with $80^2$ spatial grid points and $300$ time grid points.}
\end{figure}
  
  In the following, we will primarily work in the macroscopic coordinates, and drop the hat from the notation  for $s$, $a$, and $p$.
\subsection{Alternate parameterization and appearance of a double-well potential}\label{sec:alternate}

In order to pass to the macroscopic limit, it is helpful to decompose the cost functional \eqref{eqn:cost} into terms that resemble more closely what has been studied in the literature.  The interaction cost \eqref{eqn:interaction_cost} will be split as a local term and a nonlocal gradient  penalization.  An identical nonlocal term has been studied in \cite{alberti1998non} alongside a double-well potential, and we use this work as a primary guide for our analysis.  The local part of interaction cost combines with the Lagrangian \eqref{eqn:Lagrangian_cost}, and then further decomposes as a double-well potential and a local penalization of the time gradient that is similar to kinetic energy.  The local terms closely relate to the gradient penalizations with double-well potentials that were studied in \cite{modica1987gradient} and many other works, except that we only have the time gradient and no spatial derivatives.  The Ising game can thus be seen as a mixture of the local and nonlocal phase transition models, which is local in the time component and nonlocal in the spatial component.  This mixture introduces many new challenges. However, using the decomposition detailed below, many of the techniques of both the local and nonlocal theory can be adapted for our analysis. 

We expand the interaction cost as
\begin{align*}
    \int_{\mathbb{T}^d}f^\lambda(z,s)\, s(z)\, dz =&\ -\frac{1}{2}\int_{\mathbb{T}^d}\hat{J}\, s(z)^2\, dz + \frac{1}{4}\int_{\mathbb{T}^d}\int_{ \mathbb{T}^d}J^\lambda(w-z)\big(s(w)-s(z)\big)^2dw\, dz.
\end{align*}
The first term may now be combined with the local control cost.  To put this into a more standard form, we express the local terms as a function of the spin field and velocity,
\begin{align}\label{eqn:double_well_velocity}
W(S,V) := \inf_{A_\pm} \Big\{ L(S,A_\pm) - \frac{1}{2}\hat{J}\, S^2 - \bs\Lambda; V = A_-\, \frac{1-S}{2} - A_+\, \frac{1+S}{2}\Big\}.
\end{align}
 At zero velocity, we denote $\mathcal{W}_\beta(S) := W(S,0)$.
When $\beta\, \hat{J}> 1$,  this is a double-well potential
$$
    \mathcal{W}_\beta(S)  = -\frac{1}{\beta}\sqrt{1-S^2} - \frac{1}{2}\hat{J}\, S^2-\bs{\Lambda}
$$
with minimizers at 
$$S=\pm \mathfrak{s}:=\pm\sqrt{1-\beta^{-2}\hat{J}^{-2}}$$
(see \fref{wfig}). These minimizers $\pm\mathfrak{s}$ correspond to the stationary equilibria of the Ising game. Recall that we have normalized by the stationary cost $\bs{\Lambda}$ so that $\mathcal{W}_\beta(\mathfrak{s})=\mathcal{W}_\beta(-\mathfrak{s})=0$. 

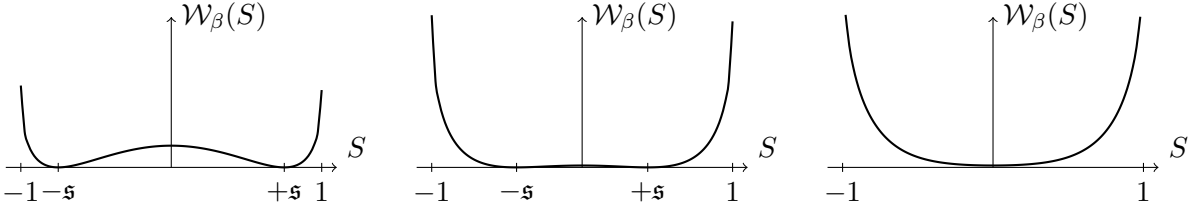
\begin{figure}
   \begin{subfigure}[b]{0.3\textwidth}
\begin{tikzpicture}[scale = 2]
\draw[->] (-1.1,0) -- (1.1,0) node[anchor = south west]{$S$};
\draw[->] (0,0) -- (0,1) node[anchor = west]{$\mathcal{W}_\beta(S)$};
\def\b{.66}
\def\J{1}
\def\L{-.5-.5*\b*\b}
\def\s{{sqrt(1-\b*\b)}}
\draw[black,thick,domain=-1:1, smooth, samples = 80, variable=\x]   plot (\x,{2.5*(-\b*sqrt(1-\x*\x)-.5*\x*\x+.5+.5*\b*\b)}) ;
\draw[] (1,.03)--(1,-.03) node [anchor = north]{$1$};
\draw[] (-1,.03)--(-1,-.03)node [anchor = north]{$-1$};
\draw[] (-\s,.03)--(-\s,-.03)node [anchor = north]{$-\mathfrak{s}$};
\draw[] (\s,.03)--(\s,-.03)node [anchor = north]{$+\mathfrak{s}$};
\end{tikzpicture}
\end{subfigure}
\begin{subfigure}[b]{0.3\textwidth}
\begin{tikzpicture}[scale = 2]
\draw[->] (-1.1,0) -- (1.1,0) node[anchor = south west]{$S$};
\draw[->] (0,0) -- (0,1) node[anchor = west]{$\mathcal{W}_\beta(S)$};
\def\b{.9}
\def\J{1}
\def\L{-.5-.5*\b*\b}
\def\s{{sqrt(1-\b*\b)}}
\draw[black,thick,domain=-1:1, smooth, samples = 80, variable=\x]   plot (\x,{2.5*(-\b*sqrt(1-\x*\x)-.5*\x*\x+.5+.5*\b*\b)}) ;
\draw[] (1,.03)--(1,-.03) node [anchor = north]{$1$};
\draw[] (-1,.03)--(-1,-.03)node [anchor = north]{$-1$};
\draw[] (-\s,.03)--(-\s,-.03)node [anchor = north]{$-\mathfrak{s}$};
\draw[] (\s,.03)--(\s,-.03)node [anchor = north]{$+\mathfrak{s}$};
\end{tikzpicture}
\end{subfigure}
\begin{subfigure}[b]{0.3\textwidth}
\begin{tikzpicture}[scale = 2]
\draw[->] (-1.1,0) -- (1.1,0) node[anchor = south west]{$S$};
\draw[->] (0,0) -- (0,1) node[anchor = west]{$\mathcal{W}_\beta(S)$};
\def\b{1.1}
\def\J{1}
\def\L{-.5-.5*\b*\b}
\def\s{{sqrt(1-\b*\b)}}
\draw[black,thick,domain=-.98:.98, smooth, samples = 80, variable=\x]   plot (\x,{2.5*(-\b*sqrt(1-\x*\x)-.5*\x*\x+.5+.5*\b*\b)}) ;
\draw[] (1,.03)--(1,-.03) node [anchor = north]{$1$};
\draw[] (-1,.03)--(-1,-.03)node [anchor = north]{$-1$};
\end{tikzpicture}
\end{subfigure}
    \caption{Plots of $\mathcal{W}_\beta$ for $\hat{J} = 1$ and for varying values of $\beta^{-1} \in \{.66,.9,1.1\}$ crossing the critical value at $\beta\hat{J} = 1$. 
    \label{f.wfig}}
\end{figure}

The local energy $W$ decomposes into the double-well potential and a convex `Dirichlet'-like energy that is quadratic near $v=0$ and grows superlinearly like $|V|\log|V|$  as $|V|\rightarrow \infty$.  An additional term  $\Phi$ appears, which is a total time derivative and can be integrated out of the cost and incorporated into the boundary conditions. Surprisingly, the total time derivative term encapsulates all of the time asymmetry of the problem.

\begin{proposition}\label{prop:W_decomposition}\label{p.W_decomposition}
The potential $W(S,V)$ decomposes as
\[  W(S,V) = \mathcal{W}_\beta(S) + \frac{1}{2\beta} V\, \Phi'(S)+ \frac{1}{2\beta}\Psi(S,V)\]
where  
\[\Psi(S,V) :=\frac{1}{2 }\int_0^V\frac{(V-Z)}{\sqrt{Z^2+(1-S)(1+S)}} dZ \]
and
\[
\Phi(S) : = (1+S)\log(1+S)+(1-S)\log(1-S).
\]
The decomposition satisfies
\begin{enumerate}
\item $\Psi(S,\cdot)$ is even, strictly convex, $\Psi(S,0) = 0$, monotone increasing on $V \in [0,\infty)$, has the bounds
\begin{equation}\label{estimate}
\Psi(S,V) \sim |V| k\left(\tfrac{V}{\sqrt{(1+S)(1-S)}}\right) \ \hbox{ with } k(r) = \min\{|r|,\log(2+|r|)\}.
\end{equation}
The derivative $\partial_V\Psi(S,V)$ is concave on $V \in [0,\infty)$, zero at $V=0$, hence $V\mapsto \partial_V\Psi(S,V)$ is subadditive, and satisfies the bounds
\begin{equation}\label{estimate2}
\partial_V\Psi(S,V) \sim  k\left(\tfrac{V}{\sqrt{(1+S)(1-S)}}\right) \ \hbox{ for } \ V \geq 0.
\end{equation}
In particular if we define $\Psi_0(V) := \Psi(0,V)$ then $\Psi(S,V) \sim \partial_V\Psi_0(V)$ and $\partial_V\Psi(S,V) \sim \Psi_0(V)$ for $S$ in compact subsets of $(-1,1)$.
\item (double-well coercivity / upper bounds) For $\beta \hat{J}>1$ the potential $\mathcal{W}_\beta(S) \geq 0$ is smooth, symmetric with respect to $0$, and has three critical points in $(-1,1)$ at $\pm\mathfrak{s}$ and $0$ which are, respectively, non-degenerate local minima and a local maximum.  In particular, we have the explicit coercivity with respect to the potential minima
\[ \mathcal{W}_\beta(S) \sim (S \pm \mathfrak{s})^2 \ \hbox{ for $S$ near $\mp \mathfrak{s}$.}\]
\end{enumerate}
\end{proposition}

See \sref{sec3proofs} for the proof.

\begin{remark}
    In view of the decomposition of Proposition \ref{p.W_decomposition} it is natural to consider the function space for the spin field $s$ to be
    $$
        W^{1,1}\big((0,T); L^1(\mathbb{T}^d)\big).
    $$
    This is sufficient to make sense of the initial condition and terminal cost in the sense of $L^1$ trace.  Due to the slightly stronger than $L^1$ growth of the time derivative energy, it is straightforward to obtain the existence of minimizers in this space. Also, since the spin field exists in $(-1,1)$ (which will soon be improved to $(-\mathfrak{s},\mathfrak{s})$) the functions are also $L^\infty$.
    
    Later, in Proposition \ref{prop:compactness}, we find that asymptotically the energy also bounds a gradient in the spatial directions, making the natural space for the macroscopic fields ${s}$ that of bounded variation functions on $(0, T)\times \mathbb{T}^d$.
\end{remark}

 Based on the decomposition,  we now introduce a handful of localized quantities. We first localize the energy by defining, for $A\subset  \mathbb{T}^d$  open, 
\begin{align}\label{eqn:local_energy}
    G^\lambda({s},{v}; A) := \lambda^{-1}\, \int_A \Big[\mathcal{W}_\beta\big({s}(z)\big) + \frac{1}{2\beta}  \Psi\big({s}(z),\lambda\, {v}(z)\big) + \frac{1}{4}\int_{ A}J^\lambda(z-w)\big({s}(z)-{s}(w)\big)^2\, dw\Big]dz
\end{align}
We also denote just the local terms in the energy as
\begin{equation}\label{eqn:very_local_energy}
    G^\lambda_{loc}({s},{v}; A) := \lambda^{-1}\, \int_A \left[\mathcal{W}_\beta\big({s}(z)\big) + \frac{1}{2\beta}  \Psi\big({s}(z),\lambda\, {v}(z)\big)\right] dz.
\end{equation}

When comparing localized energies, we must consider the \emph{locality defect}, as in \cite{alberti1998non}, corresponding to the discrepancy in nonlocal terms. For $A,A' \subset \mathbb{T}^d$,
\begin{align}\label{eqn:locality_defect}
    N^\lambda ({s};A,A') := \frac{1}{4\lambda}\textbf{}\int_A\int_{A'}J^\lambda(z-w)\big({s}(z)-{s}(w)\big)^2\, dz\, dw.
\end{align}

When defining the macroscopic costs, it is useful to consider a cost where the nonlocal term is integrated over all of $\R^d$ (where if $s$ is defined on $\T^d$ it can be extended periodically).  That is
$$
F^\lambda({s},{v};A) := \lambda^{-1}\, \int_A \Big[\mathcal{W}_\beta\big({s}(z)\big) + \frac{1}{2\beta}  \Psi\big({s}(z),\lambda\, {v}(z)\big) + \frac{1}{4}\int_{ \R^d}J^\lambda(z-w)\big({s}(z)-{s}(w)\big)^2\, dw\Big]dz.
$$
Clearly, we have
\begin{equation}\label{FGN}
    F^\lambda({s},{v};A) =  G^\lambda({s},{v};A) + N^\lambda({s}; A, \R^d\backslash A).
\end{equation}

We also consider time-integrated versions of the above quantities.  We will take the convention of naming the time-integrated energies with calligraphic font.  If $A\subset \R^{d+1}$ we let $A_\tau$ denote the time slices of $A$ and $\tau_1$ and $\tau_2$ be the lower and upper bounds in time.  Then

\begin{align}\label{eqn:space_time_energies}
    \mathcal{G}^\lambda({s}; A) &:= \int_{\tau_1}^{\tau_2} G^\lambda\big({s}(\tau,\cdot),\lambda\, \partial_\tau {s}(\tau,\cdot); A_\tau\big) d\tau\\
    \mathcal{G}^\lambda_{loc}({s}; A) &:= \int_{\tau_1}^{\tau_2} G^\lambda_{loc}\big({s}(\tau,\cdot),\lambda\, \partial_\tau {s}(\tau,\cdot); A_\tau\big) d\tau\nonumber\\
    \mathcal{N}^\lambda({s}; A,A') &:= \int_{\tau_1}^{\tau_2} N^\lambda\big({s}(t,\cdot); A_\tau,A_\tau'\big) d\tau\nonumber\\
    \mathcal{F}^\lambda({s}; A) &:= \int_{\tau_1}^{\tau_2} F^\lambda\big({s}(\tau,\cdot),\lambda\, \partial_\tau {s}(\tau,\cdot); A_\tau\big) d\tau.\nonumber
\end{align}

In terms of these definitions, we can reformulate the cost only in terms of the spin field $s$. 
\begin{corollary}\label{decomposition}
Assuming that the controls $a_\pm$ are optimal given $\partial_t s(t,x)$, we can write
\begin{equation}\label{cost:decomposition}
    \mathcal{C}^\lambda\big({s},{a}_\pm)=\mathcal{G}^\lambda(\hat{s};(0,T)\times \T^d) + \int_{\mathbb{T}^d} \Big[g(z)\,\hat{s}(T,z) + \frac{1}{2\, \beta} \Phi\big(\hat{s}(T,z)\big) - \frac{1}{2\, \beta} \Phi\big(\hat{s}(0,z)\big) \Big] dz.
\end{equation}
\end{corollary}

\subsection{Asymptotic heuristics and preliminary results}
We assume that $\beta \hat{J} > 1$, for which there are two stable long time equilibria, $\mathfrak{s}= \sqrt{1-\beta^{-2}J^{-2}}$ and $-\mathfrak{s}$, with cost $\bs\Lambda = -\frac{\hat{J}}{2}-\frac{1}{2\, \beta^2\, \hat{J}}$, as shown in Proposition~\ref{prop:W_decomposition}. Corresponding to each equilibrium there are unique controls given by the constrained minimization procedure of Proposition \ref{prop:W_decomposition} at zero velocity, $A_\pm(\mathfrak{s},0)$ and $A_\pm(-\mathfrak{s},0)$.

The stable equilibria correspond to the leading asymptotic term of the cost that is canceled by the $\bs\Lambda$ in the definition (\ref{eqn:cost}) of $\mathcal{C}^\lambda$.

These results are summarized by the following proposition.
\begin{proposition}\label{prop:const_optimality}
   Assume that $J(x)\geq 0$ for all $x\in \R^d$ and $\beta\, \hat{J}>1$. Then the constant solutions $(\mathfrak{s},A_\pm(\mathfrak{s},0))$ and $(-\mathfrak{s},A_\pm(-\mathfrak{s},0))$ are globally optimal in the sense that if ${s}(0,x)={s}(\lambda^{-1} T,x)=\mathfrak{s}$ for all $x\in \lambda^{-1}\mathbb{T}^d$ , then
    $$
        \mathcal{C}^\lambda\big({s},{a}_\pm) \geq \mathcal{C}^\lambda\big(\mathfrak{s},A_\pm(\mathfrak{s},0)\big),
    $$
    and the same holds with $(\mathfrak{s},A_\pm(\mathfrak{s},0))$ replaced by $(-\mathfrak{s},A_\pm(-\mathfrak{s},0))$.
    Equivalently,
    $$
        \mathcal{G}^\lambda\big(\hat{s};(0,T)\times \T^d\big) \geq 0.
    $$
\end{proposition}

See Section 3.5 for the proof, which follows directly from Proposition \ref{p.W_decomposition}.

The spin fields may be restricted to take values in the interval $[-\mathfrak{s},\mathfrak{s}]$.  We assume that the initial and final data $s_0$ and $g$ respect this condition as well. This assumption is probably not truly required, since the solution should be approximately in the interval $[-\mathfrak{s},\mathfrak{s}]$ outside of some initial/final layers.

\begin{lemma}\label{lem:bdd}
Assume that $J(x)\geq 0$ for all $x\in \R^d$ and $\beta\, \hat{J}>1$.
In addition, suppose that $s(0,z)\in [-\mathfrak{s},\mathfrak{s}]$ for all $z\in  \mathbb{T}^d$ and that $|g(z)|\leq \frac{1}{2\beta}\Phi'(\mathfrak{s})$. Then the cut-off function 
$$
    {s}^{b}(\tau, z) := \max\big\{\min\big\{s(\tau,z),\mathfrak{s}\big\}, -\mathfrak{s}\big\} \quad \hbox{ in } [0,T]\times \mathbb{T}^d
$$
satisfies
\begin{align*}
 &\ \mathcal{G}^\lambda\big(s^b;(0,T)\times \T^d\big) +\int_{\T^d} \Big[s^b(T,z)\, g(z) + \frac{1}{2\beta} \Phi\big(s^b(T,z)\big)\Big]dz\\ \leq&\  \mathcal{G}^\lambda\big(s;(0,T)\times \T^d\big)+ \int_{\T^d}\Big[s(T,z)\, g(z)+\frac{1}{2\beta}\Phi\big(s(T,z)\big)\Big]dz.
\end{align*}
\end{lemma}

We refer again to Section 3.5 for the proof of the above lemma.

 \subsection{Macroscopic Energy}
Let  $\bar{s}$ denote a macroscopic field defined in $(0,T)\times \T^d$, which takes values in the equilibria $\{- \mathfrak{s},\mathfrak{s}\}$ almost everywhere with a discontinuity along some $d-1$ dimensional interface.
In the next section,  we will prove that for minimizers $(s^\lambda,a^\lambda_\pm)$
\begin{align}\label{eqn:asymptotics}
    \lim_{\lambda\rightarrow^+0} \mathcal{C}^\lambda\big(s^\lambda,a_\pm^\lambda\big) =  \inf_{\bar{s}\in BV((0,T)\times \T^d)}\Big\{ \bar{V}(s_0,g,\bar{s})\Big\},
\end{align}
where $\bar{V}$ is the effective cost which we will make precise shortly. This result follows from a more general result in the framework of $\Gamma$-convergence, and helps to characterize asymptotically the minimizers $(s^\lambda,a^\lambda_\pm)$ in the sense that, passing to a subsequence,
$$
\lim_{\lambda\rightarrow 0^+} \hat{s}^\lambda = \bar{s}\in {\rm argmin}_{\bar{s}\in BV((0,T)\times \T^d)}\Big\{ \bar{V}(s_0,g,\bar{s})\Big\}.
$$

  The time scale (in the macroscopic scale) of convergence to the equilibrium is $O(\lambda)$, and the boundary layer terms at the initial and final times correspond to solutions to `infinite time horizon' problems where the spatial interactions are small.   There is also a boundary layer coming from the deviation from the long-time equilibria, $\{-\mathfrak{s},\mathfrak{s}\}$, in a distance $O(\lambda)$ from an interface of $d-1$ dimensions, corresponding to solutions of a `traveling wave'-type cell problem. See \fref{lambda-scale-fig}.

\medskip

The energy $\bar{V}$ will be interfacial, i.e. $\bar{V}(s_0,g,\bar{s}) = +\infty$ unless $\bar{s}(\tau,z)\in \{\pm \mathfrak{s}\}$ for almost every $(\tau,z)$. We denote by $BV((0,T)\times \mathbb{T}^d; \{\pm \mathfrak{s}\})$ the set of bounded variation functions that take values in $\{\pm \mathfrak{s}\}$.  The initial and final values $\bar{s}(0,\cdot)$ and $\bar{s}(T,\cdot)$ can be understood in the sense of $BV$ trace, and both take values in $\{\pm \mathfrak{s}\}$ (see for example Theorem 5.6 of \cite{evans2018measure} and consider that the traces at time $\delta_i>0$ converge in $L^1$ as $\delta_i\rightarrow^+0$ implying that the limit trace will take values in  $\{\pm \mathfrak{s}\}$). Let $\partial_*\{s = \mathfrak{s}\} \cap \{0 < \tau < T\}$ denote the essential boundary of the positive phase region in $(0,T)$, i.e. the phase interface. On the interface, let $\nu(\tau,z)$ denote the measure theoretic unit normal pointing from where $\bar{s}=-\mathfrak{s}$ to where $\bar{s}=\mathfrak{s}$, i.e. $\nu = D\bar{s}/|D\bar{s}|$ the Radon-Nikodym derivative. 

The macroscopic cost $\bar{V}$ is defined as
\begin{align}\label{eqn:bar_V}
    \bar{V}(s_0,g,\bar{s}) :=  \int_{\mathbb{T}^d}V^{init}\big({s}_0(z),\bar{s}(0,z)\big)dz+\int_{\mathbb{T}^d}V^{end}\big(\bar{s}(T,z),g(z)\big)dz +\int_{\Sigma} \bar{L}\big(\nu(\tau,z)\big) d\mathcal{H}^d.
\end{align}
The cost term $V^{init}$ incorporates the initial condition $s_0$.  The initial and terminal boundary layer reduces to an optimization on the mixed scale that is microscopic in time and macroscopic in space.

\begin{figure}
    \centering
    \begin{subfigure}[b]{0.4\textwidth}
         \centering
    \includegraphics[width=\textwidth]{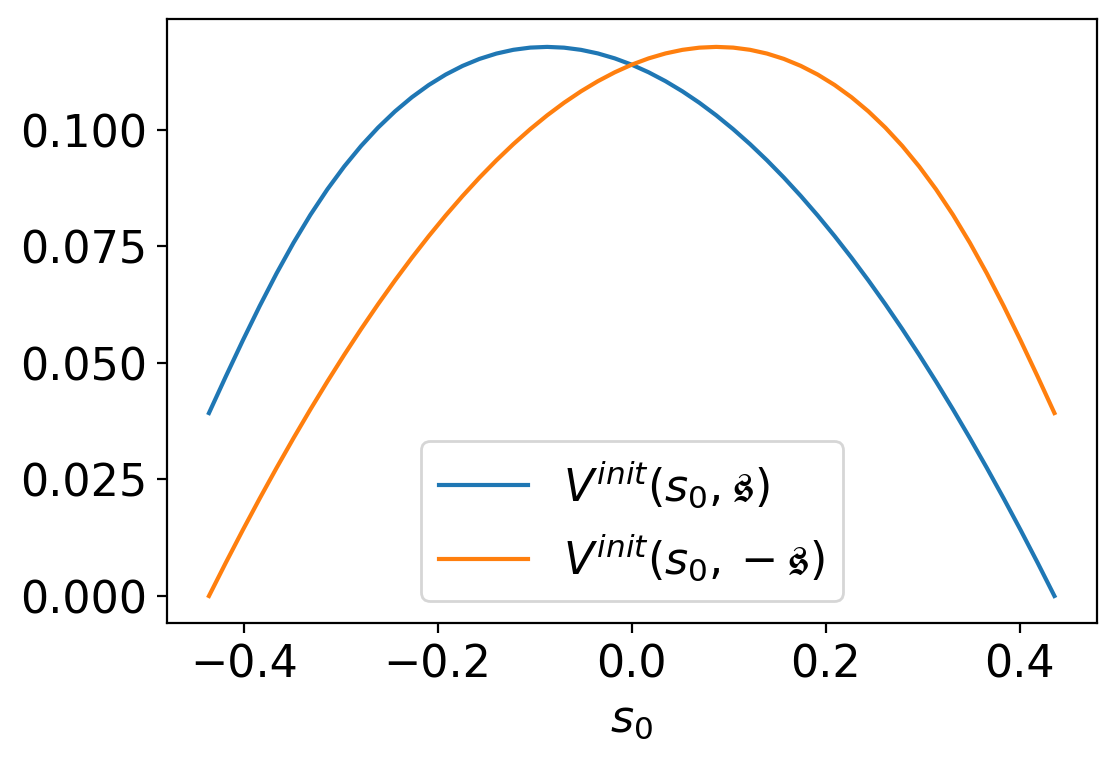}
    \end{subfigure}
    \hfill
     \begin{subfigure}[b]{0.4\textwidth}
         \centering
         \includegraphics[width=\textwidth]{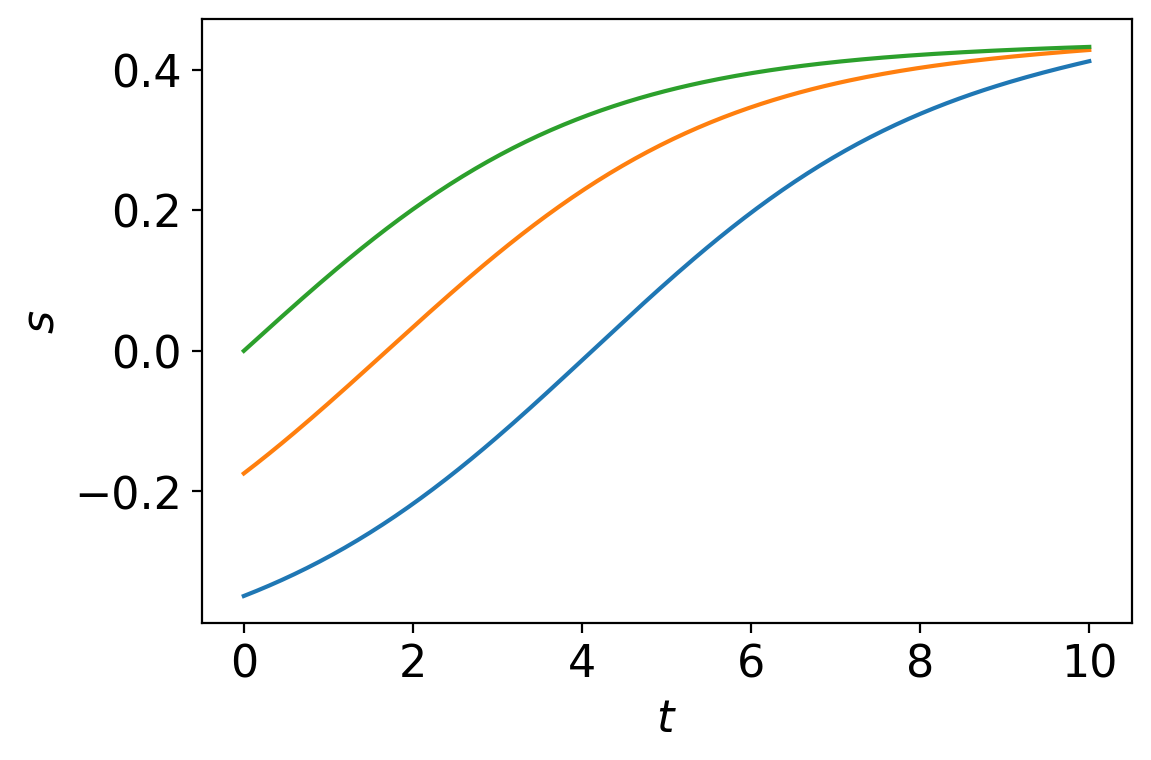}
     \end{subfigure}
    \caption{A plot of $V^{init}$ on the interval $[-\mathfrak{s},\mathfrak{s}]$, with $\beta=0.9^{-1}$ and $\hat{J}=1$ on the left.  On the right, three of one-dimensional solutions for the mean spin field.}
\end{figure}

We now proceed to define formally the macroscopic energy in the interior. In this section, we will further characterize the initial and end costs and characterize heuristically the interfacial cost. 

\medskip

Following the general proof of \cite{alberti1998non}, we consider the localized unscaled energy functional, $\mathcal{F}^1$, defined in (\ref{eqn:space_time_energies}).  We define the rescaling $R_{(\tau,z),r}$ of $s$ to be
    \begin{equation}\label{rescaled_s}
        R_{(\tau,z),r}s (t,x) := s(\tau+r\, t,z+r\, x).
    \end{equation}
Recall that we can decompose $\mathcal{F}^\lambda$ to 
$$
    \mathcal{F}^\lambda(s;A) = \int_{\tau_1}^{\tau_2}\Big(G^\lambda\big(s(\tau,\cdot),\lambda\, \partial_\tau s(\tau,\cdot);A_\tau\big) + N^\lambda \big({s}(\tau,\cdot);A_\tau,\R^d\backslash A_\tau\big)\Big)d\tau,
$$
where $G^\lambda$ and $N^\lambda$ are defined in \eqref{eqn:local_energy} and \eqref{eqn:locality_defect}.
Both $F^\lambda$ and $G^\lambda$ satisfy the following scaling identity:
\begin{lemma}\label{lem:scaling_identity}
For every set $A\in \R^{d}$, we set $z + r\, A = \{z + r\, x:x\in A\}$, and we have
\begin{align}\label{eqn:scaling_local_energy_G}
G^\lambda\big(s(\tau,\cdot),\lambda\, \partial_\tau s(\tau,\cdot);z + r\,  A\big) = r^{d-1}G^{\lambda/r} \big(R_{(0,z),r}s(\tau/r,\cdot),\lambda\, \partial_\tau R_{(0,z),r}s(\tau/r,\cdot);A)
\end{align}
and for every set $B\in \R^{d+1}$, we set $(\tau,z) + r\, B = \{(\tau,z) + r\, (t,x):(t,x)\in B\}$, and we have
\begin{align}\label{eqn:scaling_local_energy_F}
    \mathcal{F}^\lambda\big(s;(\tau,z) + r \, B\big) = r^{d}\mathcal{F}^{\lambda/r} \big(R_{(\tau,z),r}s;B).
\end{align}
\end{lemma}
\begin{proof}
We calculate directly
\begin{align*}
&\ G^\lambda\big(s(\tau,\cdot),\lambda\, \partial_\tau s(\tau,\cdot);z + r\,  A\big)\\
=&\ \int_{r\, A}\Big[ \frac{1}{\lambda}W\big(s(\tau, z +x),\lambda\, \partial_\tau s(\tau,z+x)\big) + \frac{1}{4\, \lambda}\int_{r\, A}J^\lambda(x-y)\big(s(\tau,z+x)-s(\tau,z+y)\big)^2dy\Big]dx\\
=&\ r^{d-1}\int_A\Big[ \frac{r}{\lambda}W\big(s(\tau, z +r\, \tilde{x}),\frac{\lambda}{r}\, r\, \partial_\tau s(\tau,z+r\tilde{x})\big) + \frac{r}{4\, \lambda}\int_{ A}J^\frac{\lambda}{r}(\tilde{x}-\tilde{y})\big(s(\tau,z+r\, \tilde{x})-s(\tau,z+r\, \tilde{y})\big)^2d\tilde{y}\Big]d\tilde{x}\\
=&\  r^{d-1}G^{\lambda/r} \big(R_{(0,z),r}s(\tau/r,\cdot),\lambda\, \partial_\tau R_{(0,z),r}s(\tau/r,\cdot);A).
\end{align*}
Including the time variable,  (\ref{eqn:scaling_local_energy_F}) follows with the additional factor of $r$ from the time integral.
\end{proof}

We now use the unscaled energy $\mathcal{F}^1$ to identify the form of the macroscopic costs by a ``cell problem", namely with test functions as periodic functions on the tangential plane of a normal direction $\nu$.  We follow nearly the same definitions as \cite{alberti1998non} for $\bar{L}(\nu)$, but here we have space-time normal $\nu$ in contrast to the spatial setting in \cite{alberti1998non}. Another new feature is the addition of a width parameter $R\in \R^+$ in the function class, that corresponds essentially to compactly supported variation from function values of $\{-\mathfrak{s},\mathfrak{s}\}$.  This compactness is helpful to restrict our arguments to near the interface, e.g., within a distance $\lambda\, R$ that will become small as $\lambda\rightarrow ^+ 0$; it was not needed in \cite{alberti1998non} due to the simpler nature of patching in their problem.  We then extend these definitions to also apply to the initial and end times, where we impose the additional initial condition as a constraint and the terminal cost in the energy.

\medskip

For a unit-normal vector $\nu$ in $\R^{d+1}$ we define $\blacksquare_\nu$ to be the set of all $d$-dimensional cubes centered at the origin and orthogonal to $\nu$.  For $\square\in \blacksquare_\nu$, we let $\square \times \R\, \nu$ denote the strip $\square \times \R\, \nu := \{(t,x) + \xi\, \nu; (t,x)\in \square, \xi\in \R\}$.  We say that $s:\R^{d+1}\rightarrow \R$ is {\it $\square$-periodic} if $s((t,x) + r\, \omega) = u(t,x)$ when $r\in \R$ is the sidelength of $\square$ and $\omega\in \R^{d+1}$ is a unit-normal vector along an axis of $\square$.  Finally, with $R>0$, we introduce the function class 
\begin{equation}\label{space_X}
\mathcal{X}_R(\square): = \big\{ s: \R^{d+1} \to (-1,1): s \hbox{ is $C^1$, $\square$-periodic, and satisfy \eref{cutoff-cond} }\big\}
\end{equation} 
where
\begin{equation}\label{e.cutoff-cond}
    s(t,x) = \begin{cases} \mathfrak{s} & (t,x)\cdot \nu \geq R \\ - \mathfrak{s} & (t,x)\cdot \nu\leq -R \end{cases}.
\end{equation}

Now we define the interfacial energy with normal $\nu$ with width $R$ to be
\begin{align}\label{eqn:interfacial_energy_R}
    \bar{L}_R(\nu):= \inf\big\{|\square|^{-1} \mathcal{F}^1(s;\square \times \R\, \nu); \square\in \blacksquare_\nu,\ s\in \mathcal{X}_R(\square)\big\}.
\end{align}
The assumption that $s$ is $C^1$ is significant here as discontinuities in the time direction across the cell boundary could result in extra, unaccounted for, energy. We will show later in Lemma \ref{p.periodization}  that the condition $s \in C^1$ can be replaced by a finite energy condition without changing the value of $\bar{L}_R$.

\medskip

With the above definitions, we finally define the interfacial energy to be
\begin{align}\label{eqn:interfacial_energy}
    \bar{L}(\nu) := \liminf_{R\rightarrow +\infty} \bar{L}_R(\nu).
\end{align}
The limit exists since $\bar{L}_R(\nu)$ is monotone decreasing in $R$ and nonnegative from Proposition \ref{prop:const_optimality}.

\medskip

For the initial time we restrict $\nu$ to be oriented in the positive $t$ direction, and for $\square\in \blacksquare_{\nu}$ we consider the half-strip  $\square \times \R^+\, \nu := \{(t,x) + \xi\, \nu; (t,x)\in \square, \xi\in \R^+\}$. For $R>0$ and $s_0\in (-1,1)$, we denote 
\begin{equation}\label{space_init}
\mathcal{X}_R^{init}(s_0,\bar{s}, \square) := \big\{ s: \R^{d+1} \to (-1,1): s \hbox{ is $C^1$, $\square$-periodic, and satisfy \eref{cutoff-cond-init} }\big\}
\end{equation}
where
 \begin{equation}\label{e.cutoff-cond-init}
    s(t,x) =  \bar{s} \hbox{ for }  t \geq R \hbox{ and } s(0,x) = s_0.
\end{equation}
Note that, although the definition above makes sense for any $s_0 \in (-1,1)$, in the paper below we will only actually consider the case $s_0 \in [-\mathfrak{s},\mathfrak{s}]$ which is easier.  

We then define
\begin{align}\label{eqn:initial_energy_R}
V^{init}_R(s_0,\bar{s}) := \inf\big\{|\square|^{-1} \mathcal{F}^1(s;\square\times \R^+\, \nu); \square\in \blacksquare_{\nu},\ s\in \mathcal{X}_R^{init}(s_0,\bar{s}, \square)\big\} - \frac{1}{2\beta}\Phi(s_0).
\end{align}
and
\begin{align}\label{eqn:initial_energy}
V^{init} (s_0,\bar{s}): = \liminf_{R\rightarrow +\infty} V^{init}_R(s_0,\bar{s}).
\end{align}

\medskip

The terminal energy is constructed similarly. We restrict $\nu$ to be oriented in the negative $t$ direction. For $R>0$ and $\square\in \blacksquare_\nu$ we denote 
\begin{equation}\label{space_end}
\mathcal{X}_R^{end}(\bar{s}, \square):= \big\{ s: \R^{d+1} \to (-1,1): s \hbox{ is $C^1$, $\square$-periodic, and satisfy \eref{cutoff-cond-end} }\big\}
\end{equation}
where
 \begin{equation}\label{e.cutoff-cond-end}
    s(t,x) =  \bar{s} \quad\hbox{ for }  t \leq -R.
\end{equation}

We then define, for $g\in \R$ (which will later be restricted to $|g|\leq \frac{1}{2\beta}\Phi(\bar{s})$),
\begin{align}\label{eqn:terminal_energy_R}
V^{end}_R(\bar{s}, g) := \inf\big\{|\square|^{-1} \Big(\mathcal{F}^1(s;\square \times \R^+\, \nu) + \int_\square \Big[g\, s(0,x)+ \frac{1}{2\beta}\Phi\big(s(0,x)\big)\Big]dx\Big) ; \square\in \blacksquare_{\nu},\ s\in \mathcal{X}_R^{end}(\bar{s}, \square)\big\},
\end{align}
and
\begin{align}\label{eqn:terminal_energy}
V^{end}(\bar{s}, g) := \liminf_{R\rightarrow +\infty} V^{end}_R(\bar{s}, g).
\end{align}
Note that the limits in (\ref{eqn:initial_energy}) and  (\ref{eqn:terminal_energy}) exist due to monotonicity, although one can easily see that it is $+\infty$ unless $\bar{s}\in \{-\mathfrak{s},\mathfrak{s}\}$.

\medskip

For the remainder of this subsection, we discuss a further characterization of the macroscopic energy terms $\bar{L}(\nu)$, $V^{init}(s_0,\bar{s})$, and $V^{end}(\bar{s},g)$. Heuristically, when $J$ is radial and monotonically decreasing in the radial direction, the simplest form of a solution is given by the one-dimensional traveling wave, namely  
$$
    s(t,x) = q\big(\nu\cdot (t,x)\big).
$$
We prove that this is indeed the case for $V^{init}$ and $V^{end}$ where the nonlocal term does not participate. It remains as a conjecture for $\bar{L}$.

\begin{theorem}\label{thm:characterization}
The macroscopic initial energy is given by the one-dimensional reduction
\begin{align}\label{eqn:initial_cost}
    V^{init}({s}_0,\bar{s})     =  \liminf_{R\rightarrow \infty}\inf_{\tilde{s}\in C^1([0,R])}\Big\{ \int_0^{R}\Big(\mathcal{W}_\beta\big(\tilde{s}(t)\big) + \frac{1}{2\beta} \Psi\big(\tilde{s}(t),\tilde{s}'(t)\big)\Big) dt; \tilde{s}(0)={s}_0,  \tilde{s}(R)=\bar{s}\Big\} -\frac{1}{2\beta} \Phi(s_0).
\end{align}
Similarly, the macroscopic terminal energy is given by
\begin{align}\label{eqn:end_cost}
    V^{end}(\bar{s},g) = \liminf_{R\rightarrow \infty}\inf_{\tilde{s}\in C^1([-R,0])} \Big\{ \int_{-R}^{0}\Big(\mathcal{W}_\beta\big(\tilde{s}(t)\big) + \frac{1}{2\beta} \Psi\big(\tilde{s}(t),\tilde{s}'(t)\big)\Big) dt +  g\,\tilde{s}(0) + \frac{1}{2\beta}\Phi\big(\tilde{s}(0)\big);  \tilde{s}(-R)=\bar{s}\Big\}.
\end{align}
\end{theorem}
\begin{proof}
The inequality $\leq$ for both (\ref{eqn:initial_cost}) and (\ref{eqn:end_cost}) is immediate as the one-dimensional solutions may be used in the definition of $V^{init}(s_,\bar{s})$ and $V^{end}(\bar{s},g)$ by extending as constants in space and incur the same cost.

For other direction we consider $s\in \mathcal{X}_R^{init}(s_0,\bar{s},\square)$.  We may find a regular value for $x$ such that
$
    q(t) = s(t,x)
$
satisfies $q(0) = s_0$ and 
$$
    \frac{1}{|\square|} \int_{0}^R\int_{\square}\Big( \mathcal{W}_\beta\big(s(t,x)\big) + \Psi\big(s(t,x),\partial_t s(t,x)\big)\Big)dx\, dt\geq\int_{0}^R\Big( \mathcal{W}_\beta\big(q(t)\big) +\frac{1}{2\beta} \Psi\big(q(t),q'(t)\big)\Big)dt.
$$
The inequality $\geq$ in (\ref{eqn:initial_cost}) follows as the nonlocal term is nonnegative.

Similarly, for (\ref{eqn:end_cost}) we consider $s\in \mathcal{X}_R^{end}(\bar{s},\square)$, and find a regular value of $x$ such that $q(t)=s(t,x)$ and
\begin{align*}
    &\ \frac{1}{|\square|} \int_{-R}^0\int_\square \Big( \mathcal{W}_\beta\big(s(t,x)\big) + \Psi\big(s(t,x),\partial_t s(t,x)\big)\Big)dx\, dt + \frac{1}{|\square|}\int_{\square}\Big[g\, s(0,x) + \frac{1}{2\beta} \Phi\big(s(0,x)\big)\Big] dx\\
    \geq&\ \int_{-R}^0 \Big( \mathcal{W}_\beta\big(q(t)\big) +\frac{1}{2\beta} \Psi\big(q(t),q'(t)\big)\Big)dt + g\, q(0) + \frac{ 1}{2\beta}\Phi\big(q(0)\big).
\end{align*}
The result follows.
\end{proof}

\begin{remark}\label{layer_symmetry} 
The decomposed energy has a symmetry, $\tilde{s}(t)=-\tilde{s}(-t)$, by evenness of $\mathcal{W}_\beta$ and $\Psi$.  This interesting observation, not obvious from the original formulation, yields in particular that
$$
    V^{end}(\bar{s},g) = \inf_{s_0}\Big\{ V^{init}({s}_0,\bar{s}) + g\, {s}_0 + \frac{1}{\beta}\Phi(s_0)\Big\} \hbox{ for } \ g \in \R,\ \bar{s} \in \{\pm\mathfrak{s}\}.
$$

So long as $-\mathfrak{s}<s_0<\mathfrak{s}$ the solution for $V^{init}(s_0,\bar{s})$ is a time translation of the same `heteroclinic' solution.  When $s_0\leq -\mathfrak{s}$ and $\bar{s}=\mathfrak{s}$ the solution for $\tilde{s}$ in the definition of $V^{init}$, (\ref{eqn:initial_cost}), does not exist, although the infimum is still well defined. 
\end{remark}

\begin{remark}\label{front_prop}
(Controlled front propagation).
We may also relate the macroscopic problem to a problem of the optimal control of the propagation front, which has been studied in \cite{bressan2021optimal}, \cite{bressan2021optimalmoving}. Consider that the unit normal $\nu=(\nu_t,\nu_x)$, and when $|\nu_x|\not=0$ the front speed may be expressed as $c = \frac{\nu_t}{|\nu_x|}$.  We let $\hat{\nu}_x = \frac{\nu_x}{|\nu_x|}$ denote the spatial unit-normal. The anisotropic minimal surface problem for $\Sigma$, may now be converted into a problem of controlled front propagation where the cost rate to propagate the front with velocity $c$ with spatial unit-normal $\hat{\nu}_x$ is given by
$$
\tilde{L}(c;\hat{\nu}_x) := \sqrt{1+c^2}\bar{L}(\nu),
$$
where clearly $\nu$ can be recovered from $c$ and $\hat{\nu}_x$ as $\nu = (c,\hat{\nu}_x) /\sqrt{1+c^2}$. 
By an application of Fubini's theorem and the coarea formula, we may express
\begin{align*}
\int_{\Sigma} \bar{L}\big(\nu(\tau,z)\big) d\mathcal{H}^d = \int_0^T \int_{\Sigma_\tau}\tilde{L}\big(c(\tau,z);\hat{\nu}_x(\tau,z)\big) d\mathcal{H}^{d-1}\, d\tau.
\end{align*}
 Thus the macroscopic problem is reinterpreted as controlling the wave speed of the evolving front $\Sigma_t$. This formulation recovers some optimal control structure of the problem. A more rigorous expression of the controlled front problem is given in \cite{bressan2021optimalmoving}.
\end{remark}

A partial result holds for the interfacial energy, reducing the problem to the directions ($t,\omega\, \hat{\nu}_x$) when $|\nu_x|\not =0$.
For a unit-normal $e$, we let
$$
    J^e(r) := \int_{\R^{d-1}} J(r\, e + y)dy,
$$
where the integral is taken over the subspace orthogonal to $e$. 
Given $\xi \in \R^d$ and a unit-vector $e\in \R^d$ we set $\xi_{\perp e} = \xi - (\xi\cdot e)e$.

\begin{proposition}\label{prop:partial_characterization}
Given a unit vector $\nu$ with $|\nu_x|\not=0$, 
 assume that the Fourier transform $\mathcal{F} J(\xi)$ is maximized at $\mathcal{F} J(\xi_{\perp \hat{\nu}_x})$. 

Then the macroscopic interfacial energy $\bar{L}(\nu)$ is given by the two-dimensional reduction where we limit the dependence of functions in $\mathcal{X}_R(\square)$ to only $(t,x\cdot \hat{\nu}_x)$.

\end{proposition}

The assumption on $\mathcal{F}J$ is satisfied for instance when $J$ is a Gaussian centered at zero.

\begin{proof}
We extend $s$ to $\R^{d+1}$ by zero, and the Plancherel/Parseval theorem and convolution formula states that
\begin{align*}
&\ \int_{\R^d} s(t,x) (J*s)(t,x)dx\\
 =&\ \int_{\R^d} (\mathcal{F}J)(\xi) \Big((\mathcal{F} s)(t,\xi)\Big)^2d\xi\\
\leq&\  \int_{\R^d} (\mathcal{F}J)(\xi_{\perp \hat{\nu}_x}) \Big((\mathcal{F} s)(t,\xi)\Big)^2d\xi
\end{align*}
by our assumption on $\mathcal{F}J$.

The inverse Fourier transform $\mathcal{F}J(\xi_{\perp \hat{\nu}_x})$ in all variables is formally $\delta_{\perp \hat{\nu}_x}J(x)$ where $\delta_{\perp \hat{\nu}_x}$ is the $d-1$ Hausdorff measure on the subspace orthogonal to $\hat{\nu}_x$, and
$$
    \big((\delta_{\perp \hat{\nu}_x} J) *s\big)(t,x)=\int_{\R} J^{\hat{\nu}_x}(\hat{\nu}_x\cdot x- \omega')s(t,x_{\perp \hat{\nu}_x} + \omega'\, \hat{\nu}_x)d\omega',
$$
so
\begin{align*}
&\ \int_{\R^d} s(t,x) (J*s)(t,x)dx\\
\leq &\ \int_{\R^d} s(t,x) (\delta_{x_{\perp \hat{\nu}_x}} J *s)(t,x)dx\\
=&\ \int_{\R} \int_{\R^{d-1}}\int_{\R} J^{\hat{\nu}_x}(\omega-\omega')s(t,x_{\perp \hat{\nu}_x} + \omega \hat{\nu}_x)s(t,x_{\perp \hat{\nu}_x} + \omega' \hat{\nu}_x)d\omega'\, dx_{\perp \hat{\nu}_x}\, d\omega.
\end{align*}
Equality above holds when $s$ does not depend on $x_{\perp\hat{\nu}_x}$.  The problem for $\bar{L}(\nu)$ is then equivalent to minimizing over $s$ that only depend on $t$ and $\omega = x\cdot \hat{\nu}_x$.
\end{proof}

\subsection*{Conjecture}

We suspect that, under the assumptions of Proposition \ref{prop:partial_characterization}, the limit cost can be characterized entirely in terms of the travelling wave solutions. Potential lack of regularity and topology of the interface associated with the limiting cost makes it difficult to verify our ansatz. More precisely we conjecture that the interfacial cost $\bar{L}$ can be characterized using the front speed $c = \nu_t/|\nu_x|$ (the ratio of the size of normal in the time and the spatial direction). Indeed we expect that (recall $\bar{L}(\nu)= 
\frac{1}{\sqrt{1+c^2}}\,\tilde{L}(c)$) we have
\begin{align}\label{eqn:interfacial_cost}
   \tilde{L}(c) = \inf_{q\in C^1(\R)}\Big\{\int_{\R}\Big( \mathcal{W}_\beta\big(q(\xi)\big) + \frac{1}{2\beta}\Psi\big(q(\xi), c\, q'(\xi)\big) + \frac{1}{4}&\int_{\R}\big(J^e(\xi-\eta)q(\xi)-q(\eta)\big)^2d\eta\Big)d\xi;\\
   &\ \lim_{\xi\rightarrow -\infty} q(\xi)=-\mathfrak{s},\ \lim_{\xi\rightarrow +\infty} q(\xi)=\mathfrak{s}\Big\}.\nonumber
\end{align}
A minimizer of (\ref{eqn:interfacial_cost}), would then construct travelling wave solutions of the form $s(t,x) = q(c\, t-\hat{\nu}_x\cdot x)$.


\begin{figure}
\centering
     \begin{subfigure}[b]{0.4\textwidth}
         \centering
         \includegraphics[width=\textwidth]{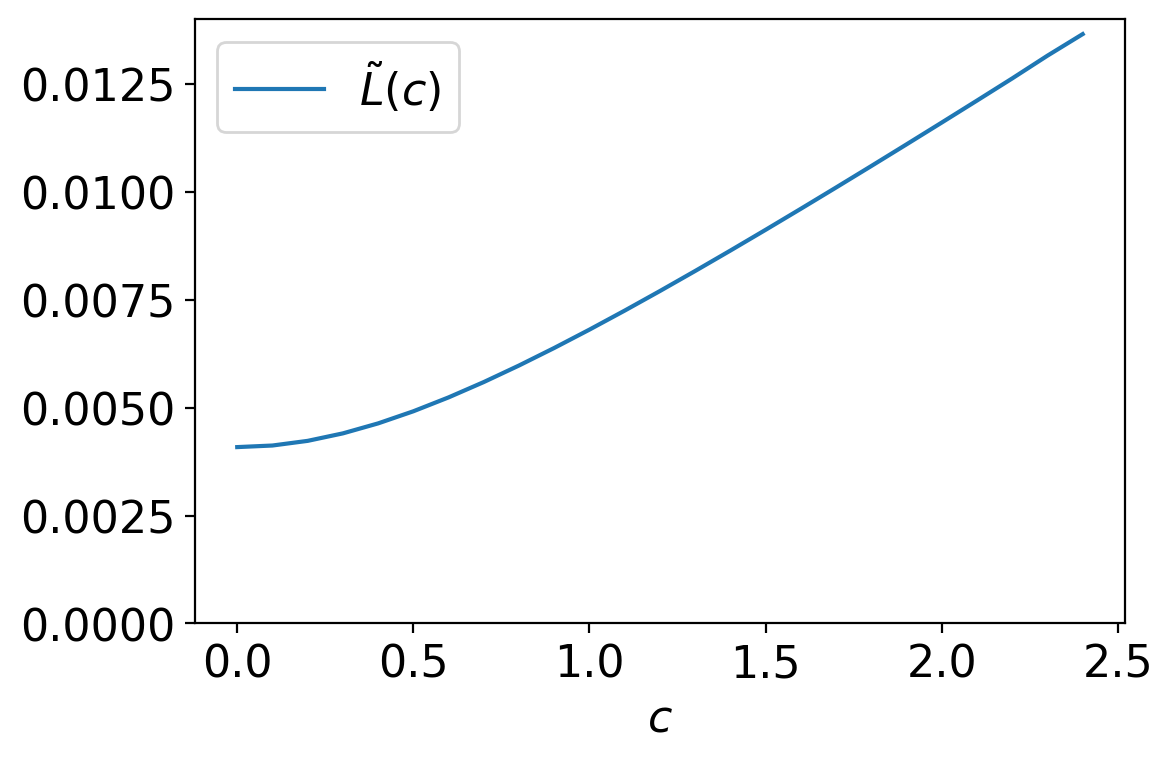}
     \end{subfigure}
     \hfill
     \begin{subfigure}[b]{0.4\textwidth}
         \centering
         \includegraphics[width=\textwidth]{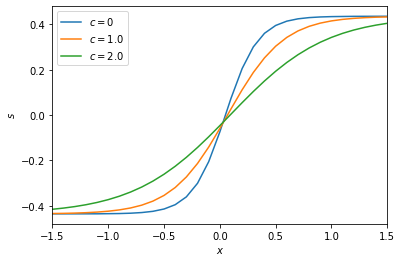}
     \end{subfigure}
\caption{On the left, a plot of $\tilde{L}(c)$. On the right, a spatial slice of three solutions to the cell problem with different front speeds $c$. Parameters are $\beta=0.9^{-1}$ and $\hat{J}=1$.}
\end{figure}

\begin{remark}

If the conjecture holds, then it follows that
$$
   \lim_{c\rightarrow \infty} \frac{1}{\sqrt{1+c^2}}\, \tilde{L}(c) =  V^{init}(-\mathfrak{s},\mathfrak{s}),
$$
as the infinite speed transition is equivalent to the microscopic in time switching from $-\mathfrak{s}$ to $\mathfrak{s}$. 
\end{remark}

\subsection{Proofs}\label{s.sec3proofs}

Here we present some of the longer proofs from earlier in this section that we postponed: the decomposition formula (Proposition ~\ref{p.W_decomposition}), optimality of the constant solutions (Proposition~\ref{prop:const_optimality}), and the improvement of cost for states bounded between the equilibria (Lemma~\ref{lem:bdd}).

\subsubsection{Proof of Proposition ~\ref{p.W_decomposition} }
First we compute the optimal controls, $A_\pm$, as a function of $V$. Changing $A_\pm$ while preserving the equality
\[ V =A_-\, \frac{1-S}{2} - A_+\, \frac{1+S}{2} \]
only affects the term $L(S,A_\pm)$ in the energy so by Lagrange multipliers there is $B\in \R$
\[ \frac{\partial L}{\partial A_\pm}(S,A_\pm) = \pm B \frac{1\pm S}{2}\]
or
\[ \frac{1\pm S}{2} \log A_\pm = \pm B \frac{1\pm S}{2}\]
and so
\[ \log A_+A_- = B - B = 0.\]

Plugging this back into the constraint ODE we find the quadratic equation
\[ A_-^2\frac{1-S}{2} - VA_- - \frac{1+S}{2} = 0\]
so taking the positive root of this equation and then using the constraint $A_+A_- = 1$ we find
\[ (1\pm S)A_\pm(S,V) = \sqrt{V^2+(1-S)(1+S)}\mp V.\]
These are strictly positive, monotone, and convex in $V$. 

Note that
\begin{equation}\label{e.oneq'deriv}
(1\pm S) \frac{\partial}{\partial V}A_\pm(S,V) =\frac{V}{\sqrt{V^2+(1-S)(1+S)}}\mp1 = \frac{\mp(1\pm S)A_\pm}{\sqrt{V^2+(1-S)(1+S)}}\end{equation}
and, in particular,
\[ (1\pm S) \frac{\partial}{\partial V}A_\pm(S,0) = \mp 1.\]
Differentiating (\ref{e.oneq'deriv}) again we note that
\begin{equation}\label{e.twoq'deriv}
 \frac{\partial^2}{\partial V^2}\left[\frac{1+ S}{2}A_+(S,V)\right] = \frac{\partial^2}{\partial V^2}\left[\frac{1- S}{2}A_-(S,V)\right]. 
 \end{equation}

Now plugging this into the definition of $W(S,V)$
\[ W(S,V) = L(S,A_\pm(S,V)) - \frac{1}{2} \hat{J} S^2 - \bs\Lambda\]
we see the desired properties of $W$ are a matter of calculus.  First, compute
\[ D^2_{A_\pm}L(S,A_\pm) = \left[\begin{array}{cc}\frac{1+S}{2\beta }\frac{1}{A_+} & 0 \\ 0 & \frac{1-S}{2\beta }\frac{1}{A_-}\end{array}\right].\]
So computing directly
\begin{equation}\label{e.dpWformula}
 \partial_{V}W(S,V) = \frac{1}{2\beta} \sum_{\pm}(1\pm S)\log A_\pm \frac{\partial E_{\pm}}{\partial V}\end{equation}
 and, in particular,
 \[ \partial_{V}W(S,0) = \frac{1}{2\beta} \sum_{\pm}\mp \log \sqrt{\frac{1\mp S}{1\pm S}} = \frac{1}{2\beta} \log \frac{1+S}{1-S}.\]
For the second derivative we continue computing
\[  \partial_{V}^2W(S,V) = \frac{1}{2\beta}\sum_{\pm} \frac{1}{1\pm S}\frac{1}{A_\pm} \left((1\pm S)\frac{\partial A_\pm}{\partial V}\right)^2 + \frac{1}{\beta}\sum_{\pm}\frac{1\pm S}{2}(\log A_\pm -1) \frac{\partial A_\pm}{\partial V^2}.\]
Using (\ref{e.twoq'deriv}) and $\log(A_+A_-) = 1$ we see that the second term above vanishes and so
\begin{align*}
 \partial_{V}^2W(S,V) &= \frac{1}{2\beta}\sum_{\pm} \frac{1}{1\pm S}\frac{1}{A_\pm} \left((1\pm S)\frac{\partial A_\pm}{\partial V}\right)^2 \\
 &= \frac{1}{2\beta}\sum_{\pm} \frac{(1\pm S)A_\pm}{V^2+(1-S)(1+S)}\\
 &= \frac{1}{2\beta} \frac{1}{\sqrt{V^2+(1-S)(1+S)}} 
 \end{align*}
where we have used (\ref{e.oneq'deriv}) to get the second equality.   This gives the convexity in the $V$ variable and we can go a bit further to make a strict convexity estimate.
 
 Then by the fundamental theorem of calculus
 \[ W(S,V) = W(S,0) + \partial_{V}W(S,0) V + \frac{1}{4\beta }\int_0^V\frac{(V-Z)}{\sqrt{Z^2+(1-S)(1+S)}} dZ.\]
We have the formula $\partial_{V}W(S,0) = \frac{1}{2\beta} \log \frac{1+S}{1-S}$ from above, which we express as $\partial_{V}W(S,0) = \frac{1}{2\beta}\Phi'(S)$. 

For the remainder term we define, as in the statement of the theorem,
\[ \Psi(S,V) =  \int_0^V\frac{(V-Z)}{\sqrt{Z^2+(1-S)(1+S)}} dZ.\]
We note that, for $ 0 \leq Z \leq V \leq \sqrt{(1-S)(1+S)}$, 
 \begin{equation}\label{e.near0-psi-bd}
     \frac{1}{\sqrt{2(1-S)(1+S)}}\leq \frac{1}{\sqrt{Z^2+(1-S)(1+S)}} \leq \frac{1}{\sqrt{(1-S)(1+S)}}
 \end{equation} 
so
\[ \frac{1}{\sqrt{2}\sqrt{(1-S)(1+S)}} V^2\leq \Psi(V) \leq \frac{1}{\sqrt{(1-S)(1+S)}} V^2.\]
Note that the upper bound is true for arbitrary $V$.

While for $V \geq \sqrt{(1-S)(1+S)}$
\begin{align*}
 \Psi(S,V) &\geq  \int_{\frac{1}{3}\sqrt{(1-S)(1+S)}}^V\frac{V-Z}{2Z} dZ \\
 &= \frac{1}{4}\bigg[V(\log Z-1)\bigg]_{\frac{1}{a}\sqrt{(1-S)(1+S)}}^V  \\
 &= \frac{1}{4}V\log\frac{3V}{\sqrt{(1-S)(1+S)}}\\
 &\geq \frac{1}{4}V\log\left(2+\frac{V}{\sqrt{(1-S)(1+S)}}\right).
 \end{align*}
 The corresponding upper bound will not be used anywhere so we omit the proof.

 \medskip
 
 Next, we discuss the properties of $\partial_V\Psi(S,V)$. note that 
 \[\partial^2_V\Psi(S,V) = 4\beta\, \partial_V^2W(S,V) = 2 \frac{1}{\sqrt{V^2+(1-S)(1+S)}} \]
 which implies the convexity of $\Psi$ and the concavity of $\partial_V\Psi$.  
 
 For $V \leq \sqrt{(1-S)(1+S)}$ we use again \eref{near0-psi-bd} to find
 \[ \partial_V\Psi(S,V) \sim \frac{1}{\sqrt{(1-S)(1+S)}} V \ \hbox{ for } \ 0< V\leq  \sqrt{(1-S)(1+S)}\]
 For $V \geq \sqrt{(1-S)(1+S)}$ we use $\partial^2_V\Psi (S,V) \sim V^{-1}$ to find
 \[ \partial_V\Psi(S,V) \sim \log \left(2+\frac{V}{\sqrt{(1-S)(1+S)}}\right).\]
 
 \medskip
 
  \emph{double-well potential.}  Lastly, we consider the double/single well nature of the potential 
  \begin{align*}
   \mathcal{W}_\beta(S) &= L\big(S,A_\pm(S,0)\big) - \frac{1}{2} \hat{J} S^2 - \Lambda \\
   &= \frac{1}{2\beta}\sum_{\pm} (1\pm S)A_\pm(S,0)\Big(\log\big(A_\pm(S,0)\big) - 1\Big) - \frac{1}{2} \hat{J} S^2 - \Lambda \\
   &=\frac{1}{2\beta}\sqrt{(1+S)(1-S)}(\log A_+A_- - 2)\\
   &=-\frac{1}{\beta}\sqrt{(1+S)(1-S)}- \frac{1}{2} \hat{J} S^2 - \Lambda.
   \end{align*}
   Note that $\mathcal{W}_\beta$ always has a critical point at $S = 0$ and
   \[ \mathcal{W}_{\beta}''(0) = \frac{1}{\beta} - \hat{J} \]
   so we can see again the critical value at $\beta \hat{J} = 1$.  When $\beta \hat{J} > 1$ the critical point at the origin is a local maximum and there are two local minima at
   \[ \mathfrak{s} = \sqrt{1-\beta^{-2}\hat{J}^{-2}}\]
   and $-\mathfrak{s}$.
\hfill$\Box$

\subsubsection{Proof of Proposition \ref{prop:const_optimality}}

\begin{proof}
The proposition also follows from Proposition \ref{prop:W_decomposition} as we have $\Psi(V,S)\geq 0$ with equality when $V=0$, 
$$
\frac{1}{4}\int_{\T^d}\int_{\T^d} J^\lambda(w-z)\big(s(w)-s(z)\big)^2dw\, dz\geq 0
$$
with equality when $s$ is constant,
and $\mathcal{W}_\beta(S)\geq 0$ with equality when $S\in \{-\mathfrak{s},\mathfrak{s}\}$.  From the proof of Proposition \ref{prop:W_decomposition} we see that the case $S=\mathfrak{s}$ corresponds exactly with $A = A_\pm (\mathfrak{s},0)$, and the case $S=-\mathfrak{s}$ corresponds exactly with $A = A_\pm (-\mathfrak{s},0)$.
\end{proof}

\subsubsection{ Proof of Lemma~\ref{lem:bdd}.}
We first show the result for  ${s}^{b-}(\tau,z)=\min\{s(\tau,z), \mathfrak{s}\}$ and then restricting to $s(\tau,z)\geq -\mathfrak{s}$ is similar.

Consider the set $\Omega = \{(\tau,z): s(\tau,z)>\mathfrak{s}\}$.  We also define ${s}^{b+}(\tau,z)=\max\{{s}(\tau,z), \mathfrak{s}\}$. 

We first note that for each $z\in \mathbb{T}^{d}$, ${s}^{b+}$ and ${s}^{b-}$ are weakly differentiable in time with 
$$
    \partial_\tau {s}^{b+}(\tau,z) = \begin{cases} \partial_\tau s(\tau,z) &  (\tau,z) \in \Omega\\
    0 & {\rm otherwise},
    \end{cases}
$$
and 
$$
    \partial_\tau {s}^{b-}(\tau,z) = \begin{cases} 0 &  (\tau,z) \in \Omega\\
    \partial_\tau s(\tau,z) & {\rm otherwise}.
    \end{cases}
$$

The local part of the cost separates into the two domains, that is,
\begin{align*}
\mathcal{W}_\beta\big(s(\tau,z)\big) + \frac{1}{2\beta}  \Psi\big(s(\tau,z),\lambda\, \partial_\tau s(\tau,z)\big)=&\ \mathcal{W}_\beta\big(s^{b-}(\tau,z)\big) + \frac{1}{2\beta}  \Psi\big(s^{b-}(\tau,z),\lambda\, \partial_\tau s^{b-}(\tau,z)\big)\\
&\ +\mathcal{W}_\beta\big(s^{b+}(\tau,z)\big) + \frac{1}{2\beta}  \Psi\big(s^{b+}(\tau,z),\lambda\, \partial_\tau s^{b+}(\tau,z)\big)
\end{align*}

For the nonlocal part,  we will use that for $(\tau,z)\in [0,T]\times \mathbb{T}^d$,
$$
    {s}^{b-}(\tau,z)+{s}^{b+}(\tau,z) = s(\tau,z)+\mathfrak{s},
$$
and ${s}^{b-}(\tau,z)\leq s(\tau,z)$, and ${s}^{b+}(\tau,z)\geq s(\tau,z)$.  By nonnegativity of $J^\lambda$ we also have that
$$
    (J^\lambda*{s}^{b-})(\tau,z)\leq (J^\lambda*s)(\tau,z), \quad (J^\lambda*{s}^{b+})(\tau,z)\geq (J^\lambda*s)(\tau,z).
$$

For $(\tau,z)\in \Omega$, we have $s(\tau,z) = {s}^{b+}(\tau,z)$ and ${s}^{b-}(\tau,z)=\mathfrak{s}$ and (dropping the dependence on $(\tau,z)$ for ease of notation)
\begin{align*}
s\, J^\lambda*s \leq&\ s\, J^\lambda*s+(\mathfrak{s}-s)\, (J^\lambda*s - J^\lambda*{s}^{b+})\\
=&\ \mathfrak{s}\, (J^\lambda*s - J^\lambda*s^{b+})  +s\, J^\lambda*s^{b+}\\
=&\ \mathfrak{s}\, (J^\lambda*s^{b-} - \hat{J^\lambda}\, \mathfrak{s})  +s^{b+}\, J^\lambda*s^{b+}\\
=&\ s^{b-}\, J^\lambda*s^{b-}+s^{b+}\, J^\lambda*s^{b+}- \hat{J^\lambda}\mathfrak{s}^2.
\end{align*}
Similarly, for $(t,x)\in \Omega^c$, we have $s(t,x) = s^{b-}(t,x)$ and $s^{b+}(t,x)=\mathfrak{s}$ and 
\begin{align*}
s\, J^\lambda*s \leq&\ s\, J^\lambda*s+(\mathfrak{s}-s)\, (J^\lambda*s - J^\lambda*s^{b-})\\
=&\ \mathfrak{s}\, (J^\lambda*s - J^\lambda*s^{b-})  +s\, J^\lambda*s^{b-}\\
=&\ \mathfrak{s}\, (J^\lambda*s^{b+} - \hat{J}\, \mathfrak{s})  +s^{b-}\, J^\lambda*s^{b-}\\
=&\ s^{b-}\, J^\lambda*s^{b-}+s^{b+}\, J^\lambda*s^{b+}- \hat{J}\mathfrak{s}^2.
\end{align*}

This implies that
$$
    \mathcal{G}^\lambda\big(s;[0,T]\times \T^d\big)\geq \mathcal{G}^\lambda\big(s^{b-};[0,T]\times \T^d\big)+\mathcal{G}^\lambda\big(s^{b+};[0,T]\times\T^d\big)-\mathcal{G}^\lambda\big(\mathfrak{s};[0,T]\times\T^d\big),
$$
 and, by Proposition \ref{prop:const_optimality}, 
$$
     \mathcal{G}^\lambda\big(s^{b+};[0,T]\times \T^d\big) \geq  \mathcal{G}^\lambda\big(\mathfrak{s};[0,T]\times\T^d\big) = 0.
$$
We conclude
$$
 \mathcal{G}^\lambda\big(s;[0,T]\times \T^d\big)\geq \mathcal{G}^\lambda\big(s^{b-};[0,T]\times \T^d\big).
 $$
 
 Observing that $\Phi$ is convex and $|g(z)|\leq \frac{1}{2\beta}\Phi'(\mathfrak{s})$, we have
 $$
    s^{b-}(T,z)\, g(z) + \frac{1}{2\beta}\Phi\big(s^{b-}(T,z)\big) \leq s(T,z)\, g(z) + \frac{1}{2\beta}\Phi\big(s(T,z)\big),
 $$
 which finishes the proof.
\hfill$\Box$

\section{Main result}

Our main result is a type of Gamma-convergence, akin to  Theorem 1.4 of \cite{alberti1998non} which addresses nonlocal Allen-Cahn equation. In addition to the previous assumption \aref{Jpos} and \aref{Jmoment} on $J$ we will always assume in this section
\begin{enumerate}[label = (A\arabic*)]
 \setcounter{enumi}{2}
\item (Super-criticality) \label{a.Jsupercritical}\[\beta \hat{J} > 1.\]
\end{enumerate}
Under this assumption $\mathcal{W}_\beta$ is a double-well potential with two distinct minimizers $\pm \mathfrak{s}$.  In the critical or subcritical case the asymptotic behavior will be completely different.
\begin{theorem}\label{thm:main}
 Consider the initial data $s_0\in L^1(\T^d)$ with $|s_0|\leq \mathfrak{s}$ and terminal data $g\in L^1(\T^d)$ with $|g(z)|\leq \frac{1}{2\beta}\Phi'(\mathfrak{s})$, where $\Phi$ is given in Proposition \ref{decomposition}. Then the following holds:
\begin{enumerate}[label=(\roman*)]
\item For every sequence $(s^\lambda,a^\lambda)$ with $\hat{s}^\lambda(0,\cdot) = s_0$ and uniformly bounded cost, there is a convergent subsequence in macroscopic variables, $\hat{s}^{\lambda} \to \bar{s}$ in $L^1 ([0,T]\times \mathbb{T}^d)$. Moreover $\hat{s}^\lambda \rightarrow\bar{s}\in BV((0,T)\times \mathbb{T}^d;\{-\mathfrak{s},\mathfrak{s}\})$ and
$$
    \liminf_{\lambda\rightarrow^+ 0} \mathcal{C}^\lambda(s^\lambda,a^\lambda) \geq \bar{V}(s_0,g,\bar{s}).
$$ 
\item For every $\bar{s}\in BV((0,T)\times \mathbb{T}^d;\{-\mathfrak{s},\mathfrak{s}\})$, there exists a sequence $(s^\lambda,a^\lambda)$ such that $\hat{s}^\lambda\rightarrow \bar{s}$ in $L^1([0,T]\times \mathbb{T}^d), \,\hat{s}^\lambda(0,\cdot) = s_0$, and
$$
    \limsup_{\lambda\rightarrow^+ 0} \mathcal{C}^\lambda(s^\lambda,a^\lambda) \leq \bar{V}(s_0,g,\bar{s}).
$$
\end{enumerate}
\end{theorem}

Let us briefly outline the strategy carried out in this section to prove the above convergence result.  

\medskip

In \sref{compactness-lemma} we show that, in an appropriate sense, the cost $\mathcal{C}^\lambda$ asymptotically controls the $BV$ norm and so sequences $s^\lambda$ with bounded cost $\mathcal{C}^\lambda$ satisfy appropriate compactness properties. The argument combines ideas for local and nonlocal Allen-Cahn problems in a slightly delicate, but largely standard, way.  Note that in this stage we are yet to characterize the macroscopic cost $\bar{V}$.

In \sref{patching-lemma}, we prove several technical ``patching" results which are key to the later $\Gamma$-convergence arguments. These are quantitative versions of localization results that are naturally needed to ensure that our macroscopic Lagrangian depends locally only on the normal directions of the interface between the state $-\mathfrak{s}$ and $\mathfrak{s}$.  More precisely, we show that sequences of test minimizers defined in disjoint domains can be patched along a joint boundary without increasing the energy too much as long as an appropriate notion of trace matches along this joint boundary.  The ideas in this section are inspired by \cite{alberti1998non}, but the argument is technically more difficult because the cost functional requires some microscopic regularity in the temporal direction.

Then in \sref{liminf} and \sref{limsup} we carry out the typical two part $\Gamma$-convergence argument.  

\medskip

The argument for the lower bound inequality in \sref{liminf} follows a classical general technique introduced by Fonseca and M\"uller \cite{Fonseca1993}: the problem can be reduced to establishing a pointwise lower bound on the densities for a subsequential limit of the particular test minimizer sequence $s^\lambda$.  In technical terms the patching and compactness lemmas play a key role here.

For the upper bound inequality in \sref{limsup} we follow a beautiful idea introduced by Alberti and Bellettini \cite{alberti1998non} of induction on polyhedral regions.  By approximating with polyhedral regions instead of smooth sets, Alberti and Bellettini reduced the entire difficulty of controlling lower order terms related to the ``bending" hyperplanes to a relatively simple patching argument where polyhedral test regions meet transversally to the interface.  This argument adapts nicely to our setting because we have also established a technique for patching local test minimizers.

\subsection{Compactness}\label{s.compactness-lemma} In this section we show that sequences with bounded $\mathcal{G}^\lambda$ are precompact in $L^1$ and all cluster points are indicator functions of sets of bounded variation.  As we have explained, the energy $\mathcal{G}^\lambda$ is understood to measure the space-time surface area of the interface between the $\pm \mathfrak{s}$ phases.  Thus a $BV$ -like compactness result is to be expected.  We note that the estimates we obtain are not uniform as $\beta\, \hat{J}\rightarrow_+1$, , i.e., $\mathfrak{s}\rightarrow_+0$,  reflecting the possibility of some more complex phenomena occurring near the critical parameter values.

Of course this type of result is well known for Allen-Cahn \cite{modica1987gradient} and nonlocal Allen-Cahn functionals \cite{alberti1998non}.  Our functional is a mix of the two, and with some technical tricks inspired by the two cases we can prove the compactness.

Our first step is to really make a decomposition into a typical local Allen-Cahn type functional measuring the temporal variations, and a nonlocal Allen-Cahn functional measuring the spatial variations:
\[
G^\lambda(\hat{s}; A') =Y^\lambda(\hat{s};A')+X^\lambda(\hat{s};A') 
\]
where 
\[ 
Y^\lambda(\hat{s};A') :=\int_{A'}\Big[\frac{1}{2\lambda}\mathcal{W}_\beta\big(\hat{s}(\tau,z)\big) +\frac{1}{2\beta\, \lambda}\Psi\big(\hat{s}(\tau,z),\lambda\, \partial_\tau\hat{s}(\tau,z)\big)\Big] dz
\]
and
\[ 
    X^\lambda(\hat{s};A') := \int_{A'}\frac{1}{2\lambda}\mathcal{W}_\beta\big(\hat{s}(\tau,z)\big) \ dz+ \frac{1}{4\lambda}\int_{A'}\int_{A'} J^\lambda(z -w)\big(\hat{s}(\tau,z) - \hat{s}(\tau,w)\big)^2\ dw\, dz.
\]
The space-time energy splits analogously
\[ \mathcal{G}^\lambda(\hat{s},A) = \mathcal{Y}^\lambda(\hat{s};A)+\mathcal{X}^\lambda(\hat{s};A) \ \hbox{ for } \ A \subset \R^{d+1}\]
where $\mathcal{X}^\lambda$ and $\mathcal{Y}^\lambda$ are naturally defined as temporal integrals of $X^\lambda$ and $Y^\lambda$ as was done for $\mathcal{G}^\lambda$ in (\ref{eqn:space_time_energies}).

\begin{proposition}\label{prop:compactness}
Let $\Omega \subset (0,T) \times \T^d$ be a polyhedral space-time region and $\hat{s}^\lambda: \Omega\to [-\mathfrak{s},\mathfrak{s}]$ be a sequence as $\lambda\rightarrow^+0$ with 
$$\sup_{\lambda>0}\mathcal{G}^{\lambda}\big(\hat{s}^\lambda;\Omega)<+\infty.$$
We assume that $\beta\, \hat{J}>1$.  Then $\hat{s}^\lambda$ is relatively compact in $L^1( \Omega)$ and each of its cluster points belongs to $BV( \Omega;\{-\mathfrak{s},\mathfrak{s}\})$.
\end{proposition}
The proof is a combination of the compactness arguments for local and nonlocal Allen-Cahn. 
\begin{proof}
By Proposition \ref{p.patchingestimate} we can extend $\hat{s}^\lambda$ to be equal to $\mathfrak{s}$ in $(0,T) \times \T^d$ and then replace this extension by an $L^1(\Omega)$ equivalent sequence defined on the entire $(0,T)\times\T^d$ and with $\mathcal{G}^\lambda(\hat{s}^\lambda;(0,T) \times \T^d) \leq CM$ where the constant depends on the domain $\Omega$.

Note that both $\mathcal{Y}^\lambda(\hat{s}^\lambda;(0,T) \times \T^d)\geq 0$ and $\mathcal{X}^\lambda(\hat{s}^\lambda;(0,T) \times \T^d)\geq 0$ so both are bounded by $M:= \sup_{\lambda>0}\mathcal{G}^\lambda(\hat{s}^\lambda;{(0,T) \times \T^d})$.  

First we estimate the time derivative using the bound on $\mathcal{Y}^{\lambda}$ and following a standard argument for local Allen-Cahn functionals. From the Young's inequality,
\[ \sqrt{2\beta^{-1}\lambda^{-2}\mathcal{W}_\beta(\hat{s})\Psi(\hat{s},\lambda \partial_\tau \hat{s})} \leq \lambda^{-1}\mathcal{W}_\beta(\hat{s})+\frac{1}{2\beta}\lambda^{-1}\Psi(\hat{s},\lambda\partial_\tau \hat{s}) \]
Using above with \eqref{estimate} for the set $|\lambda \partial_\tau \hat{s}| \leq 1$ and using the $\Psi$ bound with \eqref{estimate} for the rest, we arrive at
\begin{align*}
&\ \iint_{(0,T) \times \T^d}\Big[ \mathcal{W}_\beta\big(\hat{s}^{\lambda}(\tau,z)\big)^{1/2}|\partial_\tau\hat{s}^{\lambda}(\tau,z)| \chi{\{|\partial_\tau \hat{s}(\tau,z)|\leq \lambda^{-1}\}} +  |\partial_\tau\, \hat{s}(\tau,z)|\chi{\{|\partial_\tau\, \hat{s}(\tau,z)|\geq \lambda^{-1}\}}\Big]dz\, d\tau\\
& \quad \quad \quad \leq\ C\,  \mathcal{Y}^{\lambda}\big(\hat{s}^\lambda; {(0,T) \times \T^d}\big)\leq CM,
\end{align*}
Call $\mathbb{W}(s) := \int_0^s\mathcal{W}_\beta(s)^{1/2}ds$ so that we have proved
\begin{equation}\label{e.ddtWbound}
\iint_{(0,T) \times \T^d} \big| \frac{d}{d\tau} \mathbb{W} \big(\hat{s}(\tau,z)\big)\big| dz\, d\tau  \leq CM.
\end{equation}
\medskip

Next, we use the bound on $\mathcal{X}^\lambda$ to obtain a uniform bound for the spatial gradient of (a mollification of)
$$
    \tilde{s}^\lambda(t,z) := \varphi\big(\hat{s}^\lambda(\tau,z)\big),
$$
where $\varphi$ is the cut-off function 
\begin{equation}\label{cut_off}
    \varphi(s) := \begin{cases} \mathfrak{s} & s >\mathfrak{s}/2 \\
    2s& s\in [-\mathfrak{s}/2,\mathfrak{s}/2]\\
    -\mathfrak{s} & s < -\mathfrak{s}/2.
    \end{cases}
    \end{equation}
    The point of this cut-off is that it is (i) Lipschitz so it doesn't affect the temporal energy $\mathcal{Y}^\lambda$ too much, (ii) it simplifies the computation of the nonlocal part of the energy essentially concentrating the energy on the interface without changing the $L^1$ limit of the sequence (an idea of \cite{alberti1998non}).

Next we mollify at scale $\lambda$ by  $\phi^\lambda(z) := c\lambda^{-d}\phi(z/\lambda)$, where $\phi$ is a nonnegative  (not identically zero) smooth function with compact support  and total mass $c^{-1}$ that satisfies
$$
\phi \leq J*J \hbox{ and } |\nabla \phi|\leq J * J.
$$
The proof of the bound on $|\nabla_z (\phi^{\lambda}*\tilde{s}^{\lambda})|$ follows closely that of Theorem 3.1 in \cite{alberti1998non}.  First, the inequality
$$
    \int_{\T^d}\int_{\T^d} (J^\lambda * J^\lambda)(z-w)|\tilde{s}(\tau,z)-\tilde{s}(\tau,w)|dw\, dz \leq 2 \hat{J} \int_{\T^d}\int_{\T^d} J^\lambda(z-w)|\tilde{s}^\lambda(\tau,z)-\tilde{s}^\lambda(\tau,w)|dz\, dw
$$
is established by direct computation with a change of variables argument. The right hand side is decomposed using the set $H^\lambda_\tau = \{z: \hat{s}^\lambda(\tau,z) \in [-\mathfrak{s}/2,\mathfrak{s}/2]\}$. On ${(0,T) \times \T^d}\backslash H^\lambda_\tau\times {(0,T) \times \T^d}\backslash H^\lambda_\tau$ we have that
\begin{align*}
    |\tilde{s}^\lambda(\tau,z)-\tilde{s}^\lambda(\tau,w)| \leq &\  \frac{4}{\mathfrak{s}}|\hat{s}^\lambda(\tau,z)-\hat{s}^\lambda(\tau,w)|^2,
\end{align*}
using the fact that $\tilde{s}^{\lambda} \in \{\pm \mathfrak{s}\}$ away from $H_{\tau}^{\lambda}$.

Whereas in $H_\tau^\lambda$ we simply bound $|\tilde{s}^\lambda(\tau,z)-\tilde{s}^\lambda(\tau,w)|\leq 2\, \mathfrak{s}$. The area of $H^\lambda_\tau$ can be bounded by a constant times the integral of $\mathcal{W}_\beta(\hat{s})$, since there is $\rho>0$ such that $\mathcal{W}_\beta(s)\geq \rho$ when $s\in [-\mathfrak{s}/2,\mathfrak{s}/2]$.  Along with nonnegativity of the nonlocal term, we have
\begin{equation}\label{bound_set}
    |H^\lambda_\tau| \leq \frac{1}{\rho}\iint_{{(0,T) \times \T^d}} \mathcal{W}_\beta\big(\hat{s}^\lambda\big) d\tau dz \leq \frac{2\lambda}{\rho} M.
    \end{equation}
Therefore
\begin{align*}
    &\ 2 \hat{J} \int_{\T^d}\int_{\T^d} J^\lambda(z-w)|\tilde{s}^\lambda(\tau,z)-\tilde{s}^\lambda(\tau,w)|dz\, dw\\
    & \quad \quad \leq\ 2 \hat{J} \Big(\int_{\T^d\backslash H^\lambda_\tau}\int_{\T^d\backslash H^\lambda_\tau} \frac{4}{\mathfrak{s}} J^\lambda(z-w)|\tilde{s}^\lambda(\tau,z)-\tilde{s}^\lambda(\tau,w)|^2\, dz\, dw + 2\int_{\mathbb{T}^d}\int_{H^\lambda_\tau} 2\, \mathfrak{s}\, dz\, dw\Big)\, d\tau\\
    & \quad \quad \leq \ \lambda\, C\,  X^\lambda\big(\hat{s}(\tau,\cdot);\T^d\big). 
\end{align*}

We use this to estimate the error in mollification
\begin{align*}
    &\ \int_0^T\int_{\T^d} \big| (\phi^\lambda * \tilde{s}^\lambda(\tau,\cdot))(z)-\tilde{s}^\lambda(\tau,z)\big|dz\, d\tau\\
    & \quad \quad =\ \int_0^T\int_{\T^d} \Big|\int_{\T^d} \phi^\lambda(z-w)\big( \tilde{s}^\lambda(\tau,w)-\tilde{s}^\lambda(\tau,z)\big)dw\Big|dz\, d\tau \\
     & \quad \quad \leq\ \int_0^T\int_{\T^d} \int_{\T^d} \phi^\lambda(z-w)\big| \tilde{s}^\lambda(\tau,w)-\tilde{s}^\lambda(\tau,z)\big|dw\, dz\, d\tau \\
    & \quad \quad \leq\  \frac{1}{c} \int_0^T\int_{\T^d}\int_{\T^d} (J^\lambda * J^\lambda)(z-w)\big|\tilde{s}^\lambda(\tau,w)-\tilde{s}^\lambda(\tau,z)\big|dw\, dz\, d\tau\\
    & \quad \quad \leq\  \lambda\, C'\,  \mathcal{X}^\lambda\big(\hat{s};[0,T] \times \mathbb{T}_d\big).
\end{align*}
Using, for the final inequality, the estimates from the previous paragraph.  Thus we obtained
\begin{equation}\label{difference_mollified}
\int_0^T\int_{\T^d} \big| (\phi^\lambda * \tilde{s}^\lambda(\tau,\cdot))(z)-\hat{s}^\lambda(\tau,z)\big|dz\, d\tau \leq CM\lambda.
\end{equation}

By similar computations
\begin{align*}
    &\ \int_0^T\int_{\T^d} \big| \nabla (\phi^\lambda * \tilde{s}^\lambda(\tau,\cdot))(z)\big|dz \, d\tau\\
    & \quad \quad =\ \int_0^T\int_{\T^d}\Big|\int_{\T^d}  (\nabla \phi^\lambda)(z-w) \big(\tilde{s}^\lambda(\tau,w) -\tilde{s}^\lambda(\tau,z)\big)dw\Big|\, dz\, d\tau\\
    & \quad \quad \leq\ \frac{1}{c\, \lambda}\int_0^\tau\int_{\T^d}\int_{\T^d} (J^\lambda * J^\lambda)(z-w)\big|\tilde{s}^\lambda(\tau,w)-\tilde{s}^\lambda(\tau,z)\big|dw\, dz\, d\tau \leq C\mathcal{X}^\lambda\big(\hat{s}^\lambda;{(0,T) \times \T^d}\big).
\end{align*}

Let us now put together above bounds to obtain the global bound. Since $\mathbb{W} $ is invertible on $[-\mathfrak{s}/2,\mathfrak{s}/2]$,  by definition \eqref{cut_off} it follows that $\varphi \circ \mathbb{W}^{-1}$ is Lipschitz.  So using the previous inequality and \eref{ddtWbound} it follows that 
\begin{align*}
    \iint_{(0,T) \times \T^d}\big|D_{t,z}\phi^\lambda *\tilde{s}^\lambda(\tau,z)\big|dz\, d\tau&\leq \iint_{(0,T) \times \T^d}\Big(\big|\nabla\phi^\lambda *\tilde{s}^\lambda(\tau,z)\big| + \big|\phi^\lambda * \frac{d}{d\tau}\left[ (\varphi\circ \mathbb{W}^{-1}) (\mathbb{W}(\hat{s}^\lambda(\tau,z)))\right]\big|\Big)dz\, d\tau\\
    &\leq CM.
\end{align*}
By standard $BV$ compactness results there is a subsequence so that $\phi^\lambda * \tilde{s}\rightarrow \bar{s}$ strongly in $L^1$ and $\bar{s}\in BV({(0,T) \times \T^d})$ with
\[ \iint_{{(0,T) \times \T^d}} |D_{\tau,z} \bar{s}| \leq \liminf_{\lambda \to 0}\iint_{(0,T) \times \T^d}\big|D_{t,z}\phi^\lambda *\tilde{s}^\lambda(\tau,z)\big| \leq CM. \]

\medskip

Finally, we show the convergence of $\hat{s}^\lambda$. Let $K_\lambda := \{(\tau,z): \hat{s}^\lambda(\tau,z)\in [-\mathfrak{s}+\delta,\mathfrak{s}-\delta]\}$.  As in \eqref{bound_set}we have
$$
    |K_\lambda|\leq C(\delta)\, \lambda\,M,
$$
and, from \eqref{difference_mollified} and, since either $|\hat{s}-\tilde{s}|=0$ outside of $K_\lambda$ and otherwise $|\hat{s}-\tilde{s}|\leq 2$,  
\begin{align*}
    &\ \int_0^T\int_{\T^d}| \hat{s}^\lambda(\tau,z)-\phi^\lambda *\tilde{s}^\lambda(\tau,z)| dz\, d\tau\\
    \leq&\ \int_0^T\int_{\T^d}\Big(| \hat{s}^\lambda(\tau,z)-\tilde{s}^\lambda(\tau,z)|+ |\tilde{s}^\lambda(\tau,z)- \phi^\lambda *\tilde{s}^\lambda(\tau,z)|\Big) dz\, d\tau\\
    \leq&\ \delta\, T + 2|K_\lambda| + C'\, \lambda.
\end{align*}
Since $\delta$ is arbitrary,  the sequence $\hat{s}^\lambda$ is equivalent to $\phi^\lambda * \tilde{s}$ in $L^1$ and thus is relatively compact with all of its cluster points in $BV({(0,T) \times \T^d};\{-\mathfrak{s},\mathfrak{s}\})$.
\end{proof}

\subsection{Patching lemma}\label{s.patching-lemma}

In this section we develop a technical tool which will be essential in the proof of $\Gamma$ convergence.  Roughly speaking we look for a way to ``patch" test minimizers which are defined in disjoint domains to be a test minimizer in the union without increasing the total energy too much. To this end  we will need to control the increase of the nonlocal ``Dirichlet" energy and the increase of the local ``Dirichlet" energy for the cost in the form \eqref{cost:decomposition}. For the nonlocal part of the energy we will  follow the ideas in \cite{alberti1998non} section 2.  However, in \cite{alberti1998non} the actual patching can be done in a straightforward way, simply defining a new, possibly discontinuous, test minimizer piecewise. We cannot do this because the local energy penalizes large time derivatives by the term $\lambda^{-1}\Psi(u,\lambda\partial_tu)$ in the energy. Thus the presence of the local energy necessitates a smoother notion of patching.

We introduce a notion of ``trace" on $d-1$-dimensional surfaces, imitating the notions introduced in \cite{alberti1998non}. Define an auxiliary potential
\[\tilde{J}(h) := \int_0^1J\left(\frac{h}{t}\right)\left|\frac{h}{t}\right| \frac{dt}{t^d}\]
and note that from (A2)
\[ \int_{\R^d} \tilde{J}(h) dh = \int_{\R^d} |h|J(h) \ dh < + \infty.\]

Because most of the notions in this section are local we will often work with test spin fields on $(\tau,z)$ in subsets of $\R\times\R^d$.

\begin{definition}
For a function $u : \R^d \to [-1,1]$, an open set $A'$ in $\R^d$, a $d-1$-dimensional Lipschitz surface $\Sigma'$ in $\R^d$, and a $v : \Sigma \to [-1,1]$ define the spatial {\it $\lambda$-trace error}
\[ \textup{tr}_\lambda(u,v,A',\Sigma') := \int_{\Sigma'} \int_{z + \lambda\,  h \in A'}\tilde{J}(h)|u(z+\lambda\, h) - v(z)| dh\, d\mathcal{H}^{d-1}.\]
For a function $u : \R \times \R^d \to [-1,1]$, an open set $A$ in $\R^{d+1}$, a $d$-dimensional Lipschitz surface $\Sigma$ in $\R^{d+1}$  with normal vector field $\nu$, and a $v : \Sigma \to [-1,1]$ define the $\lambda$-trace error
\[ \textup{Tr}_\lambda(u,v,A,\Sigma) := \int_0^T \textup{tr}_\lambda(u,v,A_\tau,\Sigma_\tau) d\tau =\int_{\Sigma} \int_{z + \lambda\, h \in A_\tau}\tilde{J}(h)|u(z+\lambda\, h) - v(z)| dh |\nu_x|d\mathcal{H}^{d},\]
\end{definition}
where $A_\tau=\{z: (\tau,z) \in A\}$.
Note that finite energy fields do not necessarily have a true trace on space-like $d$-dimensional surfaces, the nonlocal energy only gives large scale not micro-scale regularity.  The notion of trace error circumvents this technical difficulty.

This leads to a notion of convergence of traces imitating \cite{alberti1998non}[Definition 2.1].  Note that we need to add some additional terms to our notion of trace convergence to deal with the temporal part of the energy.  First we define the distance to a space-time surface $\Sigma$ in the pure temporal variable
\begin{equation}\label{temporal} 
\mathfrak{t}((\tau,z),\Sigma) = \inf \{ |\sigma| : (\tau+\sigma,z) \in \Sigma\}.
\end{equation}

\begin{definition}\label{d.lambdatrace}
We say that the $\lambda$-traces on $\Sigma$ of a sequence $u^\lambda : A \to [-\mathfrak{s},\mathfrak{s}]$ converge to $v: \Sigma \to [-\mathfrak{s},\mathfrak{s}]$ if
\[\lim_{\lambda \to^+ 0}\left[\int_{\Sigma} |u^\lambda - v||\nu_t| d\mathcal{H}^{d} + \int_{A\cap\{\mathfrak{t}((\tau,z),\Sigma) \leq \lambda\}}\frac{1}{\lambda}\Psi_0\big(\lambda\, \partial_\tau u^\lambda(\tau,z)\big) \ d\tau\, dz +\textup{Tr}_\lambda(u^\lambda,v,A,\Sigma)\right] = 0\]
where we note that the trace of $u^\lambda$ on $\Sigma$ is defined $|\nu_t|\mathcal{H}^d$-almost everywhere on $\Sigma$ due to the superlinear energy bound on $\partial_\tau u$.  Recall from \pref{W_decomposition} that $\Psi_0(V) = \Psi(0,V)$.
\end{definition}

\begin{remark}
Note that the energy bound $\sup_\lambda \mathcal{G}^\lambda(u^\lambda;A)<+\infty$ morally is a uniform $BV$-norm bound and is not enough to show that the traces of $u^\lambda$ on a hypersurface $\Sigma$ lie in a strongly compact subset of $L^1(\Sigma,\mathcal{H}^{d})$.  Also, as mentioned earlier, the energy bound is not sufficient regularity to give a notion of trace for $u^\lambda$ on parts of $\Sigma$ where $|\nu_t| = 0$.

Thus, as remarked in \cite{alberti1998non}, the convergence of $\lambda$-traces is not so easy to verify for a specific surface.  On the other hand, if we take a foliation by Lipschitz hypersurfaces we can get convergence of the $\lambda$-traces on almost every surface in the foliation as we show in \lref{tracesonaesurface}.  
\end{remark}
\begin{lemma}\label{l.tracesonaesurface}
Suppose that $A$, $\Sigma$, and $u^\lambda$ are as in \dref{lambdatrace} and
\[ \sup_{\lambda>0}\int_{A}\frac{1}{\lambda}\Psi_0(\lambda\, \partial_\tau u^\lambda) \ d\tau\, dz < +\infty.\]
Let $g: A \to \R$ be a Lipschitz function with $|\grad_{\tau,z} g| = 1$ a.e. and $\Sigma_a$ be the $a$-level set of $g$.  Suppose that $u^\lambda \to u_0$ in $L^1(A)$. Then, along a subsequence, the $\lambda$ traces of $u^\lambda$ on $\Sigma_a$ relative to $A$ converge to $u_0$ for a.e. $a \in \R$.
\end{lemma}
\begin{proof}
This proof is a slight generalization of \cite{alberti1998non}[Proposition 2.5]. We may suppose that $u^\lambda$ and $u_0$ are defined on $\R^{d+1}$, extended by $0$ away from $A$.  Define
\begin{equation}\label{phi_def}
 \phi_\lambda(\tau,z) := \int_{\R^d}\tilde{J}(h)|u^\lambda(\tau,z+\lambda h) - u_0(\tau,z)| dh+\left\{ \int_{-\lambda}^\lambda\frac{1}{\lambda}\Psi_0(\lambda \partial_\tau u^\lambda(\tau+\sigma,z)) {\bf 1}_{(\tau+\sigma,z) \in A}\ ds+|u^\lambda - u_0|\right\}|\partial_\tau\, g|  
 \end{equation}
and
\[ Q_\lambda(a) := \int_{\Sigma_a} \phi_{\lambda}(\tau,z) \ d\mathcal{H}^d, \qquad M := \sup_{\lambda>0}\int_{A}\frac{1}{\lambda}\Psi_0(\lambda\, \partial_\tau u^\lambda) \ d\tau\, dz. \] By co-area formula
\begin{align*}
\int_{\R} Q_\lambda(a) da &= \int_{A} \phi_\lambda(\tau,z)|\grad g| d\tau\, dz\\
&\leq \int_{\R^{d+1}} \int_{ \R^d}\tilde{J}(h)|u^\lambda(\tau,z+\lambda\, h) - u_0(\tau,z)| dh\, d\tau\, dz\\
&\quad \quad \quad+\int_{-\lambda}^\lambda\int_A\frac{1}{\lambda}\Psi_0(\lambda\, \partial_\tau u^\lambda(\tau+\sigma,z)) {\bf 1}_{(\tau+\sigma,z) \in A}\ d\tau\, dz\, d\sigma+\|u^\lambda -u_0\|_{L^1}\\
& \leq \int_{\R^{d+1}} \int_{ \R^d}\tilde{J}(h)|u^\lambda(\tau,z+\lambda\, h) - u^\lambda(\tau,z)| dh\, d\tau\, dz+\int_{\R^{d+1}} \int_{ \R^d}\tilde{J}(h)|u^\lambda(\tau,z) - u_0(\tau,z)| dh\, d\tau\, dz\\
&\quad \quad \quad +  2\lambda M + \|u^\lambda -u_0\|_{L^1}\\
&\leq \int_{\R^d}\tilde{J}(h)\|u_0(\cdot+\lambda\, h) - u_0\|_{L^1} \, dh +2M\lambda+C\|u^\lambda - u_0\|_{L^1}.
\end{align*}
Each term on the right converges to zero as $\lambda \to 0$.  Note that $\|u_0(\cdot+\lambda h) - u_0\|_{L^1(A)} \to 0$ for each fixed $h$ and the integrand is dominated by $2\|u_0\|_{L^1(A)}\tilde{J}(h)$.

Since $Q_\lambda(a) \geq 0$ it converges to zero in $L^1$ and so, up to a subsequence, it converges to zero pointwise a.e. $a \in \R$.
\end{proof}

We will also use another criterion for trace convergence, modified from \cite{alberti1998non}:
\begin{lemma}\label{l.trace-conv-criterion}
Consider $u^\lambda : A \to [-\mathfrak{s},\mathfrak{s}]$ and $v: \Sigma \to [-\mathfrak{s},\mathfrak{s}]$. If for almost every $(t,x) \in \Sigma$ the sequence $u^\lambda(\tau,z+\lambda\, h)$ converges to $v(\tau,z)$ locally uniformly on $h \in \lambda^{-1}(A_\tau - z)$, i.e.
\[ \lim_{\lambda \to 0}\sup_{h \in K \cap \lambda^{-1}(A_\tau - z)} |u^\lambda(\tau,z+\lambda\, h) - v(\tau,z)| = 0 \ \hbox{ for all $K \subset \R^d$ compact}\]
then
\[ \lim_{\lambda \to 0}\left[\int_{\Sigma} |u^\lambda - v||\nu_t| d\mathcal{H}^{d}+ \textup{Tr}_\lambda(u^\lambda,v,A,\Sigma)\right] = 0.\]
\end{lemma}

Now we move forward to prove a bound on the patching error in terms of the tracial quantities we have defined. First, we recall a definition from \cite{alberti1998non} which was used for the nonlocal energy control.
\begin{definition}
We say $\Sigma$ {\it strongly divides} $A_-$ and $A_+$ if $\Sigma$ is the Lipschitz boundary of some set $\Omega$ with $A_+ \subset\Omega$ and $A_- \subset \R^{d+1} \setminus \Omega$.
\end{definition}
We recall a localization Lemma of \cite{alberti1998non}.

\begin{lemma}\label{lem:localization}
    Suppose that $A_\pm$ are disjoint subsets of $\R^{d+1}$ and are strongly divided by $\Sigma$.  Suppose further that $u^\lambda:A_+\cup A_-\rightarrow (-1,1)$ and 
    \[\lim_{\lambda \to 0}\textup{Tr}_\lambda(u,v_\pm,A_\pm,\Sigma) = 0.\]
    Then the discrepancy cost $\mathcal{N}^\lambda$ defined in \eqref{eqn:space_time_energies}  satisfies  $$
        \limsup_{\lambda\rightarrow0+} \mathcal{N}^\lambda(u^\lambda;A_+,A_-)\leq \frac{\hat{J}}{2} \int_{\Sigma}|v_+-v_-||\nu_x|d\mathcal{H}^{d}.
    $$
\end{lemma}
\begin{proof}
    In \cite{alberti1998non}, this is proved for each time-slice, and the result is obtained simply by integrating in time and using co-area formula.  
\end{proof}

The next result shows how to patch test minimizers across a regular (Lipschitz) boundary. As in \cite{alberti1998non} patching creates extra nonlocal energy due to the nonlocal defect.  However now we also have a local term in the energy $\lambda^{-1}\Psi(u,\lambda\, \partial_\tau u)$ which grows superlinearly in the $\partial_\tau u$ variable.  This means that we cannot simply patch discontinuously, we need to make a regularization at the $\lambda$ length scale across the patching boundary. The next proposition shows that such a regularization can be made, at an additional energy cost which is controlled by a trace error of the type introduced in \dref{lambdatrace}. This result addresses the temporal Dirichlet type energy which is not present in \cite{alberti1998non}.  

\begin{proposition}[Defect estimate]\label{p.patchingestimate}
Let $A$ be an open set, ${\Sigma}$ be a finite union of subsets of affine $d$-dimensional planes with a normal direction $\nu \in S^d$ defined $\mathcal{H}^d|_\Sigma$-a.e., and $u : A \to [-\mathfrak{s},\mathfrak{s}]$ with $\mathcal{G}^\lambda(u;A) < +\infty$.  Let $u_\pm$ be the respective traces of $u$ on $\Sigma \cap \{|\nu_t| > 0\}$.

Then there is $\tilde{u}: (A \cup \Sigma)^o \to [-\mathfrak{s},\mathfrak{s}]$ such that $\tilde{u} = u$ outside of a $\lambda$ neighborhood of $\Sigma$ and for every $\delta>0$ and for any subregion $B \subset A$
\begin{align}
& \mathcal{G}^{\lambda}_{loc}(\tilde{u};(A \cup \Sigma)^o) -\mathcal{G}^\lambda_{loc}(u;A)+|\mathcal{N}^\lambda(\tilde{u};B,B)-\mathcal{N}^\lambda(u;B,B)|\notag\\
&\quad \quad \leq  C\int_{{\Sigma}} \left\{ |u_+ - u_-|+\lambda+\delta\right\} |\nu_t| d\mathcal{H}^d  +  C\delta^{-1}\left[\frac{1}{\lambda}\int_{A\cap\{0<\mathfrak{t}((\tau,z),\Sigma) \leq \lambda\}}\Psi_0(\lambda\, \partial_\tau u) \ d\tau\, dz. \right]\label{e.defectRHS}
\end{align}
Here $\mathcal{G}^\lambda_{loc}$ is defined in  \eqref{eqn:space_time_energies} and the constants $C$ depend on $\Sigma$.
\end{proposition}

\begin{proof} 

We will first assume that $\Sigma$ is a subset of a single affine $d$-plane, at the end of the proof we will explain how to extend to the general case of a finite union of affine pieces.  

\medskip

In the single plane case there is a single normal direction $\nu$ constant on $\Sigma$. Note that if $\nu_t = 0$ no argument is needed, simply take $\tilde{u} = u$ so we can assume $\nu_t \neq 0$. We may further assume that $0 \in \Sigma$. Define 
$$
A_\pm := A \cap \{\pm \nu \cdot (\tau,z) >0\} \hbox{ and } r_*(\tau,z):= {\rm argmin}_{\sigma}\{|\tau-\sigma|:(\sigma,z)\in \Sigma\}. 
$$
If $z$ is not in the projection of $\Sigma$ onto $\T^d$ then we define $r_*(\tau,z):=\mp\infty$ in $A_\pm$. So we have $\pm(\tau-r_*(\tau,z))\in [0,+\infty]$ in $A_\pm$ respectively.

Also, in the typical style of a priori estimates, we can assume that $u$ is $C^1$ individually in $\overline{A_\pm}$ (but not their union) so that the computations below are justified, but then the estimate obtained will not depend on the $C^1$ norm so we can remove that assumption in the end.

\medskip

Now we proceed in several steps. 

{\bf Step 1:} First we introduce $\tilde{u}$ which essentially averages $u$ in a temporal neighborhood of $\Sigma$ of size $O(\lambda)$.  This is the exact scale at which we must perform the regularization, smaller scales would have too large temporal ``Dirichlet" energy and larger scales would magnify the energy of transitions from $-\mathfrak{s}$ to $\mathfrak{s}$ too much.  The energy error of the regularization will be related to the trace difference which needs to be traversed over the $\lambda$ scale.

 Let $\zeta$ be a cut-off function that satisfies
\[ \zeta(\tau,z) = \begin{cases}1 & |\tau-r_*(\tau,z)| \leq \lambda /4\\
0 & |\tau-r_*(\tau,z)| \geq \lambda/2\end{cases} \ \hbox{ and } \ |\partial_\tau\zeta| \leq C\lambda^{-1}.\]  
 Define
\[ \tilde{u}:= \zeta \hat{u} + (1-\zeta)u \ \hbox{ with } \ \hat{u}(\tau,z) = \frac{1}{\lambda}\int_{-\lambda/2}^{\lambda/2} u(\tau+\sigma,z) \ d\sigma.\]

We make a few computations relating $\hat{u} - u$ and $\partial_\tau\hat{u}$ to the traces on $\Sigma$.  When $\zeta(\tau,z)>0$ then $r_*(\tau,z) \in [\tau-\lambda/2,\tau+\lambda/2]$.  We use this to write, on $\{\zeta>0\}$,
\[ u(\tau,z) = u^\pm(r_*(\tau,z),z) + \int_{r_*(\tau,z)}^\tau \partial_\tau u(\sigma,z) \ d\sigma \ \hbox{ if } \ (\tau,z) \in A_\pm.\]
Then we can use this decomposition in $\hat{u}$ as well.  By definition we have
\begin{equation}\label{11}
\hat{u}(\tau,z) = \mu(\tau,z) u^+(r_*(\tau,z),z)+(1-\mu(\tau,z)) u^-(r_*(\tau,z),z)+ \lambda K(\tau,z),
\end{equation}
where $\mu(\tau,z) \in (0,1)$ is defined as the fraction of $\sigma\in [-\lambda/2,\lambda/2]$ so that $(\tau+\sigma ,z) \in A_+$ and
\[ K(\tau,z) = \frac{1}{\lambda^2} \int_{-\lambda/2}^{r_*(\tau,z)} \int_{\tau+\sigma }^{r_*(\tau,z)} \partial_\tau u(k,z) \ dk\, d\sigma+ \frac{1}{\lambda^2}\int_{r_*(\tau,z)}^{\lambda/2}\int_{r_*(\tau,z)}^{\tau+\sigma}\partial_\tau u(k,z) dk\, d\sigma.\]

The appearance of this type of error term motivates the following definition on $A_\pm \cap \{\zeta>0\}$, 
\[ \textup{avg}_\pm(|\partial_\tau u|)(\tau,z) := \frac{1}{\lambda}\int_{[r_*(\tau,z)\pm\lambda,r_*(\tau,z)]} |\partial_\tau u(\sigma, z)| \, d\sigma+\frac{2}{\lambda^2} \int_{[r_*(\tau,z)\pm\lambda,r_*(\tau,z)]}\int_{[\sigma,r_*(\tau,z)]} |\partial_\tau u(k,z)| \ dk\, d\sigma \]
Note that $\textup{avg}_\pm(|\partial_\tau u|)$ are, respectively, integral averages of $\partial_\tau u$ purely on $A_\pm$ respectively, they do not see the discontinuity across $\Sigma$. Recall that we have reduced, for convenience, to the case where $\Sigma$ is a graph over the $t$ direction and $\pm(\tau-r_*(\tau,z))>0$ on $A_{\pm}$.

One particular consequence of these computations is that on $\{\zeta>0\}$,
\begin{equation}\label{e.hatuudiff}
    |\hat{u}(\tau,z) - u(\tau,z)| \leq |u^+(r_*(\tau,z),z) - u^-(r_*(\tau,z),\tau)| + \lambda \sum_{\pm}\textup{avg}_\pm(|\partial_\tau u|)(\tau,z).
\end{equation} 
We can also make a similar decomposition for $\partial_\tau \hat{u}$. Note
\[ \partial_\tau \hat{u} = \frac{1}{\lambda}[u(\tau+\tfrac{\lambda}{2},z) - u(\tau-\tfrac{\lambda}{2},z)].\]
When $\zeta(\tau,z)>0$ then $r_*(\tau,z) \in [\tau-\lambda/2,t+\lambda/2]$ so we can write
\begin{align*}
 \frac{1}{\lambda}|u(\tau+\tfrac{\lambda}{2},z) - u(\tau-\tfrac{\lambda}{2},z)| &= \frac{1}{\lambda}|u^+(r_*(\tau,z),z) - u^-(r_*(\tau,z), z)|\\
 &\quad \quad \quad +\frac{1}{\lambda}\left|\left\{\int_{r_*(\tau,z)-\lambda/2}^{r_*(\tau,z)} + \int_{r_*(\tau,z)}^{r_*(\tau,z)+\lambda/2}\right\} \partial_\tau u(\sigma, z) \ d\sigma\right|\\
 &\leq \frac{1}{\lambda}|u^+(r_*(\tau,z),z) - u^-(r_*(\tau,z),z)| +  \sum_{\pm} \textup{avg}_\pm(|\partial_\tau u|)(\tau,z).
 \end{align*}
 
 {\bf Step 2.}  Next we make a general comment on integrals of the type
 \[ \frac{1}{\lambda}\int_{\{\zeta>0\}} h(r_*(\tau,z),z) \ d\tau\, dz\]
 for a function $h \in L^1(\Sigma,d\mathcal{H}^d|_{\Sigma})$.  By co-area formula
 \[ \frac{1}{\lambda}\int_{\{\zeta>0\}} h(r_*(\tau,z),z) \ d\tau\, dz = \frac{1}{\lambda}\int_{-\lambda/2}^{\lambda/2}\int_{\Sigma + (\tau,0)}h(r_*(\tau,z),z)|Ds_*|^{-1} d\mathcal{H}^{d} d\tau.\]
 Note that, since $\nu_t \neq 0$,  $|\partial_\tau r_*(\tau,z)| = 1$  and $|D_xr_*(\tau,z)| = \frac{|\nu_x|}{|\nu_t|} (r_*(\tau,z),z)$, so $|D_{t,x}r_*(\tau,z)| = |\nu_t|(r_*(\tau,z),z)^{-1}$.  Thus we obtain the change of variables formula
 \begin{equation}\label{e.covnut} \frac{1}{\lambda}\int_{\{\zeta>0\}} h(r_*(\tau,z),z) \, d\tau\, dz = \int_{\Sigma} h |\nu_t| d\mathcal{H}^d.
 \end{equation}
 
{\bf Step 3.} Using above formula, here we will see that error terms of type $\textup{avg}_\pm(|\partial_\tau u|)$ can be controlled by the energy in a $\lambda$ - temporal neighborhood of $\Sigma$.

The following formulae will be applied below with $f = |\partial_\tau u|$ or other related functions in later steps.  We use the change of variables formula, applied to $h(\tau,x):= f(\tau+\sigma,z)$, twice to compute
\begin{align*}\frac{1}{\lambda}\int_{\{\zeta>0\}}  \lambda \frac{1}{\lambda}\int_{[\pm\lambda,0]} |f(r_*(\tau,z)+\sigma,z)| \ d\sigma\, d\tau\, dz  &= \int_{\Sigma}\int_{[\pm\lambda,0]} |f(\tau+\sigma,z)| \, d\sigma |\nu_t| d\mathcal{H}^d(\tau,z) \\
&=\int_{\{0 < \pm (\tau - r_*(\tau,z)) < \lambda\}}|f(\tau,z)|d\tau\, dz.\end{align*}
Similarly,
\begin{align*}
&\frac{1}{\lambda}\int_{\{\zeta>0\}}  \lambda \frac{2}{\lambda^2} \int_{[\pm\lambda,0]}\int_{[\sigma,0]} |f(r_*(\tau,z)+k,z)| \ dk\, d\sigma\, d\tau\, dz  \\
\quad \quad &= \int_{\Sigma}\frac{2}{\lambda}\int_{[\pm\lambda,0]}\int_{[\sigma,0]} |f(\tau+k,z)| \, dk\, d\sigma |\nu_t| d\mathcal{H}^d(\tau,z) \\
&=\frac{2}{\lambda}\int_{[\pm\lambda,0]}\int_{[\sigma,0]}\int_{\Sigma} |f(\tau+k,z)| |\nu_t(\tau,z)| d\mathcal{H}^d(\tau,z)dk\, d\sigma\\
&=\frac{2}{\lambda}\int_{[\pm\lambda,0]}\int_{[\sigma,0]}\int_{\Sigma} |f(\tau+k,z)| |\nu_t(\tau,z)| d\mathcal{H}^d(\tau,z)dk\, d\sigma\\
&=\frac{2}{\lambda}\int_{[\pm\lambda,0]}\int_{\{0 < \pm (\tau - r_*(\tau,z)) < \sigma\}}|f(\tau,z)|d\tau\, dz\, d\sigma\\
&\leq 2\int_{\{0 < \pm (\tau - r_*(\tau,z)) < \lambda\}}|f(\tau,z)|d\tau\, dz.
\end{align*}
Combining the above we find
\begin{equation}\label{111}
    \frac{1}{\lambda}\int_{\{\zeta>0\}}  \lambda \textup{avg}_\pm(|f|) d\tau\, dz \leq 3\int_{\{0 < \pm (\tau - r_*(\tau,z)) < \lambda\}}|f(\tau,z)|d\tau\, dz.
\end{equation}
Now, since we will often take $f = |\partial_\tau u|$ or similar below, we need to explain how to estimate the right hand side in \eqref{111} with that choice of $f$.  For the below we use the version of Young's inequality $|V| \leq \delta + C\delta^{-1}\Psi_0(V)$ for arbitrary $1>\delta>0$:
 \begin{align}
     \int_{\{0 < \pm (\tau - r_*(\tau,z)) < \lambda\}}|\partial_\tau u(\tau,z)|d\tau\, dz &=\lambda^{-1}\int_{\{0 < \pm (\tau - r_*(\tau,z)) < \lambda\}}\lambda|\partial_\tau u(\tau,z)|d\tau\, dz\notag\\
     &\leq \int_{\{0 < \pm (\tau - r_*(\tau,z)) < \lambda\}}\delta\lambda^{-1}+C\delta^{-1}\lambda^{-1}\Psi_0(\lambda\, \partial_\tau u))d\tau\, dz \notag\\
 &\leq  \int_{\Sigma} \delta |\nu_t| d\mathcal{H}^{d} + C\delta^{-1}\frac{1}{\lambda}\int_{A_\pm\cap\{\mathfrak{t}((\tau,z),\Sigma) \leq \lambda\}}\Psi_0(\lambda\, \partial_\tau u) \ d\tau\, dz \label{e.dtuL1est}
 \end{align} 
 where we used \eref{covnut} at the last step with $h \equiv 1$.

 \medskip
 
{\bf Step 4.}  
In this step, we work to estimate the local terms in the energy, with the first focus on the ``Dirichlet" type term. We aim to estimate from above the difference 
\[ \int_{A_+ \cup A_-} \lambda^{-1}\Psi(\tilde{u},\lambda\, \partial_\tau \tilde{u})d\tau\, dz - \sum_{\pm}\int_{A_\pm}\lambda^{-1}\Psi(u,\lambda\, \partial_\tau u)d\tau\, dz. \]
We define
 \[ v := \zeta\, \partial_\tau\hat{u} + (1-\zeta)\partial_\tau u \ \hbox{ so that } \ \partial_\tau \tilde{u} = v + \partial_\tau \zeta (\hat{u} - u).\]
As part of estimating the previous energy difference we will estimate
\[\Psi(\tilde{u},\lambda \partial_\tau \tilde{u}) - \Psi(\tilde{u},\lambda v).\]  
Using abstract variables $a = \lambda v$ and $h = \lambda\, \partial_\tau \zeta (\hat{u} - u)$ and dropping the $\tilde{u}$ dependence because it is the same in each term
\begin{align*}
    \Psi(a+h) - \Psi(a) & = \int_a^{a+h} \Psi'(k) dk \\
    &\leq (|\Psi'(a)| + |\Psi'(a+h)|)|h| \\
    &\leq 2\Psi'(|a| + |h|)|h|\\
    &\leq 2\Psi'(|a|)|h| + 2\Psi'(|h|)|h|\\
    &\leq \Psi'(|a|)^2{\bf 1}_{\{|h|>0\}}+|h|^2 + C \Psi(|h|) \\
    &\leq C\Psi(|a|){\bf 1}_{\{|h|>0\}} + |h|^2 + C\Psi(|h|).
\end{align*}
Where we used in order, from  Proposition \ref{prop:W_decomposition}, that $\Psi'$ is monotone increasing, odd symmetric, subadditive on $[0,\infty)$, and for the remaining inequalities we used the bounds \eqref{estimate} and \eqref{estimate2}. 

Applying this we arrive at 
\[
  \int_{A_+ \cup A_-} \frac{1}{\lambda}\Psi(\tilde{u},\lambda\, \partial_\tau \tilde{u}) \ d\tau\, dz \leq \frac{C}{\lambda}\int_{A_+ \cup A_-} \big[\Psi(\tilde{u},\lambda v) +\Psi(\tilde{u},\lambda |v|){\bf 1}_{\{|\partial_\tau \zeta|>0\}} + |\lambda\, \partial_\tau\zeta|^2|\hat{u} - u|^2 + 2\Psi(\tilde{u},\lambda|\partial_\tau \zeta||\hat{u}-u|) \big]d\tau\, dz.\]
  For the first term on the right we use non-negativity of $\Psi$
\[ \int_{A_+ \cup A_-} \frac{1}{\lambda}\Psi(\tilde{u},\lambda\, v) \ d\tau\, dz  \leq \int_{A_+ \cup A_-} \frac{1}{\lambda}\Psi(u,\lambda\, \partial_\tau u)d\tau\, dz+\int_{\{ \zeta >0\}}\frac{1}{\lambda}  \Psi(\tilde{u},\lambda v) d\tau\, dz.\]
  Since we can bound $\Psi(\tilde{u},\cdot) \leq C\Psi_0(\cdot)$ it remains for us to bound the error terms (using even symmetry of $\Psi$)
  \[ (I):= \int_{\{ \zeta >0\}}\frac{1}{\lambda}  \Psi_0(\lambda v) d\tau\, dz, \ (II) := \frac{1}{\lambda}\int_{\{\zeta>0\}} |\lambda\, \partial_\tau\zeta|^2|\hat{u} - u|^2 \ d\tau\, dz, \ \hbox{ and } \ (III) := \frac{1}{\lambda}\int_{\{\zeta>0\}}\Psi_0(\lambda|\partial_\tau \zeta||\hat{u}-u|) d\tau\, dz. \]
 Note that $|\lambda\, \partial_\tau\zeta| \leq C$ and $|\Psi_0(P)| \leq C|P|^2$ so
\[ (II),(III) \leq C\frac{1}{\lambda}\int_{\{\zeta>0\}} |\hat{u}-u|^2 \ d\tau\, dz \leq C\frac{1}{\lambda}\int_{\{\zeta>0\}} |\hat{u}-u| \ d\tau\, dz.\]
At the last step we used $|\hat{u} - u| \leq 2$ to bound the $L^2$ norm by $L^1$. Then, applying \eqref{e.hatuudiff} and \eqref{111} this becomes
\begin{align*}
    \frac{1}{\lambda}\int_{\{\zeta>0\}} |\hat{u}-u| \ d\tau\, dz &\leq \frac{1}{\lambda}\int_{\{\zeta>0\}}|u^+(r_*(\tau\, z),z) - u^-(r_*(\tau, z),z)| + \lambda \sum_{\pm}\textup{avg}_\pm(|\partial_\tau u|)(\tau,z) d\tau\, dz\\
    & = \int_{\Sigma} |u^+ - u^-| d\mathcal{H}^d + 3\int_{\{0 < \pm (\tau - r_*(\tau,z)) < \lambda\}}|\partial_\tau u(\tau,z)|d\tau\, dz
\end{align*} 
The second term can be estimated by the discussion in Step 3, in particular \eref{dtuL1est}.

It remains to discuss the error term $(I)$, relying on the convexity of $\Psi_0$. Observe first that
\[ (I) = \int_{\{ \zeta >0\}}\frac{1}{\lambda}  \Psi_0(\lambda v) d\tau\, dz \leq \frac{C}{\lambda}\int_{\{\zeta>0\}} \zeta\Psi_0(\lambda\, \partial_\tau\hat{u}) + (1-\zeta)\Psi_0(\lambda\, \partial_\tau u) \ d\tau\, dz.\]
The second term is already one of the claimed error terms in the statement.  The first term is bounded by using formula \eqref{e.hatuudiff} to relate with the traces on $\Sigma$:
\begin{align*}
\int_{A_+ \cup A_-} \frac{1}{\lambda}\zeta\Psi_0(\lambda\, \partial_\tau \hat{u}) \ d\tau\, dz &\leq C\int_{\{\zeta >0\}} \frac{1}{\lambda}\Psi_0(\lambda\, \partial_\tau \hat{u}) \ d\tau\, dz \\
& \leq C\frac{1}{\lambda}\int_{\{\zeta >0\}} \Psi_0(|u^+(r_*(\tau,z),z) - u^-(r_*(\tau,z),z)|)+C\Psi_0(\lambda\sum_{\pm}\textup{avg}_\pm(|\partial_\tau u|)) \ d\tau\, dz.
\end{align*}
The last line using that $\Psi_0(A+B)\leq C(\Psi_0(A)+\Psi_0(B))$.
It follows by again convexity of $\Psi_0$ and Jensen's inequality
\[\Psi_0(\lambda\textup{avg}_\pm(|\partial_\tau u|)) \leq \textup{avg}_\pm(\Psi_0(\lambda|\partial_\tau u|)).\]
Now \eqref{111} yields
\[\frac{1}{\lambda}\int_{\{\zeta >0\}}\Psi_0(\lambda\sum_{\pm}\textup{avg}_\pm(|\partial_\tau u|)) d\tau\, dz \leq C\sum_{\pm}\int_{\{0 < \pm (\tau - r_*(\tau,z)) < \lambda\}}\lambda^{-1}\Psi_0(\lambda |\partial_\tau u|)d\tau\, dz.\]
This type of term appears on the right hand side of the claimed estimate so we are done estimating term $(I)$.

\medskip

Next we deal with the double-well term
\[   \int_{A_+ \cup A_-} \frac{1}{\lambda}\mathcal{W}_\beta(\tilde{u}) \ d\tau\, dz -\int_{A_+\cup A_-} \frac{1}{\lambda}\mathcal{W}_\beta(u)\ d\tau\, dz=   \int_{\{\zeta >0\}} \frac{1}{\lambda}(\mathcal{W}_\beta(\tilde{u}) - W_\beta(u)) \ d\tau\, dz.\]
So we need to deal with this term on the right
\[ W_\beta(\tilde{u}) = W_\beta ( u + \zeta (\hat{u}-u)).\]
Applying \eref{hatuudiff} and using that $W_\beta$ is Lipschitz on $[-\mathfrak{s},\mathfrak{s}]$
\[ |W_\beta(\tilde{u}) -  W_\beta ( u)| \leq C|u_+(r_*(\tau,z),z) - u_-(r_*(\tau,z),z)| + C\lambda \sum_{\pm}\textup{avg}_\pm(|\partial_\tau u|) \ \hbox{ on } \{\zeta >0\} .\]
From there the estimate is the same as in Step 3.

{\bf Step 5.} 
We still need to bound $|\mathcal{N}^\lambda(\tilde{u},B,B)-\mathcal{N}^\lambda(u,B,B)|$. For that we write $\tilde{u} = u + \zeta (\hat{u} - u)$ and bound
\[ |\mathcal{N}^\lambda(\tilde{u},B,B)-\mathcal{N}^\lambda(u,B,B)| \leq \frac{C}{\lambda}\int_{B} \zeta(\tau,z)|\hat{u}(\tau,z) - u(\tau,z)| \ d\tau\, dz\]
This type of error term was already bounded in step 4 above.
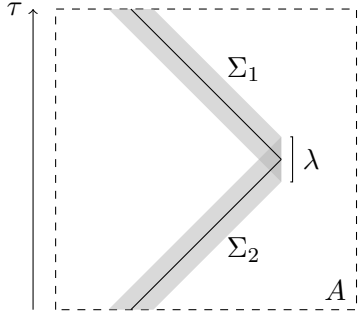
\begin{figure}
    \centering
   \begin{tikzpicture}[scale = 1]
\draw[dashed] (-2,-2) -- (-2,2) -- (2,2) -- (2,-2) node[anchor = south east] {$A$} -- (-2,-2);
\def\l{.3}
\path[fill=gray!60,semitransparent] (-1-\l,2) -- (-1+\l,2)-- (1,\l) -- (1,-\l) -- (-1-\l,2);
\path[fill=gray!60,semitransparent] (-1-\l,-2) -- (-1+\l,-2)-- (1,-\l) -- (1,\l) -- (-1-\l,-2);
\node (A) at (.5,1.2) {$\Sigma_1$};
\node (B) at (.5,-1.2) {$\Sigma_2$};
\draw (-1,2) -- (1,0) ;
\draw (-1,-2) --(1,0);
\draw[->] (-2-\l,-2) -- (-2-\l,2) node[anchor = east] {$\tau$};
\draw[thin] (1+.4*\l,\l)--(1+.5*\l,\l)--(1+.5*\l,0) node[anchor = west]{$\lambda$} -- (1+.5*\l,-\l)--(1+.4*\l,-\l);
\end{tikzpicture}
    \caption{Open region $A$ intersected by defect surface $\Sigma$ made up of two affine pieces $\Sigma_j$, $\lambda$-time neighborhoods and their overlap are displayed.}
    \label{f.defect-estimate}
\end{figure}

{\bf Step 6.} Finally we consider the case when $\Sigma$ is a finite union of pieces $\Sigma_1,\dots,\Sigma_J$ which are each contained in affine planes $P_j$, see \fref{defect-estimate}.  We can assume that the planes $P_j$ are all distinct, otherwise, the corresponding sets could be regrouped with a smaller $J$.  For $\lambda>0$ sufficiently small the $\lambda$-neighborhoods of any two parallel planes of $\{P_j\}$ will be disjoint.  Any two non-parallel $P_j$ meet, at most, on a set of Hausdorff dimension $d-1$ and we can bound, using the compactness of the region $[0,T]\times\T^d$, 
\[ \mathcal{H}^d(\Sigma_j^\lambda \cap \Sigma_k) \leq C \lambda \ \hbox{ for any } \ j \neq k.\]
With this in mind, we proceed inductively and assume we have constructed $\tilde{u}_j$ satisfying the conclusion of the theorem on $(A \cup \Sigma_1 \cup \cdots \cup \Sigma_j)^o$.  Then we apply the single plane case to define $\tilde{u}_{j+1}$ by the mollification of $\tilde{u}_j$.  Notice that the traces $(\tilde{u}_j)_\pm$ on $\Sigma_{j+1}$ only differ from the traces of $u$ on the intersections of $\Sigma_k^\lambda$ for $1 \leq k \leq j$ with $\Sigma_{j+1}$ and by the previous argument these intersections have $\mathcal{H}^d$ measure bounded by $C\lambda$ so
\[\int_{{\Sigma_{j+1}}}  |(\tilde{u}_j)_+ - (\tilde{u}_j)_-|d \mathcal{H}^d \leq \int_{{\Sigma_{j+1}}} ( |u_+ - u_-|+ C\lambda)d \mathcal{H}^d. \]
Furthermore the energy
\[ \frac{1}{\lambda}\int_{A\cap\{0<\mathfrak{t}((\tau,z),\Sigma_{j+1}) \leq \lambda\}}\Psi_0(\lambda\, \partial_\tau \tilde{u}_j) \ d\tau\, dz  \leq \mathcal{G}^{\lambda}_{loc}(\tilde{u}_j;(A \cup \Sigma\cup\cdots\cup\Sigma_j)^o) -\mathcal{G}^\lambda_{loc}(u;A) \]
which we have already assumed, in the inductive hypothesis, to be bounded by the right-hand side of \eref{defectRHS}.
\end{proof}

We now state several useful consequences of \pref{patchingestimate}, restated in terms that are more directly useful for section 4.4.

The primary use of \pref{patchingestimate} is to patch together two solutions that agree in trace along the dividing boundary.  We state as a corollary that this can be done without introducing extra cost in the limit as $\lambda\rightarrow 0$.  We use the notation $A_1\sqcup A_2$ to denote the interior of the closure of $A_1\cup A_2$.
\begin{corollary}\label{cor:patching}
Suppose that a surface $\Sigma$, which is a finite union of pieces of $d$-dimensional affine planes, strongly divides  a pair of sets $A_1$ and $A_2$, and $u_1^\lambda: A_1 \rightarrow [-\mathfrak{s},\mathfrak{s}]$ and $u_2^\lambda: A_2 \rightarrow [-\mathfrak{s},\mathfrak{s}]$ with $\sup_\lambda \mathcal{G}(u_1^\lambda;A_1) +\sup_\lambda \mathcal{G}(u_2^\lambda;A_2)<\infty$.  Suppose further that the $\lambda$-traces on $\Sigma$ of $u_1^\lambda$ converge to $v_1:\Sigma\rightarrow [-\mathfrak{s},\mathfrak{s}]$ and the $\lambda$-traces on $\Sigma$ of $u_2^\lambda$ converge to $v_2:\Sigma\rightarrow [-\mathfrak{s},\mathfrak{s}]$.  Then there exists $u^\lambda:A_1\sqcup A_2\rightarrow [-\mathfrak{s},\mathfrak{s}]$ that  satisfies
$$
\limsup_{\lambda\rightarrow^+0} \Big(\|u^\lambda - u_1^\lambda\|_{L^1(A_1)}+\|u^\lambda - u_2^\lambda\|_{L^1(A_2)}\Big) = 0
$$
and
$$
    \limsup_{\lambda\rightarrow^+0}\Big(\mathcal{G}^\lambda(u^\lambda; A_1\sqcup A_2) -\mathcal{G}^\lambda(u_1^\lambda; A_1) -\mathcal{G}^\lambda(u_2^\lambda; A_2) \Big)\leq C\int_\Sigma |v_1-v_2|d\mathcal{H}^d.
$$
Furthermore, the construction of $u$ depends only locally on the values of $u_1$ and $u_2$, and $u=u_1$ or $u=u_2$ a distance greater than $\lambda$ from the boundary of $A_1$ or $A_2$.
\end{corollary}
\begin{proof}
    The proof is a direct application of \pref{patchingestimate} with the function 
    \[u =  u^\lambda_1 {\bf 1}_{A_1} + u^\lambda_2{\bf 1}_{A_2}\]
    and $u^\lambda = \tilde{u}$.  Definition \ref{d.lambdatrace} ensures that the right hand side of \eref{defectRHS} are controlled in the limit as $\lambda\rightarrow^+ 0$ by $C\int_\Sigma |v_1-v_2|d\mathcal{H}^d$ and also
    \begin{align*} \limsup_{\lambda\to0}\mathcal{N}^\lambda(\tilde{u},A_1,A_2) &= \frac{1}{2}\limsup_{\lambda \to 0}\left[\mathcal{N}^\lambda(\tilde{u},A_1\cup A_2,A_1\cup A_2) - \mathcal{N}^\lambda(\tilde{u},A_1,A_1) - \mathcal{N}^\lambda(\tilde{u},A_2,A_2)\right]\\
    & \leq \limsup_{\lambda \to 0} \mathcal{N}^\lambda(u,A_1,A_2) + C\int_\Sigma |v_1-v_2|d\mathcal{H}^d
    \end{align*}
    applying \eref{defectRHS} with $B$ respectively to be $A_1\cup A_2$, $A_1$, $A_2$. Recall from \dref{lambdatrace} and the assumption of $\lambda$-trace convergence on $\Sigma$ that the last term on the right of \eref{defectRHS} involving $\lambda^{-1}\Psi_0(\lambda \cdot)$ also converges to zero as $\lambda \to 0$.
    
    Furthermore, Lemma \ref{lem:localization} implies that the nonlocality defect $\mathcal{N}^\lambda(u,A_1,A_2)$ is bounded by $C\int_\Sigma |v_1-v_2|d\mathcal{H}^d$. The $L^1$ equivalence follows from the fact that $u_1=u$ and $u_2=u$ a distance greater than $\lambda$ from $\Sigma$.
\end{proof}

We similarly will use an adjustment of functions defined on a square to periodic functions in $\mathcal{X}_R$. We fix a space-time unit-normal vector $\nu$ and $\square\in \blacksquare_\nu$, and consider the step function
    \begin{equation}\label{step}
        v_\nu(t,x) := \begin{cases} +\mathfrak{s} & (t,x)\cdot \nu\geq 0,\\
        -\mathfrak{s} & (t,x)\cdot \nu<0.\end{cases}
    \end{equation}
Recall that we define spaces for $\square$-periodic $C^1$ functions in section 3 (\eqref{space_X}, \eqref{space_init} and \eqref{space_end}). Recall also that we consider a sequence $\lambda_i>0$ such that $\lambda_i\rightarrow 0$ as $i\rightarrow \infty$, which we denote as $\lambda\rightarrow^+0$.
\begin{proposition}\label{cor:periodization}\label{p.periodization}
   \begin{itemize}
\item[(a)]    Given $u^\lambda:\square \times \R\, \nu\rightarrow [-\mathfrak{s},\mathfrak{s}]$ satisfying the following:
    \begin{itemize}
        \item $\sup_\lambda \mathcal{G}(u^\lambda;\square \times \R\, \nu) <\infty$,
        \item the $\lambda$ traces agree with $v_\nu$ on the $\partial \square \times \R\, $ and on $\square \times \{-R,R\}\nu$,
    \end{itemize}
    then there is $\tilde{u}^\lambda\in \mathcal{X}_R(\square)$ such that 
    $$
        \limsup_{\lambda\rightarrow^+0} \Big(\|\tilde{u}^\lambda - u^\lambda\|_{L^1(\square \times \R\, \nu)}\Big) = 0
    $$
    and
    $$
        \limsup_{\lambda\rightarrow^+0}\Big(\mathcal{F}^\lambda(\tilde{u}^\lambda; \square \times \R\, \nu) -\mathcal{G}^\lambda(u^\lambda; \square \times \R\, \nu)\Big)\leq 0.
    $$
    
  \item[(b)]  If $u^\lambda:\square \times \R^+\, \nu\rightarrow [-\mathfrak{s},\mathfrak{s}]$ where $\nu_t > 0$ with
    \begin{itemize}
        \item $\sup_\lambda \mathcal{G}(u^\lambda;\square \times \R^+\, \nu) <\infty$,
        \item the $\lambda$-traces agree with $v_\nu$ on the $\partial \square \times \R^+\, $ and on $\square \times \{R\}\nu$,
        \item the $L^1$-trace converges on $\square \times \{0\}\nu$ to a constant $s_0\in (-\mathfrak{s},\mathfrak{s})$,
    \end{itemize}
    then there is $\tilde{u}^\lambda\in \mathcal{X}_R^{init}(s_0, \pm \mathfrak{s}, \square)$ such that 
    $$
        \limsup_{\lambda\rightarrow^+0} \Big(\|\tilde{u}^\lambda - u^\lambda\|_{L^1(\square \times \R^+\, \nu)}\Big) = 0
    $$
    and
    $$
        \limsup_{\lambda\rightarrow^+0}\Big(\mathcal{F}^\lambda(\tilde{u}^\lambda; \square \times \R^+\, \nu) -\mathcal{G}^\lambda(u^\lambda; \square \times \R^+\, \nu)\Big)\leq 0.
    $$
    
    \item[(c)] If $u^\lambda:\square \times \R^+\, \nu\rightarrow [-\mathfrak{s},\mathfrak{s}]$ where $\nu_t <0$ with \begin{itemize}
        \item $\sup_\lambda \mathcal{G}(u^\lambda;\square \times \R^+\, \nu) <\infty$,
        \item the $\lambda$-traces agree with $v_\nu$ on the $\partial \square \times \R^+\, $ and on $\square \times \{R\}\nu$,
    \end{itemize}
    then there is $\tilde{u}^\lambda\in \mathcal{X}_R^{end}(\pm \mathfrak{s}, \square)$ such that 
    $$
        \limsup_{\lambda\rightarrow^+0} \Big(\|\tilde{u}^\lambda - u^\lambda\|_{L^1(\square \times \R^+\, \nu)}\Big) = 0
    $$
    and
    $$
        \limsup_{\lambda\rightarrow^+0}\Big(\mathcal{F}^\lambda(\tilde{u}^\lambda; \square \times \R^+\, \nu) -\mathcal{G}^\lambda(u^\lambda; \square \times \R^+\, \nu)\Big)\leq 0.
    $$
    \end{itemize}
\end{proposition}
\begin{proof}
    Let $A_1 := \square \times \R\, \nu\rightarrow [-\mathfrak{s},\mathfrak{s}]$ and $A_2$ be the union of all the translations along one sidelength of $\square$.  Then Corollary \ref{cor:patching} constructs $\tilde{u}^\lambda$ on $A_1 \sqcup A_2$ with
    $$
\limsup_{\lambda\rightarrow^+0} \Big(\|\tilde{u}^\lambda - u^\lambda\|_{L^1(A_1)}+\|\tilde{u}^\lambda - u^\lambda\|_{L^1(A_2)}\Big) = 0
$$
and
$$
    \limsup_{\lambda\rightarrow^+0}\Big(\mathcal{G}^\lambda(\tilde{u}^\lambda; A_1\sqcup A_2) -\mathcal{G}^\lambda(u^\lambda; A_1) -\mathcal{G}^\lambda(u^\lambda; A_2) \Big)\leq 0.
$$
 The constructed $\tilde{u}^\lambda$ is periodic along the translations as the construction of \pref{patchingestimate} is local, making the adjustment of $\tilde{u}^\lambda$ on one edge the same as the adjustment of $\tilde{u}^\lambda$ on the opposite edge.  In this way we may consider $\tilde{u}^\lambda$ defined on all of $\R^{d+1}$.  
 
 We now proceed with a mollification of $\tilde{u}^\lambda$ to $\tilde{u}^\lambda_\epsilon$ at a scale $\epsilon$ much smaller than $\lambda$.  The mollification converges in $L^1$ and the local time gradient term is lower semicontinuous due to convexity.  Furthermore, the $\lambda$-trace error does not increase more than order $\lambda$, and thus the $\lambda$-traces of the mollified sequence converges.  By the nonlocal defect estimate of \pref{patchingestimate}, we have that the nonlocal defect vanishes across $A_1$ and $\R^{d+1}\backslash \bar{A_1}$ and thus from \eqref{FGN}  we have
 $$
        \limsup_{\lambda\rightarrow^+0}\Big(\mathcal{F}^\lambda(\tilde{u}^\lambda; \square \times \R\, \nu) -\mathcal{G}^\lambda(u^\lambda; \square \times \R\, \nu)\Big)\leq 0.
$$
 
 We proceed similarly at the initial and end times.  We need to patch on a boundary that is orthogonal to the time direction here. Having extended $u$ to the half space with $t\geq0$, we now also patch with the constant function $s_0$ on the half space with $t<0$.  We mollify the sequence at a scale $\epsilon$ and shift it forward on the $\lambda$ scale to construct a sequence $\tilde{u}^\lambda$ that agrees with the constant $s_0$ at $t=0$.  The shift also converges in $L^1$.
 
 The end time is exactly the same, except that we can simply extend by $\tilde{u}^\lambda(0,\cdot)$ to times greater than $0$.

\end{proof}

\subsection{ Proving \emph{(i)} of Theorem \ref{thm:main} by lower-semicontinuity.}\label{sec:lower-semicontinuity}\label{s.liminf}

In this section, we show part (i) of Theorem \ref{thm:main}, the lower bound inequality of the $\Gamma$-convergence: any sequence $s^\lambda$ with bounded cost which converges in $L^1$ to a limit $\bar{s}$ has asymptotic cost bounded from below by the effective cost $\bar{V}(s_0,g,\bar{s})$.  For this we follow a now standard idea introduced by Fonseca and M\"uller \cite{Fonseca1993}: it suffices to show that the $\mathcal{H}^d$ density of the limiting total variation measure is bounded from below by the respective value of the effective functional.  The key technical tool in this argument is the patching estimates  \pref{patchingestimate} and Proposition \ref{cor:periodization}, which allow us to patch the local values of $s^\lambda$ into a global periodic test minimizer for the appropriate cell problem.

\begin{proposition}\label{prop:lower-semicontinuity}
Consider a sequence $\hat{s}^\lambda$ satisfying
$$
    \hat{s}^\lambda(\lambda^{-1}\, \tau,\lambda^{-1}\, z)\rightarrow \bar{s}(\tau,z) \quad \hbox{ in } L^1([0,T]\times \mathbb{T}^d)
$$
and
$$
    \liminf_{\lambda\rightarrow 0}\mathcal{G}^\lambda(\hat{s}^\lambda;[0,T]\times \T^d) + \int_{\mathbb{T}^d} \Big[g(z)\,\hat{s}^\lambda(T,z) + \frac{1}{2\, \beta} \Phi\big(\hat{s}^\lambda(T,z)\big) - \frac{1}{2\, \beta} \Phi\big(\hat{s}^\lambda(0,z)\big) \Big] dz <+\infty.
    $$
    Then $\bar{s}\in BV((0,T)\times \mathbb{T}^d; \{\mathfrak{s},-\mathfrak{s}\})$ and

    $$
        \liminf_{\lambda\rightarrow^+ 0} \mathcal{G}^\lambda(\hat{s}^\lambda;[0,T]\times \T^d) + \int_{\mathbb{T}^d} \Big[g(z)\,\hat{s}^\lambda(T,z) + \frac{1}{2\, \beta} \Phi\big(\hat{s}^\lambda(T,z)\big) - \frac{1}{2\, \beta} \Phi\big(\hat{s}^\lambda(0,z)\big) \Big] dz\geq \bar{V}(s_0,g,\bar{s}).
    $$
\end{proposition}

\begin{proof}
    For each point $(\tau,z)\in [0,T]\times \mathbb{T}^d$, define the energy density
    $$
        h^\lambda(\tau,z) := \lambda^{-1}\, \Big[\mathcal{W}_\beta\big(\hat{s}^\lambda(\tau,z)\big) + \frac{1}{2\beta}  \Psi\big(\hat{s}^\lambda(\tau,z),\lambda\, \partial_\tau \hat{s}^\lambda(\tau,z)\big) + \frac{1}{4}\int_{ \T^d}J^\lambda(z-w)\big(\hat{s}^\lambda(\tau,z)-\hat{s}^\lambda(\tau,w)\big)^2\, dw\Big],
    $$
    and the energy measure
    $$
        \sigma^\lambda(A) := \int\int_A h^\lambda(\tau,z)dz\, d\tau+\int_{A\cap \{T\}\times \mathbb{T}^d}\Big( g(z)\, \hat{s}^\lambda(T,z)+\frac{1}{2\beta}\Phi\big(\hat{s}^\lambda(T,z)\big)\Big)dz - \int_{A\cap \{0\}\times \mathbb{T}^d}\frac{1}{2\beta}\Phi\big(\hat{s}^\lambda(0,z)\big)dz,
    $$
    for a measurable set $A$ in $[0,T]\times \mathbb{T}^d$.

    Note that 
    $$\sigma^\lambda(A) = \mathcal{G}^\lambda(\hat{s}^\lambda;A) + \mathcal{N}^\lambda(\hat{s}^\lambda;A,[0,T]\times \mathbb{T}^d\backslash A) \hbox{ for any }A \subset (0,T)\times\T^d.$$ In particular, the total mass of $\sigma^\lambda$ is bounded above by the total cost. Therefore there is a subsequence and a nonnegative measure $\sigma$ on $[0,T]\times \T^d$ such that the $\sigma^\lambda$ converge in the weak-$\star$ topology, $\sigma^\lambda \rightharpoonup^\star \sigma$.  
   
   We aim to show the following density lower bounds with respect to the interfacial surface measure as well as the initial and end-time surface measures. Call $\Sigma$ to be the set of points $(\tau,z) \in (0,T) \times \T^d$ where the measure theoretic limit of $\bar{s}$ is not in $\pm \mathfrak{s}$. Note that by our definition this is does not include any initial or final time points.
    \begin{enumerate}[label = (\alph*)]
        \item On $(0,T)\times \mathbb{T}^d$
        $$
            \frac{d\sigma}{d\mathcal{H}^d|_{\Sigma}}(\tau,z) \geq \bar{L}\big(\nu(\tau,z)\big) \hbox{ for $\mathcal{H}^d|_{\Sigma}$-a.e. } (\tau,z) \in \Sigma.
        $$
        Here $\nu(\tau,z)$ is the measure theoretic unit normal direction pointing outward to $\{\bar{s} = \mathfrak{s}\}$, defined $\mathcal{H}^{d}|_{\Sigma}$-almost everywhere.
        \item On $\tau=0$, 
        $$
            \frac{d\sigma}{d\mathcal{H}^d|_{\tau=T}}(0,z) \geq V^{init}\big(s_0(z),\bar{s}(0,z)\big) \hbox{ for $\mathcal{H}^d|_{\tau = T}$-a.e. } z \in \T^d.
        $$
        \item On $\tau=T$, 
        $$
            \frac{d{\sigma}}{d\mathcal{H}^d|_{\tau=0}}(T,z) \geq V^{end}\big(\bar{s}(T,z),g(z)\big)  \hbox{ for $\mathcal{H}^d|_{\tau = 0}$-a.e. } z \in \T^d.
        $$
        
    \end{enumerate}

    To begin the proof of (a), we consider a point $(\tau_0,z_0)\in \Sigma$ where the outward unit-normal $\nu_0$ to $\{\bar{s} = - \mathfrak{s}\}$ is well defined. Consider a unit $d$-cube in the subspace orthogonal to $\nu_0$, $\square\in \blacksquare_{\nu_0}$ from the definitions of (\ref{eqn:interfacial_energy_R}). Call the $d+1$ dimensional unit cube $Q = \Box \times [-1/2,1/2]\nu_0$ with one of the axes oriented in the $\nu_0$ direction and also the rescaled cubes $(\tau_0,z_0) + r\, Q$ that are centered at $(\tau_0,z_0)$  with side lengths $r$.      We say points $(\tau_0,z_0)$ are {\it regular} if the limit exists
    \begin{align}\label{eqn:regular_point_cube}
        \frac{d\sigma}{d\mathcal{H}^d|_{\Sigma}}(\tau_0,z_0) = \lim_{r\rightarrow 0}\frac{\sigma((\tau_0,z_0) + r\, Q)}{r^{d}},
    \end{align}
    and the rescaled $\bar{s}$, see the notation defined in \eqref{rescaled_s}, satisfies
    \begin{align}\label{eqn:regular_point_step}
         R_{(\tau_0,z_0),r}\bar{s} \rightarrow v_{\nu_0} \hbox{ strongly in } L^1_{loc},
    \end{align}
    where $v_{\nu}$ is the step function from \eqref{step}.  Standard results \cite{EG} imply that  (\ref{eqn:regular_point_cube}) and (\ref{eqn:regular_point_step}) will hold $\mathcal{H}^d|_{\Sigma}$ almost everywhere provided that $\Sigma$ is rectifiable, which holds for the jump set when $\bar{s}$ is in BV. Similarly, for (b) and (c), these conditions hold at the beginning and end times where $\Sigma$ is replaced by the slice $\{0\}\times \mathbb{T}^d$ or $\{T\} \times \mathbb{T}^d$, and $\nu$ is replaced by the appropriate normal vector (the sign of  $\bar{s}(0,z)$ or negative sign of $\bar{s}(T,z)$ in the time direction). 
    
    Now consider a regular point $(\tau_0,z_0) \in \Sigma$ as above satisfying (\ref{eqn:regular_point_cube}) and (\ref{eqn:regular_point_step}). From weak convergence, we deduce that
    $$
        \lim_{\lambda\rightarrow 0}\sigma^\lambda\big((\tau_0,z_0) + r\, Q\big) = \sigma\big((\tau_0,z_0) + r\, Q\big)
    $$
    except for a countable set $N$ of values of $r$.  Furthermore, by (\ref{eqn:regular_point_cube}) we have
    $$
        \lim_{r\rightarrow 0;\ r\not\in N}  \lim_{\lambda\rightarrow^+0}\frac{\sigma^\lambda\big((\tau_0,z_0) + r\, Q\big)}{r^d} =  \lim_{r\rightarrow 0;\ r\not\in N}  \frac{\sigma\big((\tau_0,z_0) + r\, Q\big)}{r^d} = \frac{d\sigma}{d\mathcal{H}^d|_{\Sigma}}(\tau_0,z_0).
    $$
    Since $\hat{s}^\lambda \rightarrow \bar{s}$ in $L^1$, by (\ref{eqn:regular_point_step}) we have
    $$
        \lim_{r\rightarrow 0;\ r\not\in N}\lim_{\lambda\rightarrow^+0} R_{(\tau_0,z_0),r}\hat{s}^\lambda =  \lim_{r\rightarrow 0;\ r\not\in N} R_{(\tau_0,z_0),r}\bar{s} = v_{\nu_0} \hbox{ in } L^1.
    $$
    Then we can choose sequences $r_i$ and $\lambda_i$ such that
    \begin{align*}
        \lim_{i\rightarrow \infty } r_i = \lim_{i\rightarrow \infty } \frac{\lambda_i}{r_i} = 0,\\
         \lim_{i\rightarrow\infty}\frac{\sigma^{\lambda_i}\big((\tau_0,z_0) + r_i\, Q\big)}{r_i^d}  = \frac{d\sigma}{d\mathcal{H}^d|_{\Sigma}}(\tau_0,z_0),\\
          \lim_{i\rightarrow \infty } R_{(\tau_0,z_0),r_i}\hat{s}^{\lambda_i}  = v_{\nu_0} \hbox{ in } L^1.
    \end{align*}
    
 By the scaling property of $\mathcal{G}^\lambda$ (Lemma  \ref{lem:scaling_identity}), and by dropping the remainder of the nonlocal term away from $(\tau_0,z_0)+r_i\, Q$, we have
 \begin{align*}
        &\ \frac{\sigma^{\lambda_i}\big((\tau_0,z_0) + r_i\, Q\big)}{r_i^d} \\
        \geq&\  \frac{\mathcal{G}^{\lambda_i}\big(\hat{s}^{\lambda_i}; (\tau_0,z_0) + r_i\, Q\big)}{r_i^d}\\
        =&\ \mathcal{G}^{\lambda_i/r_i}\big(R_{(\tau_0,z_0),r_i}\hat{s}^{\lambda_i}; Q\big).
    \end{align*}

    By \lref{tracesonaesurface} we can choose $t \in (0,1)$ arbitrarily close to $1$ so that, up to a subsequence, the $\lambda_i$-traces of $R_{(\tau_0,z_0),r_i}\hat{s}^{\lambda_i}$ on $\partial (t\, Q)$ converge to $v_{\nu_0}$ in the sense of \dref{lambdatrace}.  The cost decreases since the new cube is smaller:
     \begin{align*}
        \mathcal{G}^{\lambda_i/r_i}\big(R_{(\tau_0,z_0),r_i}s^{\lambda_i}; Q\big)\geq \mathcal{G}^{\lambda_i/r_i}\big(R_{(\tau_0,z_0),r_i}s^{\lambda_i}; t\, Q\big).
    \end{align*}

      Using \pref{patchingestimate} we construct $s_{patched}^i$ which ``extends" $R_{(\tau_0,z_0),r_i}s^{\lambda_i}$ to $t\, \square \times \R\, \nu_0$ by patching with $v_{\nu_0}$ at distance $t/2$ away from the tangent hyperplane. Corollary \ref{cor:patching} and the convergence of the $\lambda_i$-traces of $R_{(\tau_0,z_0),r_i}s^{\lambda_i}$ to $v_{\nu_0}$ on $\partial t Q \cap \{(\tau-\tau_0,z-z_0)\cdot\nu_0 = \pm t/2\}$ shows that 
      \begin{align*}
          \liminf_{i\rightarrow \infty}\Big\{\mathcal{G}^{\lambda_i/r_i}\big(R_{(\tau_0,z_0),r_i}s^{\lambda_i}; t\, Q\big) + \mathcal{G}^{\lambda_i}(
          v_{\nu_0}; (t\, \square \times \R\, \nu_0) \backslash t\, Q\big) - \mathcal{G}^{\lambda_i/r_i}\big(s_{patched}^i;t\, \square \times \R\, \nu_0\big)\Big\} \geq 0.
      \end{align*}
      We use \pref{periodization}(a) to further replace $s_{patched}^i$ by $s_{periodic}^i$, a $t\, \square$ periodic function of $\R^{d+1}$.  This does not increase the cost due to the agreement of the $\lambda^i$ trace limits along the boundary of $t\, \square \times \R\, \nu_0$,
      \begin{align*}
         \liminf_{i\rightarrow \infty}\Big\{ \mathcal{G}^{\lambda_i/r_i}\big(s_{patched}^i;t\, \square \times \R\, \nu_0\big)- \mathcal{G}^{\lambda_i/r_i}\big(s_{periodic}^i;t\, \square \times \R\, \nu_0\big)\Big\} \geq 0.
      \end{align*}
      Since $v_{\nu_0}$ is constant on the components of $(t\, \square \times \R\, \nu_0) \backslash t\, Q$, we have
      $$
        \mathcal{G}^{\lambda_i}(
          v_{\nu_0}; (t\, \square \times \R\, \nu_0) \backslash t\, Q\big)=\mathcal{F}^{\lambda_i}(
          v_{\nu_0}; (t\, \square \times \R\, \nu_0) \backslash t\, Q\big)= 0.
      $$
    Lemma \ref{lem:localization} bounds the nonlocal defect for the periodic approximation $s^i_{periodic}$ so that
          $$
             \liminf_{i\rightarrow \infty }\Big\{\mathcal{G}^{\lambda_i}(
          s_{periodic}^i; t\, \square \times \R\, \nu_0\big)-\mathcal{F}^{\lambda_i}(
         s_{periodic}^i;t\, \square \times \R\, \nu_0\big)\Big\}\geq 0.
          $$
    
    Again using the scaling Lemma \ref{lem:scaling_identity}, and Proposition \ref{p.periodization} that allows us to assume that $s_{periodic}^i$ is continuously differentiable,
     we have
    $$
        t^d\, \bar{L}_R\big(\nu_0\big)\leq \mathcal{F}^{\lambda_i/r_i}\big(s_{periodic}^i, t\, \square \times \R\, \nu_0\big)
    $$
      which concludes the proof for (a) after chaining together the inequalities and taking $t$ close to $1$.

    \medskip
    
    At the initial time the argument is identical, except that when defining $s_{periodic}^i$ we must enforce that $s_{periodic}^i\in \mathcal{X}_R^{init}(s_0(z),\bar{s}(0,z),\square)$, i.e., that $s_{periodic}^i(0,x) = s_0(z)$.  This is also done by \pref{periodization}(b) by patching with the constant function $s_0(z)$ in the domain $t<0$ and shifting slightly forward in time so that $s_{periodic}^i(0,x) = s_0(z)$ holds. The rest of the argument goes through exactly working on $t\, \square \times \R^+\, \nu_0$ where $\nu_0$ points forward in time.

    \medskip

At the final time we have (for $Q^-$ the intersection of the cube with the lower half plane and $\nu_1$ the unit-vector in the negative time direction)
\begin{align*}
    &\ \frac{{\sigma}^{\lambda_i}_{end}\big((\tau_0,z_0) + r_i\, Q^-\big)}{r_i^d} \\
        &\geq\  \frac{\mathcal{G}^{\lambda_i}\big(s^{\lambda_i}; (\tau_0,z_0) + r_i\, Q^-\big)}{r_i^d} + r^{-d}\int_{z+r\, \square}\Big(g(x)\, \hat{s}^{\lambda_i}(T,x) + \frac{1}{2\beta} \Phi\big(\hat{s}^{\lambda_i}(T,x)\big)\Big)dx\\
        =&\ \mathcal{G}^{\lambda_i/r_i}\big(R_{(\tau_0,z_0),r_i}s^{\lambda_i}; Q^-\big) + \int_\square\Big( g(z+r_i\, y)\,\hat{s}^{\lambda_i}(T,z+ r\, y)+\frac{1}{2\beta} \Phi\big(\hat{s}^{\lambda_i}(T,z+r\, y)\big)\Big)dy.
\end{align*}
At points of Lebesgue density of $g$, we can approximately replace $g(z+r_i\, y)$ in the line above with $g(z)$.  As before we construct $s_{periodic}^i \in \mathcal{X}_R^{end}(\bar{s}(T,z),\square)$ in Proposition \ref{p.periodization} (c), making sure to preserve
\begin{align*}
    \liminf_{i\rightarrow\infty}\big\{&\int_{t\square}\Big(g(z)\, R_{(T,z),r_i}s(0,x) + \frac{1}{2\beta}\Phi\big( R_{(T,z),r_i}s(0,x)\big)\Big)dx\\
    &\ -\int_{t\square}\Big(g(z)\, s_{periodic}^i(0,x)+\frac{1}{2\beta}\Phi\big(s_{periodic}^i(0,x)\big)\Big)dx\big\}\leq 0.
\end{align*}

We arrive at, using again Lemma \ref{lem:localization} to equate $\mathcal{F}^{\lambda_i/r_i}\big(s_{periodic}^i;t\, \square \times \R^+\, n_1\big)$ and $\mathcal{G}^{\lambda_i/r_i}\big(s_{periodic}^i;t\, \square \times \R^+\, n_1\big)$,
    \begin{align*}
        \liminf_{i\rightarrow \infty}\Big\{&\mathcal{G}^{\lambda_i/r_i}\big(R_{(T,z),r_i}s^{\lambda_i}; t\, Q^-\big) +  \int_{  t\, \square}\Big(g\big(z+r_i\, y)\, s^{\lambda_i}(T,z+r_i\, y)+\frac{1}{2\beta}\Phi\big(s^{\lambda_i}(T,z+r_i\, y)\big)\Big)dy\\
        &\ -\mathcal{F}^{\lambda_i/r_i}\big(s_{periodic}^i;t\, \square \times \R^+\, n_1\big) -\int_{t\square}\Big(g(z)\, s_{periodic}^i(0,x)+\frac{1}{2\beta}\Phi\big(s_{periodic}^i(0,x)\big)\Big)dx\Big\}\geq 0.
    \end{align*}
    We repeat the final scaling argument with
    $$
       t^d\, V^{end}\big(\bar{s}(T,z), g(z)\big)\leq \mathcal{F}^{\lambda_i/r_i}\big(s_{periodic}^i;t\, \square \times \R^+\, n_1\big) +
       \int_{t\square}\Big(g(z)\, s_{periodic}^i(0,x)+\frac{1}{2\beta}\Phi\big(s_{periodic}^i(0,x)\big)\Big)dx,
    $$
    to conclude the claim of (c).
    
\end{proof}

\begin{proof}[Proof of Theorem \ref{thm:main} part (i)]
Let $s^\lambda$ as in the statement and choose a subsequence so that 
\[\lim_{i \to 0} C^{\lambda_i}(s^{\lambda_i},a^{\lambda_i}) = \liminf_{\lambda \to 0}C^{\lambda}(s^{\lambda},a^{\lambda}).\]  By Proposition \ref{prop:compactness} $s^{\lambda_i}$ has a subsequence (not relabeled) converging in $L^1((0,T) \times \T^d)$.  Now the hypotheses of Proposition \ref{prop:lower-semicontinuity} are satisfied and the conclusion of \ref{prop:lower-semicontinuity} is the desired conclusion of Theorem \ref{thm:main}.

\end{proof}

\subsection{Proof of \emph{(ii)} of Theorem \ref{thm:main} by an upper bound inequality}\label{s.limsup}

The difficulty in constructing the recovery sequence lies in approximating a general smooth interface locally by flat interfaces. We  follow the beautiful idea introduced by Alberti and Bellettini \cite{alberti1998non}.  Essentially the concept is to reduce to the case of polyhedral sets, via a typical argument with a Reshetnyak Theorem \cite{spector2011simple}, and then prove the case of polyhedral sets by an inductive argument.  We can mostly follow \cite{alberti1998non} until we come to the point of ``patching" neighboring cells at which point we reuse the ideas from \sref{patching-lemma}.  
\begin{definition}
We say that two sets $E$ and $F$ in $\R^{d+1}$ are {\it transversal} if $\mathcal{H}^{d}(E \cap F) = 0$.
\end{definition}
\begin{definition}
\begin{itemize}
\item An $d+1$-dimensional {\it polyhedral } set $E$ in $\R^{d+1}$ is an open or closed set whose boundary is a Lipschitz surface contained in the union of finitely many affine hyperplanes.  The faces of $\partial E$ are intersections of $\partial E$ with one of those hyperplanes, edges points of $E$ are boundary points which are in multiple faces.  The normal direction $\nu_E$ is defined at all non-edge points.  

\item A $k$-dimensional polyhedral set is a polyhedral set in a $k$-dimensional affine subspace or the closure of such a set. Note that intersections of polyhedral sets in $\R^{d+1}$ with $k$-dimensional affine subspaces are $k$-dimensional polyhedral sets.

\item A polyhedral set in a domain $\Omega\subset \R^{d+1}$ is the intersection of a polyhedral set in $\R^{d+1}$ with $\Omega$.  

\item A function ${s} \in BV(\Omega; \{ \pm \mathfrak{s}\})$ is called a polyhedral function if there is an $d+1$ dimensional polyhedral set $E$ which has $\partial E$ transversal to $\partial \Omega$, such that $s = \mathfrak{s}{\bf 1}_{E} - \mathfrak{s} {\bf 1}_{E^C}$ almost everywhere in $\Omega$.  More generally $f \in BV(\Omega,\R)$ is called polyhedral if there is a finite collection of disjoint polyhedral sets $E_j$ so that $\cup \overline{E}_j \cap \Omega = \Omega$ and $f$ is constant on each $E_j$.
\end{itemize}
\end{definition}
We also make the following notation: given a set $E$ in $\R^N$ and $\delta>0$ we call $E_\delta$ to be the set of points with Euclidean distance at most $\delta$ to $E$.

\medskip

We may localize the limit energy $\bar{V}$ from (\ref{eqn:bar_V}) on an open subset $A\subset [0,T]\times \mathbb{T}^d$ as
\begin{equation}\label{eqn:bar_V_star_loc}
    \bar{V}(s_0,g,\bar{s};A) :=  \int_{ A_0}[V^{init}\big({s}_0(z),\bar{s}(0,z)\big) dz + \int_{A_T} V^{end}\big(\bar{s}(T,z) ,g(z)\big) dz+\int_{A\backslash (A_0\cup A_T)} \bar{L}\big(\nu(\tau,z)\big) d\mathcal{H}^d.\nonumber
   \end{equation}

Now we construct the recovery sequence for polyhedral functions $s_0$ and $g$.
\begin{theorem}\label{t.polyhedral-limsup}
Let $\bar{s} \in BV((0,T)\times \mathbb{T}^d; \{\pm \mathfrak{s}\})$, $s_0 \in BV(\T^d;(-\mathfrak{s},\mathfrak{s}))$, and $g\in BV(\T^d; \R)$ with $|g|\leq \frac{1}{2\beta} \Phi'(\mathfrak{s})$ be polyhedral functions.    There are functions $s^\lambda$ on $[0,T]\times \mathbb{T}^d$ with $\lim_{\lambda\rightarrow^+0}\|s^\lambda(0,\cdot) - s_0\|_{L^1(\T^d)}=0$ and $|s^\lambda| \leq \mathfrak{s}$ so that $s^\lambda \to \bar{s}$ uniformly on every compact subset of $(0,T)\times \mathbb{T}^d \setminus \textup{Jump}(\bar{s})$ and
\[ \limsup_{\lambda \to^+ 0} \Big\{\mathcal{G}^\lambda(s^\lambda; (0,T)\times\mathbb{T}^d\big)d\tau + \int_{\mathbb{T}^d}\Big[ {s}^\lambda(T,z){g}(z)+\frac{1}{2\beta}\Phi\big(s^\lambda(T,z)\big)\Big]dz - \int_{\mathbb{T}^d}\frac{1}{2\beta}\Phi\big(s_0(z)\big)dz \Big\}\leq \bar{V}(s_0,g,\bar{s}).\]
\end{theorem}

\begin{proof}[Proof of \tref{polyhedral-limsup}]
The proof is a direct adaptation of \cite{alberti1998non} until we reach the proof of (c) below,
which considers patching recovery sequences in neighboring domains.

\medskip

Call $\Gamma = \textup{Jump}(\bar{s}) \cup \{0\}\times \mathbb{T}^d \cup  \{T\}\times \mathbb{T}^d$.  Note that by definition $\Gamma$ is a $d$-dimensional closed polyhedral set. 

\medskip

For a given $\delta>0$, consider the class $\mathcal{A}$ of $d+1$-dimensional open polyhedral sets $A$ in $[0,T]\times \mathbb{T}^d$ with the following  properties:
\begin{enumerate}[label=(\roman*)]
\item  $\partial A$ and $\Gamma $ are transversal.
\item There is a sequence of functions $s^\lambda$ defined and continuous on $\overline{A}$ and a constant $K \geq 1$ (which may depend on $A$) so that
\begin{equation}\label{e.equivSprop}
 s^\lambda  = \bar{s} \ \hbox{ on } \{\xi\in \overline{A}: d(\xi, \Gamma )> K\lambda\}, \quad \lim_{\lambda\rightarrow^+0}\int_{(0,z)\in\overline{A}}|s^\lambda(0,z) - s_0(z)| dz =0,  \ \quad |\partial_ts|\leq K\lambda^{-1},
 \end{equation}
and
\[ \mathcal{G}^\lambda(s^\lambda;A) + \int_{(T,z)\in A}\Big[ s^\lambda(T,z)g(z) + \frac{1}{2\beta}\Phi\big(s^\lambda(T,z)\big)\Big]dz - \int_{
(0,z)\cap A} \frac{1}{2\beta}\Phi\big(s_0(z)\big)dz \leq \bar{V}(s_0,g,\bar{s};A)+\delta .\]
\end{enumerate}

Denote $A_1 \sqcup A_2$ to be the interior of $\overline{A_1} \cup \overline{A}_2$.  We prove that $[0,T]\times \mathbb{T}^d \in \mathcal{A}$ by the following inductive steps.

\begin{enumerate}[label = (\alph*)]
\item \label{item.empty} If $A$ is a $d+1$-dimensional polyhedral set in $\Omega$ such that $\mathcal{H}^d(\overline{A} \cap \Gamma) = 0$ then $A \in \mathcal{A}$.
\item \label{item.proj} Let $\Sigma$ be one of the following: a connected polyhedral subset of $\{T\}\times \mathbb{T}^d \setminus \textup{Jump}(g)$, a connected polyhedral subset of $\{0 \}\times \mathbb{T}^d \setminus [\textup{Jump}(\bar{s}) \cup \textup{Jump}(s_0)]$, or a face of $\textup{Jump}(\bar{s})$. Let $\pi$ be the projection map onto the affine subspace containing $\Sigma$.  Suppose that $A$ is an $d+1$-dimensional polyhedral set in $[0,T]\times \mathbb{T}^d$ so that $\Gamma \cap A = \Sigma$ and $\pi(A) = \Sigma$.  Then $A \in \mathcal{A}$.
\item \label{item.patch} If $A_1,A_2 \in \mathcal{A}$ are disjoint then $A_1\sqcup A_2 \in \mathcal{A}$.
\end{enumerate}

\begin{figure}
    \centering
    \begin{tikzpicture}[scale = 1]
\path[fill = gray!10] (-2,-1) -- (1,0) -- (2,-2) -- (-1,-2);
\path[fill = gray!55] (-2,-1) -- (1,0) -- (-2,0);
\path[fill = gray!30] (2,-2) -- (1,0) -- (2,0);
\draw (-1.6,-.4) node{$A_1$};
\draw (1.6,-.4) node{$A_3$};
\draw[dashed] (-2,-1)-- (1,0) -- (2,-2);
\draw[thick] (-2,0)--(2,0) node[anchor = west]{$t=T$};
\draw[thick] (-1,-2) node[anchor = north]{$\textup{Jump}(\bar{s})$} -- (1,0);
\node (B) at (1,-2.3) {$A_2$};
\end{tikzpicture}
    \caption{Example of a polyhedral decomposition so that each subregion is either of the type considered in \ref{item.empty} or in \ref{item.proj}.}
    \label{f.polyhedral-decomp}
\end{figure}
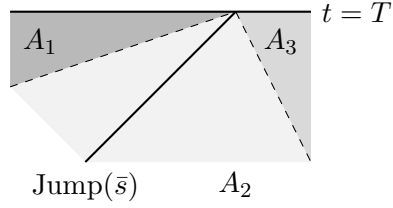
\medskip

 Since $\bar{s}$ is polyhedral we can write $[0,T]\times \mathbb{T}^d$ as a finite $\sqcup$ union of polyhedral subdomains satisfying the hypotheses of \ref{item.empty} or \ref{item.proj}, see \fref{polyhedral-decomp}. (For example do a Voronoi type decomposition, and then add regions of type \ref{item.empty} as necessary to achieve the projection hypothesis in \ref{item.proj}.)  In particular, even though constants $K$ in (ii) may increase by a finite factor at each union stage, there is no problem since there are only finitely many such unions.

\medskip

Note that once we have proven \ref{item.empty}-\ref{item.patch}, since $\delta>0$ was arbitrary, by a diagonal argument, we can find the recovery sequence $s^\lambda$.

\medskip

{\bf Proof of (a):} In this case, $\bar{s}$ is constant equal to either $\pm \mathfrak{s}$ in each connected component of $A$.  In this case, the recovery sequence is trivial $s^\lambda = \bar{s}$.  The nonlocal and Dirichlet parts of the energy are zero for constants, and the double-well potential is zero on $\pm \mathfrak{s}$ so
\[ \mathcal{G}^\lambda(s^\lambda;A) = 0,\]
$\mathcal{H}^d(\{0\} \times \T^d \cap \overline{A}) = 0$ so the initial data condition in \eref{equivSprop} is trivially satisfied, and
\[ \int_{(T,z)\in \overline{A}}\Big[ s^\lambda(T,z)g(z)+\frac{1}{2\beta}\Phi\big(s^\lambda(T,z)\big)\Big]dz = 0\]
because $\mathcal{H}^{d}(\{T\}\times \mathbb{T}^d \cap \overline{A}) = 0$.

\medskip

{\bf Proof of (b):}  We divide into cases depending whether the flat interface $\Sigma$ is in $\{0\}\times \mathbb{T}^d$, $\{T\}\times \mathbb{T}^d$ or is a face of $\textup{Jump}(\bar{s})$.

First suppose $\Sigma$ is a face of $\textup{Jump}(\bar{s})$.  Let $\nu$ be the (constant) inner space-time normal to the affine plane containing $\Sigma$.  From the definition of $\bar{L}(\nu)$, let $\square \in \blacksquare_\nu$ and $w$ be an element $\mathcal{X}_R(\square)$ (defined in \eqref{space_X}) with
\[ |\square|^{-1}\mathcal{F}^1(w; \square \times \R\, \nu) \leq \bar{L}(\nu) + \delta. \]
 We fix some $(\bar{\tau},\bar{z})\in \Sigma$ and let $s^\lambda(\tau, z) := w\big(\lambda^{-1}(\tau-\bar{\tau}),\lambda^{-1}(z-\bar{z})\big)$.  Note that since $w \in \mathcal{X}_R(\square)$ the property \eref{equivSprop} is satisfied with $K = R$.  Recall that $\mathcal{X}_R(\square)$ consists of $C^1$ functions so $|\partial_ts^\lambda| \leq K\lambda^{-1}$ increasing $K$ if necessary. The remainder of the argument is the same as \cite{alberti1998non}, the $\lambda \square$ period cells tile most of $\Sigma$ except for a $O(\lambda)$-neighborhood of $\partial \Sigma$ which has surface measure $O(\lambda)$ because $\Sigma$ is polyhedral.

Next suppose $\Sigma$ is a component of $\{T\}\times \mathbb{T}^d \setminus (\textup{Jump}(\bar{s}) \cup \textup{Jump}(g))$. We fix some $\bar{z}$ with $(T,\bar{z})\in \Sigma$, so then $\bar{s}$ takes a constant value either $\pm \mathfrak{s}$ on $A$, which we call $\bar{s}(A)$.  Also $g$ takes a constant value on $\Sigma$, $g(\bar{z})$.  Let $\nu$ now denote the normal-vector oriented in the negative time direction.  From the definition of $V^{end}$, we choose $\square\in \blacksquare_\nu$ and $w$ be an element $\mathcal{X}^{end}_R(\bar{s}(A), \square)$ (defined in \eqref{space_end}) with
\[ |\square|^{-1} \Big(\mathcal{F}^1(w;\square \times \R^+\, \nu) + \int_\square \Big[g(\bar{z})\, w(0,x)+ \frac{1}{2\beta}\Phi\big(w(0,x)\big)\Big]dx\Big) \leq V^{end}(\bar{s}(A),g(\bar{z})) + \delta. \]
Let $s^\lambda(\tau, z) := w\big( \lambda^{-1}(\tau-T),\lambda^{-1}(z- \bar{z})\big)$. As before we can conclude the compact support and time derivative bound properties of \eref{equivSprop} from the properties of the space $\mathcal{X}_R^{end}$. Using the projection condition $\pi(A) = \Sigma$ and tiling $\Sigma$ with $\lambda \square$ period cells, up to an $O(\lambda)$-error from the period cells intersecting $\partial \Sigma$ as before, we have
\begin{align*}
    \lim_{\lambda\rightarrow^+ 0}\left\{\mathcal{G}^\lambda(s^\lambda;A) + \int_{(T,z)\in A}\Big[ s^\lambda(T,z)g(z)+\frac{1}{2\beta}\Phi\big(s^\lambda(T,z)\big)\Big]dz\right\}  &\leq \int_{\Sigma}V^{end}\big(\bar{s}(T,z),g(\bar{z})\big) dz + \delta |\Sigma|.
\end{align*}

Finally suppose $\Sigma$ is a component of $\{0\}\times \mathbb{T}^d \setminus \textup{Jump}(\bar{s})$ so again $\bar{s}$ takes a constant value either $\pm \mathfrak{s}$ on $A$, call that value $\bar{s}(A)$.  Similarly, $s_0$ is constant on $\Sigma$ so we let $s_0(\Sigma)$ denote the value. From the definition of $V^{init}$, let $R>0$, $\nu$ be oriented in the positive time direction, $\square\in \blacksquare_\nu$ and $w$ be an element $\mathcal{X}^{init}_R(s_0(\Sigma),\bar{s}(A),\square)$ (defined in \eqref{space_init}) with
\[ |\square|^{-1}\mathcal{F}^1(w;\square \times \R^+\, \nu) - \frac{1}{2\beta}\Phi(s_0(\Sigma)) \leq V^{init}(s_0(\Sigma),\bar{s}(\Sigma)) + \delta. \]
We similarly fix some $(0,\bar{z})\in \Sigma$ and let $s^\lambda(\tau, z) := w\big(\lambda^{-1}\tau,\lambda^{-1}(z-\bar{z})\big)$, then proceed as in the previous cases to conclude.

\medskip

{\bf Proof of (c)}
This is the point where we need new arguments. Essentially the patching procedure of \pref{patchingestimate} is carried out again here, but with simpler boundary conditions we are able to make more explicit estimates.

Given disjoint sets $A_1,A_2 \in \mathcal{A}$ set $A := A_1 \sqcup A_2$ and ${\Delta} := \partial A_1 \cap \partial A_2$. Note that ${\Delta}$ is contained in a finite union of affine hyperplanes.  By assumption, there are sequences $s^\lambda_j$ defined, respectively, on $\overline{A}_j$ satisfying hypothesis (ii).

Define
\[ \tilde{s}^\lambda := \begin{cases} 
s^\lambda_1 & \hbox{in } \ A_1 \\
s^\lambda_2 & \hbox{in } \overline{A}_2 .
\end{cases}\]
We need to regularize $\tilde{s}^{\lambda}$ across the interface $\partial A_1 \cap \partial A_2$ at least in the time variable.  For given $r \geq \lambda$, let $\zeta: [0,T]\times\mathbb{T
}^d \to [0,1]$ a continuous cutoff function, which is $1$ in an $r$-neighborhood of $\partial A_1 \cap \partial A_2$ and zero outside of a $2r$-neighborhood with $|\grad_{\tau,z} \zeta| \leq r^{-1}$.  Let $\phi^\lambda = \lambda^{-(d+1)} \phi(\lambda^{-1}\cdot)$ be a standard mollifier at scale $\lambda$, and define
\[ s^\lambda := \zeta \phi^\lambda * \tilde{s}^\lambda + (1-\zeta)\tilde{s}^\lambda.\]

Note that, because $\tilde{s}^\lambda$ is only defined in $A_1 \sqcup A_2$ we mean technically
\[ \phi^\lambda * \tilde{s}^\lambda(\tau,z) = Z(\tau,z)^{-1}\int_{A_1 \sqcup A_2} \phi^\lambda(\tau-u, z-y) \tilde{s}^\lambda(u,y) \ dy\, du \]
with the normalization factor
\[Z(\tau, z) := \frac{1}{\int_{A_1 \sqcup A_2} \phi^\lambda(\tau-u,z-y)dydu}.\]
Since $A_1 \sqcup A_2$ is a polyhedral domain, $\inf_{A_1 \sqcup A_2} Z(x,t) \geq c > 0$ where the constant depends on the domain Lipschitz property.  Due to the hypothesis (ii) and its definition, the function $s^\lambda$ is continuous in $\overline{A}_1 \cup \overline{A}_2$ with 
\begin{equation}\label{e.convolutiondtbound}
|\partial_t s^\lambda| \leq C\lambda^{-1} + Cr^{-1}.
\end{equation}
As long as $r \geq \lambda$ this is bounded by $C\lambda^{-1}$.

By hypothesis (ii) we know
\[  \tilde{s}^\lambda  \equiv \bar{s} \ \hbox{ in } \ A_1 \sqcup A_2 \setminus \Gamma_{(K_1+K_2)\lambda}\]
and so we can conclude
\[ s^\lambda  \equiv s^{\lambda}_j \ \hbox{ in } \ A_j \setminus [\Gamma_{(K_1+K_2+1)\lambda} \cap {\Delta}_{2r}]. \]
The mollification converges uniformly away from the jump set, which also implies convergence in $L^1$ at the initial time. 
Call $K = K_1+K_2$.  Then we have shown that $s^\lambda$ satisfies \eref{equivSprop} with the constant $K$.  

Because ${\Delta}$ and $\Gamma$ are finite unions of $d$-dimensional polyhedral sets which meet transversally,
\begin{equation}\label{e.overlapsize}
 |\Gamma_{K\lambda} \cap {\Delta}_{2r}| \leq C \lambda r
 \end{equation}
for some constant $C$ depending on the sets but not on $\lambda$ or $r$.

Now we use this to compute the energy
\begin{align*}
    &\ \mathcal{G}^\lambda(s^\lambda; A_1 \sqcup A_2)+ \int_{\{T\}\times \mathbb{T}^d \cap (A_1 \sqcup A_2)}\Big[ s^\lambda(T,z)g(z)+\frac{1}{2\beta}\Phi\big(s^\lambda(T,z)\big)\Big]dz\\
    \leq&\  \mathcal{G}^\lambda(s_1^\lambda; A_1)+ \int_{\{T\}\times \mathbb{T}^d \cap A_1}\Big[ s_1^\lambda(T,z)g(z)+\frac{1}{2\beta}\Phi\big(s_1^\lambda(T,z)\big)\Big]dz\\
    &\ +\mathcal{G}^\lambda(s_2^\lambda;A_2)+ \int_{\{T\}\times \mathbb{T}^d \cap A_2}\Big[ s^\lambda_2(T,z)g(z)+\frac{1}{2\beta}\Phi\big(s_2^\lambda(T,z)\big)\Big]dz\\
    &\ +\mathcal{G}^\lambda(s^\lambda;{\Delta}_{2r} \cap \Gamma_{K\lambda}) + \mathcal{N}^\lambda(s^\lambda;A_1,A_2).
\end{align*} 
We estimate the energy in the overlap region using \eref{overlapsize}.  The double-well term is immediate using $W_\beta([-\mathfrak{s},\mathfrak{s}]) \leq W_\beta(0)$
\[\int_{{\Delta}_{2r} \cap \Gamma_{K\lambda}} \frac{1}{\lambda}\mathcal{W}_\beta(s^\lambda) \ d\tau\, dz \leq C r.\]
The derivative term is estimated using \eref{convolutiondtbound} and \eref{overlapsize}
\[ \int_{{\Delta}_{2r} \cap \Gamma_{K\lambda}} \frac{1}{\lambda}\Psi_0(\lambda\, \partial_\tau s^\lambda) d\tau\, dz \leq Cr.\]
The nonlocal part of the energy is bounded similarly by \eref{overlapsize} 
\[ N^\lambda(s^\lambda; {\Delta}_{2r} \cap \Gamma_{K\lambda}, {\Delta}_{2r} \cap \Gamma_{K\lambda}) \leq C r \]
using the simple inequality
\[ \mathcal{N}^\lambda(s,A,A) \leq C\lambda^{-1}|A|. \]
Finally for the nonlocal cross term we use Lemma \ref{lem:localization} to find
\[ \limsup_{\lambda \to 0} \mathcal{N}^{\lambda}(s^\lambda;A_1,A_2) =0.\]

Combining the above we find
\begin{align*}
     &\ \limsup_{\lambda\rightarrow^+0}\Big\{\mathcal{G}^\lambda(s^\lambda; A_1 \sqcup A_2)+ \int_{\{T\}\times \mathbb{T}^d \cap (A_1 \sqcup A_2)}\Big[ s^\lambda(T,z)g(z)+\frac{1}{2\beta}\Phi\big(s^\lambda(T,z)\big)\Big]dz\Big\}\\
 \leq&\ \bar{V}(s_0,g,\bar{s};A_1)+\bar{V}(s_0,g,\bar{s};A_2) + Cr.
    \end{align*}
The transversality condition used again implies
\begin{align*}\limsup_{\lambda \to^+ 0} \Big\{\mathcal{G}^\lambda(s^\lambda; A_1 \sqcup A_2)+ \int_{\{T\}\times \mathbb{T}^d \cap (A_1 \sqcup A_2)}\Big[ s^\lambda(T,z)g(z)+\frac{1}{2\beta}\Phi\big(s^\lambda(T,z)\big)\Big]dz\Big\}\leq\bar{V}(s_0,g,\bar{s};A_1 \sqcup A_2)+ Cr.
\end{align*}
We can also choose $r \to 0$ as $\lambda \to 0$ to get (ii), actually $r = \lambda$ works.  

While the construction may not satisfy $|s^\lambda|\leq \mathfrak{s}$, we may apply Lemma \ref{lem:bdd} to find an asymptotically equivalent sequence that satisfies this property.
\end{proof}

From \tref{polyhedral-limsup} we can conclude the proof of Theorem \ref{thm:main} part (ii) by the density of polyhedral sets/functions and the Reshetnyak continuity theorem.  We just need to establish the upper-semicontinuity of the surface energy density $\bar{L}(\nu)$.

\begin{proposition}\label{prop:upper-semicontinuity}
The maps $\bar{L}:S^d \mapsto \R$ and $V^{init}(\cdot,\bar{s}), V^{end}(\bar{s},\cdot): (-1,1) \to \R$ with $\bar{s} \in \{-\mathfrak{s}, \mathfrak{s}\}$  are upper-semicontinuous.
\end{proposition}
\begin{proof}
    Given a base direction $\nu_0 \in S^d$ there is a mapping $I : \nu\in S^d \to SO(d)$ such that $I(\nu) \nu_0=\nu$.  Note that for any $d$-dimensional periodic cube $\Box$ in the orthogonal complement of $\nu_0$, the map  $O \in SO(d)\mapsto |\Box|^{-1}\mathcal{F}^1(s\circ O;\Box \times \R O\nu_0)$ is continuous for any fixed  $s \in \mathcal{X}_R(\Box)$.
    Thus the formula 
    $$
        \bar{L}(\nu) = \liminf_{R\to\infty}\big[\inf\{|\Box|^{-1}\mathcal{F}^1(s\circ I(\nu);\square \times \R\, \nu); \square\in \blacksquare_{\nu},\  s\in \mathcal{X}_R(\square) \}\big]
    $$
    represents $\bar{L}$ as an infimum of continuous functions of $\nu \in S^d$, making it upper-semicontinuous.
    
    The argument for $V^{end}$ is immediate using continuity of the integral for the terminal cost evaluation.
    
    To prove upper-semicontinuity of $V^{end}(\bar{s}, \cdot)$, we can simply consider a linear extension and compare costs. For instance, fix $s_{0,1}$, $s_{0,2}$, $R$ and $\square$.  For $s_1\in \mathcal{X}_R^{init}(s_{0,1},\bar{s},\square)$ we can define $s_2\in \mathcal{X}_R^{init}(s_{0,2},\bar{s},\square)$ by
    $$
        s_2(t,x)=\begin{cases}
            (|s_{0,2}-s_{0,1}|-t)\frac{s_{0,2}}{|s_{0,2}-s_{0,1}|}+t\frac{s_{0,1}}{|s_{0,2}-s_{0,1}|} & t <|s_{0,2}-s_{0,1}|\\
            s_1(t-|s_{0,2}-s_{0,1}|,x) & t\geq|s_{0,2}-s_{0,1}|.
        \end{cases}
    $$
    The cost is continuous with respect to $s_{0,2}$, making $V^{init}$ upper-semicontinuous when we take the infimum over $R$ and $\square$.
\end{proof}

\begin{proof}[Proof of Theorem \ref{thm:main} Part (ii)]
    Let $\bar{s} \in BV((0,T)\times \mathbb{T}^d; \{\pm \mathfrak{s}\})$, $s_0 \in L^1(\T^d;(-1,1))$, and $g\in L^1(\T^d; \R)$ be general, not necessarily polyhedral, data.  As a consequence of Theorem 1.24 in \cite{giusti1984minimal} (and approximation of smooth functions by polyhedral ones), there are sequences of polyhedral functions $\bar{s}^n$, $s^n_0$, and $g^n$ in the same spaces and so that
    \[(\bar{s}^n,s^n_0,g^n) \to (\bar{s},s_0,g) \ \hbox{ in $L^1$ norm,} \]
    and also
    \[D\bar{s}^n \overset{\ast}{\rightharpoonup} D\bar{s} \ \hbox{ and } \ |D\bar{s}^n| \overset{\ast}{\rightharpoonup} |D\bar{s}| \]
    in duality with continuous functions.  Furthermore, this convergence implies that
    $$
        (\bar{s}^n(0,\cdot), \bar{s}^n(T,\cdot)) \to (\bar{s}^n(0,\cdot), \bar{s}^n(T,\cdot)) \ \hbox{ in $L^1$ norm,}
    $$
    as shown in Theorem 2.2 of \cite{farah2020proving}. By Theorem \ref{t.polyhedral-limsup}, for each $n$ there are functions $s^{\lambda,n}$ with $\lim_{\lambda\rightarrow^+0} \|s^{\lambda,n}(0,\cdot)-s_0^n\|_{L^1(\T^d)}=0$, $\lim_{\lambda\rightarrow^+0} \|s^{\lambda,n}-\bar{s}\|_{(L^1((0,T)\times\T^d)} = 0$, and
    $$
\limsup_{\lambda \to^+ 0} \Big\{\mathcal{G}^\lambda(s^{\lambda,n}; (0,T)\times\mathbb{T}^d\big)d\tau + \int_{\mathbb{T}^d}\Big[ {s}^{\lambda,n}(T,z){g}(z)+\frac{1}{2\beta}\Phi\big(s^{\lambda,n}(T,z)\big)\Big]dz - \int_{\mathbb{T}^d}\frac{1}{2\beta}\Phi\big(s_0^n(z)\big)dz \Big\}\leq \bar{V}(s_0^n,g^n,\bar{s}^n).
    $$
    
    We may now consider a sequence $s^{\lambda_n,n}$ such that $\lim_{n\rightarrow \infty} \lambda_n = 0$, and we will adjust the initial condition of $s^{\lambda_n,n}$ so that it agrees with $s_0$.  This may be done by the linear interpolation
    $$
        \tilde{s}^n(\tau,z) = \max\{1 - \frac{\tau}{\lambda}, 0\} s_0(z) +  \min\{\frac{\tau}{\lambda}, 1\} s^{\lambda_n,n}(\tau,z).
    $$
    Arguing as in the proof of \pref{patchingestimate}, seeing that $\lim_{n\rightarrow \infty}\|s^{\lambda_n,n}(0,\cdot)-s_0\|_{L^1(\T^d)} =0$, we have
    \begin{align*}
&\ \limsup_{n\rightarrow \infty} \Big\{\mathcal{G}^\lambda(\tilde{s}^{n}; (0,T)\times\mathbb{T}^d\big)d\tau + \int_{\mathbb{T}^d}\Big[\tilde {s}^{n}(T,z){g}(z)+\frac{1}{2\beta}\Phi\big(\tilde{s}^n(T,z)\big)\Big]dz - \int_{\mathbb{T}^d}\frac{1}{2\beta}\Phi\big(s_0(z)\big)dz \Big\}\\
\leq&\  
\limsup_{n\rightarrow \infty} \Big\{\mathcal{G}^\lambda(s^{\lambda_n,n}; (0,T)\times\mathbb{T}^d\big)d\tau + \int_{\mathbb{T}^d}\Big[ {s}^{\lambda_n,n}(T,z){g}(z)+\frac{1}{2\beta}\Phi\big(s^{\lambda_n,n}(T,z)\big)\Big]dz - \int_{\mathbb{T}^d}\frac{1}{2\beta}\Phi\big(s_0^n(z)\big)dz \Big\}.
    \end{align*}
    
    We may now conclude upper-semicontinuity of the limit
    \begin{align*}
    \limsup_{n\rightarrow \infty} \bar{V}(s_0^n,g^n,\bar{s}^n) \leq \bar{V}(s_0,g,\bar{s}).
    \end{align*}
    Using Proposition \ref{prop:upper-semicontinuity}, $L^1$ convergence of $s_0^n$ and $g^n$ at the initial and final times with Fatou's lemma and Egorov's theorem we have the upper-semicontinuous limit for $V^{init}$ and $V^{end}$.  Again using Proposition \ref{prop:upper-semicontinuity}, and the well-known result of Reshetnyak that weak convergence of $\bar{s}^n$ in BV combined with convergence of the perimeter implies upper-semicontinuity of the surface area functional (see Theorem 1.3 of \cite{spector2011simple}). 
\end{proof}

\bibliography{IsingBib}
\bibliographystyle{plain}

\end{document}